\documentclass[12pt]{article}
\usepackage[paperwidth=235mm, paperheight=315mm]{geometry}
\usepackage{graphicx}
\usepackage{amssymb}
\usepackage{amsmath}
\usepackage{hyperref}
\usepackage{enumerate}
\usepackage{amsthm}
\usepackage{amsmath}
\usepackage{float}
\usepackage{color}   
\usepackage{lineno}

\makeatletter
\renewenvironment{proof}[1][\proofname]{%
   \par\pushQED{\qed}\normalfont%
   \topsep6\p@\@plus6\p@\relax
   \trivlist\item[\hskip\labelsep\bfseries#1\@addpunct{.}]%
   \ignorespaces
}{%
   \popQED\endtrivlist\@endpefalse
}
\makeatother

\newtheorem{theorem}{Theorem}
\newtheorem{proposition}[theorem]{Proposition}
 \numberwithin{theorem}{section}
 \newtheorem{corollary}[theorem]{Corollary}
\newtheorem{lemma}[theorem]{Lemma}

\newtheorem{remark}[theorem]{Remark}

\numberwithin{equation}{section}

\renewcommand{\P}{\mathbb{P}}
\newcommand{\E}{\mathbb{E}}
\newcommand{\R}{\mathbb{R}}
\newcommand{\mR}{\mathcal{R}}

\newcommand{\cL}{\mathcal{L}}
\newcommand{\pmR}{\partial\mathcal{R}}

\newcommand{\Q}{\mathbb{Q}}
\newcommand{\N}{\mathbb{N}}

\newcommand{\cC}{\mathcal C}
\newcommand{\cF}{\mathcal F}
\newcommand{\cE}{\mathcal E}
\newcommand{\cK}{\mathcal K}

\newcommand{\cW}{\mathcal W}
\newcommand{\cY}{\mathcal Y}
\newcommand{\cO}{\mathcal{O}}
\newcommand{\eps}{\varepsilon}
\newcommand{\veps}{\varepsilon}
 \newcommand{\nn}{\nonumber}
 \newcommand{\no}{\noindent}
 \newcommand{\lL}{\mathsf{L}}
\newcommand{\rR}{\mathsf{R}}
\newcommand{\pf}[1]{{\color{black}#1}}
\renewcommand{\pf}[1]{}

\newcommand{\ARXIV}[1]{{\color{blue}#1}}
\renewcommand{\ARXIV}[1]{}

\begin{document}
\author{
Jieliang Hong\footnote{Department of Mathematics, University of British Columbia, Canada, E-mail: {\tt jlhong@math.ubc.ca} }
}
\title{On the boundary local time measure of super-Brownian motion}
\date{\today}
\maketitle
\begin{abstract}
     If $L^x$ is the total occupation local time of $d$-dimensional super-Brownian motion, $X$, for $d=2$ and $d=3$, we construct a random measure $\cL$, called the boundary local time measure, as a rescaling of $L^x e^{-\lambda L^x} dx$ as $\lambda\to \infty$, thus confirming a conjecture of \cite{MP17} and further show that the support of $\cL$ equals the topological boundary of the range of $X$, $\pmR$. This latter result uses a second construction of a boundary local time $\widetilde{\cL}$ given in terms of exit measures and we prove that $\widetilde{\cL}=c\cL$ a.s. for some constant $c>0$. We derive reasonably explicit first and second moment measures for $\cL$ in terms of negative dimensional Bessel processes and use it with the energy method to give a more direct proof of the lower bound of the Hausdorff dimension of $\pmR$ in \cite{HMP18}.  The construction requires a refinement of the $L^2$ upper bounds in \cite{MP17} and \cite{HMP18} to exact $L^2$ asymptotics. The methods also refine the left tail bounds for $L^x$ in \cite{MP17} to exact asymptotics. We conjecture that  the Minkowski content of $\pmR$ is equal to the total mass of the boundary local time $\cL$ up to some constant.
\end{abstract}

\section{Introduction and main results}

\subsection{Introduction}
Let $M_F=M_F(\R^d)$ be the space of finite measures on $(\R^d,\mathfrak{B}(\R^d))$ equipped with the topology of weak convergence of measures.  A super-Brownian motion (SBM) $(X_t, t\geq 0)$ starting at $X_0 \in M_F$ is a continuous $M_F$-valued strong Markov process defined on some filtered probability space $(\Omega, \cF, \cF_t, P)$ described below and we let $\P_{X_0}$ denotes any probability under which $X$ is as above. We write $\mu(\phi)=\int \phi(x) \mu(dx)$ for any measure $\mu$ and take our branching rate to be one so that for any non-negative bounded Borel functions $\phi,f$ on $\R^d$, 
\begin{equation}\label{ev6.1}
\E_{X_0}\Bigl(\exp\Bigl(-X_t(\phi)-\int_0^t X_s(f)ds\Bigr)\Bigr) =\exp\Big(-X_0(V_t(\phi,f))\Big).
\end{equation}

\no Here $V_t(x)=V_t(\phi,f)(x)$ is the unique solution of the mild form of 
\begin{equation}\label{ev6.2}
\frac{\partial V}{\partial t}=\frac{\Delta V_t}{2}-\frac{V_t^2}{2}+f,\quad V_0=\phi,
\end{equation}
that is,
\begin{equation*}
V_t=P_t(\phi)+\int_0^tP_s\Bigl(f-\frac{V_{t-s}^2}{2}\Bigr)\,ds.
\end{equation*}
In the above $(P_t)$ is the semigroup of standard $d$-dimensional Brownian motion. See Chapter II 
of \cite{Per02} for the above and further properties.  

It is known that the extinction time of $X$ is a.s. finite (see, e.g., Chp II.5 in \cite{Per02}). The total occupation time measure of $X$ is the (a.s. finite) measure defined as
 \begin{eqnarray*}
I(A)=\int_0^\infty X_s(A)ds.
\end{eqnarray*}
Let $S(\mu)=\text{Supp}(\mu)$ denote the closed support of a measure $\mu$. We 
define the range, $\mR$, of $X$ to be 
$\mR=\textnormal{Supp}(I).$
In dimensions $d\leq 3$, the occupation measure $I$ has a density, $L^x$, which is called (total) local time of $X$, that is,
$$I(f)=\int_0^\infty X_s(f)\,ds=\int_{\R^d} f(x)L^x\,dx\text{ for all  non-negative measurable }f.$$
Moreover, $x\mapsto L^x$ is lower semicontinuous, is continuous on $S(X_0)^c$, and  for $d=1$ is globally continuous (see Theorems~2 and 3 of \cite{Sug89}).
Thus one can see that in dimensions $d\leq 3$,
\begin{equation*}
\mR=\overline{\{x:L^x>0\}},
\end{equation*}
and $\mR$ is a closed set of positive Lebesgue measure. In dimensions $d\ge 4$, $\mR$ is a Lebesgue null
set of Hausdorff dimension $4$ for SBM starting from $\delta_0$ (see Theorem~1.4 of \cite{DIP89}), which explains our restriction to $d\le 3$ in this work.


We will largely be considering the case when $X_0=\delta_0$. The Hausdorff dimensions of the boundaries of SBM have been studied in \cite{MP17} and \cite{HMP18}. Let $\pmR$ be the topological boundary of the range $\mR$ and define $F$ to be the boundary of the set where the local time is positive, i.e. $F:=\partial\{x:L^x>0\}$.
Let dim denote the Hausdorff dimension and introduce:
\begin{equation}\label{ep1.1}
p=p(d)=\begin{cases}
3 &\text{ if }d=1\\
2\sqrt 2 &\text{ if }d=2\\
\frac{1+\sqrt {17}}{2} &\text{ if }d=3,
\end{cases}
\text{ and } \alpha=\frac{p-2}{4-d}=\begin{cases}
1/3 &\text{ if }d=1\\
\sqrt 2-1 &\text{ if }d=2\\
\frac{\sqrt {17}-3}{2} &\text{ if }d=3.
\end{cases}
\end{equation}

\begin{theorem}[\cite{MP17},\cite{HMP18}]\label{tm1} 
With $\P_{\delta_0}$-probability one, 
$$\textnormal{dim}(F)=\textnormal{dim}(\pmR)= d_f:=d+2-p=\begin{cases}0&\text{ if } d=1\\
4-2\sqrt2\approx1.17&\text{ if }d=2\\
\frac{9-\sqrt{17}}{2}\approx 2.44&\text{ if }d=3.
\end{cases}$$
\end{theorem}

It is also natural to consider SBM under the canonical measure $\N_{x_0}$. Recall from Section II.7 in \cite{Per02} that $\N_{x_0}$ is a $\sigma$-finite measure on $C([0,\infty),M_F)$, which is the space of continuous $M_F(\R^d)$-valued paths furnished with the compact-open topology, such that if we let
$\Xi=\sum_{i\in I}\delta_{\nu^i}$ be a Poisson point process on $C([0,\infty),M_F)$ with intensity $\N_{X_0}(d\nu)=\int \N_{x}(d\nu) X_0(dx)$, then
\begin{equation}\label{ev2.1}
X_t=\sum_{i\in I}\nu_t^i=\int \nu_t\ \Xi(d\nu),\ t>0, 
\end{equation}
has the law, $\P_{X_0}$, of a super-Brownian motion $X$ starting from $X_0$. In this way, $\N_{x_0}$ describes the contribution of a cluster from a single ancestor at $x_0$ and the super-Brownian motion is then obtained by a Poisson superposition of such clusters. 
We refer the readers to Theorem II.7.3(c) in \cite{Per02} for more details. The existence of the local time $L^x$ under $\N_{x_0}$ will follow from this decomposition and the existence under $\P_{\delta_{x_0}}$. Therefore the local time $L^x$ under $\P_{X_0}$ may be decomposed as 
\begin{align}\label{e1.21}
L^x=\sum_{i\in I}L^x(\nu^i)=\int L^x(\nu)\Xi(d\nu).
\end{align}
The global continuity of local times $L^x$ under $\N_{x_0}$ is given in Theorem 1.2 of \cite{Hong18}. 
It is not surprising that Theorem \ref{tm1} continues to hold under the canonical measure.
\begin{theorem}[\cite{MP17},\cite{HMP18}]
$\N_{0}$-a.e. $\textnormal{dim}(F)=\textnormal{dim}(\pmR)= d_f$.
\end{theorem}

The definition of $F$ is natural from an analytical perspective but the topological boundary $\pmR$ is a more natural random set from a geometrical point of view. One can check that 
\begin{equation}
\pmR\subset F.
\end{equation}
In $d=1$, it has been shown in Theorem 1.7 in \cite{MP17} and Theorem 1.4 in \cite{Hong19} that there exist random variables $\lL$ and $\rR$ such that 
\begin{equation}
F=\pmR=\{\lL, \rR\} \text{ where } \lL<0<\rR,\quad \N_0-a.e. \text{ or } \P_{\delta_0}-a.s.
\end{equation}
 Whether or not $F=\pmR$ remains open in $d=2$ and $d=3$. Given the simple nature of $F=\pmR$ in $d=1$, we largely will focus on $d=2$ and $d=3$ in what follows.
 
Our main goal in this paper is to construct a random measure on $\pmR$ or $F$. Recall $\alpha$ as in \eqref{ep1.1}.
 For any $\lambda>0$, under $\P_{\delta_0}$ and $\N_{0}$ we define a random measure $\cL^\lambda$ on $\R^d$ by
\begin{align}\label{e7.3}
d{\cL}^\lambda(x)=\lambda^{1+\alpha} L^x e^{-\lambda L^x} dx.
\end{align}

The two authors in \cite{MP17} conjecture that as $\lambda \to \infty$, ${\cL}^\lambda$ converges in probability in the space $M_F (\R^d)$ to a finite measure ${\cL}$ which necessarily is supported on $F$. In this paper, we confirm this conjecture and further show that the support of $\cL$ is precisely $\pmR$.\\

\no ${\bf Convention\ on\ Functions\ and\ Constants.}$ Constants whose value is unimportant and may change from line to line are denoted $C, c, c_d, c_1,c_2,\dots$, while constants whose values will be referred to later and appear initially in say, Lemma~i.j are denoted $c_{i.j},$ or $ \underline c_{i.j}$ or $C_{i.j}$. \\

 \no ${\bf Notation.}$ Let $M_F$ be equipped with any complete metric $d_0$ inducing the weak topology and let $\{\mu_t, t\in T\}$ be a collection of $M_F$-valued random vectors. We use $\mu_t \overset{P}{\rightarrow} \mu_{t_0}$ as $t\to t_0$ to denote the convergence in probability under $\P_{X_0}$ if for any $\eps>0$, we have $\P_{X_0}(d_0(\mu_t,\mu_{t_0})>\eps) \to 0$ as $t\to t_0$. We slightly abuse the notation and use $\mu_t \overset{P}{\rightarrow}  \mu_{t_0}$ as $t\to t_0$ to denote the convergence in measure under $\N_{X_0}$ if for any $\eps>0$, we have $\N_{X_0}(\{d_0(\mu_t,\mu_{t_0})>\eps\} \cap A) \to 0$ as $t\to t_0$ where $A$ is any measurable set such that $\N_{X_0}(A)<\infty$.

\subsection{Main Results}

\begin{theorem}\label{t0}
Let $d=2$ or $3$. Under both $\N_{0}$ and $\P_{\delta_0}$, there exists a random measure ${\cL} \in M_F(\R^d)$, supported on $\pmR$, such that ${\cL}^\lambda \overset{P}{\rightarrow} {\cL}$ as $\lambda \to \infty$ and there is a sequence $\lambda_n \to \infty$ such that ${\cL}^{\lambda_n} \to {\cL}$ a.s. as $n \to \infty$.
\end{theorem}

 Next we consider the case under $\N_{X_0}$ or $\P_{X_0}$ for general initial condition $X_0$. Since the above theorem holds under $\N_x$ for any $x$ by translation invariance of SBM, and $\N_{X_0}(\cdot)=\int \N_x(\cdot) X_0(dx)$, it is easy to see that the above result continues to hold under $\N_{X_0}$ for any $X_0$. However, the case under $\P_{X_0}$ is somehow more delicate--instantaneous extinction at time $t=0$ will make the behavior of $\pmR\cap S(X_0)$ quite different than that under $\P_{\delta_0}$ or $\N_0$; see Proposition 1.6 and Remark 1.8(b) of \cite{MP17} for such examples. Therefore under $\P_{X_0}$ we will restrict our interest in $S(X_0)^c$. For any $\lambda>0$, under $\P_{X_0}$ we define a random measure $\cL^\lambda$ supported on $S(X_0)^c$ by
\begin{align}\label{em7.3}
d{\cL}^\lambda(x)=
\lambda^{1+\alpha} L^x e^{-\lambda L^x} 1(x\in S(X_0)^c)dx.
\end{align}

${\bf Notation.}$ For any $\delta>0$ and any set $K$, we let $K^{\geq \delta}=\{x: d(x,K)\geq \delta\}$ where $d(x, K)=\inf\{|x-y|: y\in K\}$. Similarly we define $K^{>\delta}, K^{\leq \delta}$ and $K^{<\delta}$. For any measure $\mu$ and any set $K$, we use $\mu|_{K}(\cdot)\equiv \mu(\cdot \cap K)$ to denote the restriction of $\mu$ to $K$.
\begin{theorem}\label{t0.0.1}
Let $d=2$ or $3$ and let $X_0\in M_F(\R^d)$. Under $\P_{X_0}$ there exists a $\sigma$-finite random measure ${\cL}$, supported on $\pmR \cap S(X_0)^c$, such that for any $k\geq 1$, ${\cL}^\lambda|_{S(X_0)^{\geq 1/k}} \overset{P}{\rightarrow} {\cL}|_{S(X_0)^{\geq 1/k}}$ as $\lambda \to \infty$ and there is a sequence $\lambda_n \to \infty$ such that ${\cL}^{\lambda_n}|_{S(X_0)^{\geq 1/k}} \to {\cL}|_{S(X_0)^{\geq 1/k}}$, $\forall k\geq 1$ a.s.
 \end{theorem}

\begin{remark}
(a) The behavior of $\pmR$ on the boundary of $S(X_0)$ depend largely on the mass distribution of $X_0$ and is still quite different from that under $\N_0$ or $\P_{\delta_0}$. In the proof we first give the existence of a finite measure $l_k$ by restricting our interest to $S(X_0)^{\geq 1/k}$ for any $k\geq 1$ and then construct a $\sigma$-finite measure $\cL$ supported on $S(X_0)^c$ by defining $\cL|_{S(X_0)^{\geq 1/k}}=l_k$ for any $k\geq 1$. In most cases we will only be considering the properties of $\cL$ on sets with positive distance away from $S(X_0)$ and the above theorem suffices for our purposes. \\
(b) One sufficient condition on $X_0$ to give the a.s finiteness of $\cL(1)$ goes back to the renormalization of local times in $d=2$ or $3$ (see \cite{Hong18}). For example in $d=2$, if we have $\inf_{x\in S(X_0)} \int \log^{+}(1/|y-x|) X_0(dy)=\infty$, then Theorem 1.11 of \cite{Hong18} will imply that $\P_{X_0}$-a.s. there is some $\delta>0$ so that $S(X_0)^{<\delta} \subset \text{Int}(\mR)$, and therefore $S(X_0)^{<\delta}$ is not in the support of $\cL$. Hence $\cL=\cL|_{S(X_0)^{\geq \delta}}$ and the a.s. finiteness of $\cL(1)$ follows.
\end{remark}
\begin{theorem}\label{tl1.1}
 $\P_{X_0}$-a.s. and  $\N_{X_0}$-a.e. for any open set  $U\subset S(X_0)^c$,
\begin{align}\label{er1.3}
U\cap \pmR\neq \emptyset \Rightarrow \cL(U)>0.
\end{align}
In particular we have  $\P_{X_0}$-a.s. that 
\begin{align}\label{ae1.5}
S(X_0)^c\cap \pmR\neq \emptyset \Rightarrow \cL>0.
\end{align}
and $\text{Supp}(\cL)=S(X_0)^c\cap \pmR$.
\end{theorem}
\noindent The hypothesis in \eqref{ae1.5} is necessary--an example is given in Proposition~1.5 of \cite{MP17} where
it fails with positive probability.

Let $B(x_0, \eps)=B_\eps(x_0)=\{x:|x-x_0|<\varepsilon\}$ and set $B_\eps=B(\eps)=B_\eps(0)$. 
\begin{corollary}\label{c1.6}
$\P_{\delta_0}$-a.s. and $\N_0$-a.e. for any open set  $U$,
\begin{align}\label{aea12.1}
U\cap \pmR\neq \emptyset \Rightarrow \cL(U)>0.
\end{align}
In particular, $\text{Supp}(\cL)=\pmR$ and $\cL>0$, $\P_{\delta_0}$-a.s. and $\N_0$-a.e.
\end{corollary}
\begin{proof}
We know from the proof of Corollary 1.4 and Theorem 1.5 of \cite{HMP18} that $\P_{\delta_0}$-a.s. or  $\N_{0}$-a.e. there exists some $\delta>0$ such that $L^x>0$ for all $|x|<\delta$ and so $0\notin \pmR$, which implies \[U\cap \pmR\neq \emptyset \Rightarrow (U\backslash \{0\}) \cap \pmR\neq \emptyset.\] Then we may apply Theorem \ref{tl1.1} with $U\backslash \{0\}$ in place of $U$ to complete the proof of \eqref{aea12.1}. 
Next for any $x\in \pmR$, take $U=B(x,\eps)$ for any $\eps>0$ and use the above to get $\pmR \subset \text{Supp}(\cL)$. Together with Theorem \ref{t0} we conclude $\text{Supp}(\cL)=\pmR$, $\P_{\delta_0}$-a.s. and $\N_0$-a.e. By \eqref{aea12.1}, it follows immediately that $\cL>0$, $\P_{\delta_0}$-a.s. and $\N_0$-a.e.
%
\end{proof}
Now we proceed to the first and second moment measures of $\cL$. Define
\begin{align}\label{ev1.5}
\mu=
\begin{cases}
-1/2\ &\text{if}\ d=1\\
0\ &\text{if}\ d=2\\
1/2\ &\text{if}\ d=3,
\end{cases}
\text{ and }
 \nu=\sqrt{\mu^2+4(4-d)}.
\end{align}
so that (recall \eqref{ep1.1})
$d=2+2\mu \text{ and } p=\mu+\nu.$ 
Let $\Hat{P}_{x}^{(2-2\nu)}$ denote the law of the $d$-dimensional process $\{Y_t: t\geq 0\}$ such that
  \begin{align}\label{be1.0}
       \begin{cases}
        &Y_t=x+\Hat{B}_t+\int_0^t (-\nu-\mu)\frac{Y_s}{|Y_s|^2}ds, \quad t< \tau_0,\\
        &Y_t=0, \quad t\geq \tau_0.
        \end{cases}
  \end{align} 
 Here $\tau_0=\inf\{t\geq 0: |Y_t|=0\}$ and $\Hat{B}$ is a standard $d$-dimensional Brownian motion starting from $x$ under $\Hat{P}_{x}^{(2-2\nu)}$. Remark \ref{r1.7}(b) below shows why $\Hat{P}_{x}^{(2-2\nu)}$ is well-defined. 
  Let $V^\infty(x):=\N_0(L^x>0)$ for all $x\neq 0$ and
for any $x_1\neq x_2$, we define for all $x\neq x_1, x_2$, 
\begin{align}\label{ev}
V^{\vec{\infty},\vec{x}}(x):=\N_x(\{L^{x_1}>0\} \cup \{L^{x_2}>0\}).
\end{align}
For $i=1,2$ we define
\begin{align}\label{eu1u2}
U_i^{\vec{\infty},\vec{x}}(x):=\frac{1}{|x-x_i|^{p}} \Hat{E}_{x-x_i}^{(2-2\nu)}\Big(  e^{-\int_0^{\tau_{0}} (V^{\vec{\infty},\vec{x}}(Y_s+x_i)-V^{\infty}(Y_s))ds}\Big),
\end{align}
and set
\begin{align}\label{eu12}
U_{1,2}^{\vec{\infty},\vec{x}}(x):=-E_x\Big(\int_0^{\infty} \prod_{i=1}^2 U_i^{\vec{\infty},\vec{x}}(B_t) \exp\Big(-\int_0^{t} V^{\vec{\infty},\vec{x}}(B_s) ds\Big) dt\Big),
\end{align}
where $B$ is a $d$-dimensional Brownian motion starting from $x$ under $P_x$. 
 \begin{theorem}\label{p0.0}
(a) There is some constant $K_{\ref{p0.0}}>0$ such that for any nonnegative measurable $\phi: \R^d \to \R$, we have
\begin{align}\label{e6.1}
    \N_0\Big(\int \phi(x) d\cL(x)\Big)=K_{\ref{p0.0}}\int |x|^{-p}  \phi(x) dx.
\end{align}
(b) For any nonnegative measurable $h: \R^d \times \R^d \to \R$, we have
\begin{align}\label{e6.2}
    \N_0\Big( (\cL\times \cL)(h)\Big)=K_{\ref{p0.0}}^2 \int  h(x_1, x_2) (-U_{1,2}^{\vec{\infty},\vec{x}}(0)) dx_1 dx_2.
\end{align}
Moreover, there is some constant $c_{\ref{p0.0}}>0$ such that
\begin{align}\label{e6.3}
     \N_0\Big(&\int h(x_1, x_2) d\cL(x_1) d\cL(x_2)\Big)\nonumber\\
    &\leq K_{\ref{p0.0}}^2 \int c_{\ref{p0.0}} (|x_1|^{-p}+|x_2|^{-p})|x_1-x_2|^{2-p} h(x_1, x_2)  dx_1 dx_2.
\end{align}
\end{theorem}

\begin{remark} \label{r1.7}
(a) The superscript $2-2\nu<0$ on $\Hat{P}_{x}^{(2-2\nu)}$ is used to indicate the fact that  $\{|Y_s|,  s\geq 0\}$ under $\Hat{P}_{x}^{(2-2\nu)}$ is a stopped Bessel process of dimension $2-2\nu$ starting from $|x|>0$ (see, e.g., \eqref{ev2.5}), thus giving the connection between the moment measures of $\cL$ and Bessel process of negative dimension. We refer the reader to \cite{GY03} for more information on Bessel process of negative dimensions. See also \cite{Leg15} where a connection is made in $d=1$ between the left-most point in the range of SBM and the Bessel process of dimension $2-2\nu=-5$ where $\nu=7/2$ as in \eqref{ev1.5} for $d=1$.\\
(b) Under $\Hat{P}_{x}^{(2-2\nu)}$, we have $\tau_0$ is the hitting time  of a $(2-2\nu)$-dimensional Bessel process and so with $\Hat{P}_{x}^{(2-2\nu)}$-probability one, $\tau_0<\infty$ (see, e.g., Exercise (1.33) in Chp. XI of \cite{RY94}). For any $\eps>0$, we have the drift in \eqref{be1.0} is bounded for all $0\leq t\leq \tau_\eps$ and hence the uniqueness of solutions to \eqref{be1.0} holds for all $0\leq t\leq \tau_\eps$ (see also \eqref{e2.5}). It then follows by continuity that the uniqueness of solutions to \eqref{be1.0} will hold for all $0\leq t\leq \tau_0$.
 \end{remark}
\begin{theorem}\label{p0.0_1}
(a) For any nonnegative measurable $\phi: \R^d \to \R$, we have
\begin{align}\label{ae8.4}
    \E_{X_0}(\cL(\phi))=K_{\ref{p0.0}}\int_{  S(X_0)^c}  \phi(x)  e^{-X_0(V^\infty(x-\cdot))} X_0(|x-\cdot|^{-p})dx.
\end{align}
(b) For any nonnegative measurable $h: \R^d\times \R^{d} \to \R$, we have
\begin{align}\label{ae8.5}
   \E_{X_0}\Big(&(\cL\times \cL)(h)\Big)=K_{\ref{p0.0}}^2\int_{(S(X_0)^c)^2} h(x_1,x_2)\nn\\
   & e^{-X_0(V^{\vec{\infty},\vec{x}})} \Big(X_0(U_1^{\vec{\infty},\vec{x}}) X_0(U_2^{\vec{\infty},\vec{x}})-X_0(U_{1,2}^{\vec{\infty},\vec{x}})\Big)dx_1dx_2.
\end{align}
Moreover,
\begin{align}\label{ae8.6}
    \E_{X_0}\Big( (\cL\times& \cL)(h)\Big)\leq K_{\ref{p0.0}}^2 \int_{(S(X_0)^c)^2} h(x_1, x_2) \Bigg(X_0(|x_1-\cdot|^{-p})X_0(|x_2-\cdot|^{-p})\nonumber\\
    &  +c_{\ref{p0.0}}\Big(X_0(|x_1-\cdot|^{-p})+X_0(|x_2-\cdot|^{-p})\Big)|x_1-x_2|^{2-p}\Bigg)  dx_1 dx_2.
\end{align}
\end{theorem}

Now that we have $\text{Supp}(\cL)=\pmR$ a.s. under $\N_0$ and $\P_{\delta_0}$,  one immediate application with the above moment measures would be to use the energy method (see, e.g., Theorem 4.27 of \cite{Mor10}) to find the lower bound of the Hausdorff dimension of $\pmR$.
\begin{theorem}\label{p0.0.1}
For any $\eta>0$, we have for all $k\geq 1$,
\begin{align*}
&(i)\  \N_{0}\Big(\int 1_{\{k^{-1}\leq |x_1|, |x_2|\leq k\}} |x_1-x_2|^{-(d+2-p-\eta)}\cL(dx_1) \cL(dx_2) \Big)<\infty,\\
&(ii)\  \P_{\delta_0}\Big(\int 1_{\{k^{-1}\leq |x_1|, |x_2|\leq k\}} |x_1-x_2|^{-(d+2-p-\eta)}\cL(dx_1) \cL(dx_2) \Big)<\infty.
\end{align*}
In particular, $dim(\pmR)\geq d+2-p$, $\N_0$-a.e. and $\P_{\delta_0}$-a.s.
\end{theorem}
\begin{proof}
For any $k\geq 1$ and $\eta>0$ small, we apply Theorems \ref{p0.0}(b) and \ref{p0.0_1}(b) with \[h(x_1, x_2)=|x_1-x_2|^{-(d+2-p-\eta)} 1(k^{-1}\leq |x_1| \leq k)1(k^{-1}\leq |x_2| \leq k)\] to get (i) and (ii). 
Take a countable union of null sets to get $\N_0$-a.e. and $\P_{\delta_0}$-a.s. that
\begin{align}
\int 1_{\{k^{-1}\leq |x_1|, |x_2|\leq k\}} |x_1-x_2|^{-(d+2-p-\eta)}\cL(dx_1) \cL(dx_2)<\infty, \forall k\geq 1.
\end{align}
By the compactness of the range of SBM (see, e.g., Corollary III.1.7 of \cite{Per02} and Theorem IV.7(iii) of \cite{Leg99}) and that $L^x$ is strictly positive for $x$ near $0$ (see the proof of Corollary \ref{c1.6}), we can conclude $\N_0$-a.e. and $\P_{\delta_0}$-a.s. that
   $ \text{Supp}(\cL)=\pmR \subset \{x: k^{-1}\leq |x|\leq k\}$ for $k$ large enough.
Therefore it follows from Theorem 4.27 of \cite{Mor10} that $\N_0$-a.e. and $\P_{\delta_0}$-a.s. $\text{dim}(\pmR)\geq d+2-p-\eta$. Let $\eta\downarrow 0$ to get the desired result.
\end{proof}

Now we say a few words on the ideas underlying Theorem \ref{t0}. For any point $x$ near $F$ and $\pmR$, its local time $L^x$ will either be zero or small and positive, and hence the asymptotics of $\P_{\delta_0}(0<L^x<\eps)$ as $\eps \downarrow 0$ will be useful in studying $F$ and $\pmR$. The Laplace transform of $L^x$ derived in Lemma 2.2 in~\cite{MP17} is given by
\begin{align}\label{ev1.0}
\E_{X_0}(e^{-\lambda L^x})&=\exp\Big(-\int \N_y(1-e^{-\lambda L^x})X_0(dy)\Big)=e^{-X_0(V^\lambda(x-\cdot))},
\end{align}
where $V^\lambda$ is the unique solution (see Section~2 of \cite{MP17} and the references given there) to 
\begin{equation}\label{ev1.1}
\frac{\Delta V^\lambda}{2}=\frac{(V^\lambda)^2}{2}-\lambda\delta_0,\ \ V^\lambda>0\text{ on }\R^d.
\end{equation}
Recall $V^\infty(x)=\N_0(L^x>0)$. Let $\lambda\uparrow\infty$ in~\eqref{ev1.0} and \eqref{ev1.1} to see that $V^\lambda(x)\uparrow V^\infty(x)$ and
\begin{align}\label{ev1.2}
\P_{X_0}(L^x=0)&=\exp\Big(-\int \N_y(L^x>0) X_0(dy)\Big)=e^{-X_0(V^\infty(x-\cdot))}.
\end{align}
It is explicitly known that (see, e.g., (2.17) in \cite{MP17})
\begin{equation}\label{ev1.3}
V^\infty(x)=\frac{2(4-d)}{|x|^2}:=\lambda_d |x|^{-2},
\end{equation}
and in particular $V^\infty$ solves
\begin{equation}\label{ev1.4}
\frac{\Delta V^\infty}{2}=\frac{(V^\infty)^2}{2}\text{ for }x\neq 0.
\end{equation}

Write $f(t)\sim g(t)$ as $t\downarrow 0$ iff $f(t)/g(t)$ is bounded below and above by constants $c, c'>0$ for small positive $t$, and similarly for $f(t)\sim g(t)$ as $t\to \infty$. By an application of Tauberian theorem, it is shown in Theorem 1.3 of \cite{MP17} that for any $x\neq 0$,
\begin{align}\label{ef2.2}
\P_{\delta_0}(0<L^x<\frac{1}{\lambda})\sim V^\infty(x)-V^\lambda(x)\sim |x|^{-p} \lambda^{-\alpha} \text{ as } \lambda\to \infty.
\end{align}
The above bounds justify our explicit construction of $\cL^\lambda$ in some way--one can check that as $\lambda$ is getting larger and larger, $\cL^\lambda$ will concentrate more and more on the set of points $x$ whose local time $L^x$ is approximately $1/\lambda$ and the probability is of order $\lambda^{-\alpha}$.  In the end as $\lambda \to \infty$ the limiting measure will be supported on $F$ or $\pmR$. 

 In fact we can refine the above bounds in \eqref{ef2.2} to exact asymptotics.
\begin{proposition}\label{p4.1}
There is some constant $c_{\ref{p4.1}}>0$ so that for all $x\neq 0$, we have
\begin{align*}
(i) &\lim_{\lambda \to \infty} \lambda^{\alpha} \N_0(0<L^x<1/\lambda) = c_{\ref{p4.1}}|x|^{-p}.\\
(ii) &\lim_{\lambda \to \infty} \lambda^{\alpha} \P_{\delta_0}(0<L^x<1/\lambda) = c_{\ref{p4.1}}|x|^{-p}e^{-V^\infty(x)}.
\end{align*}
\end{proposition}

\no The above exact asymptotic results may allow us to get an insight of the Minkowski content of $\pmR$. 

\no $\bf{Conjecture\ 1:}$ There is some constant $c_1=c_{\ref{p4.1}} K_{\ref{p0.0}}^{-1} >0$ such that
\begin{align}
\lambda^{\alpha} 1_{\{0<L^x<1/\lambda\}}dx \overset{P}{\rightarrow} c_1\cL \text{ as } \lambda \to \infty  \text{ under } \N_0 \text{ or } \P_{\delta_0}.
\end{align}
Recall $\alpha=(p-2)/(4-d)$. By an application of the improved $4-d-\eta$ H\"{o}lder continuity of $L^x$ for $x$ near $\pmR$ for any $\eta>0$ (see \cite{Hong19}), we further conjecture that

\no $\bf{Conjecture\ 2:}$ There is some constant $c_2>0$ such that
\begin{align}
\lambda^{p-2} 1_{\{d(x,\pmR)\leq 1/\lambda\}}dx \overset{P}{\rightarrow} c_2\cL \text{ as } \lambda \to \infty  \text{ under } \N_0 \text{ or }  \P_{\delta_0},
\end{align}
which gives our conjecture on the Minkowski content of $\pmR$:

\no $\bf{Conjecture\ 3:}$ 
\begin{align}
\text{Cont}_{d+2-p}(\pmR)=c_2\cL(1), \ \N_0\text{-a.e.} \text{ or }  \P_{\delta_0}\text{-a.s.}
\end{align}
Here $\text{Cont}_{\delta}(A)$ is the $\delta$-dimensional Minkowski content of any compact set $A\subset \R^d$ defined by $\text{Cont}_{\delta}(A)=\lim_{r\to \infty} r^{(d-\delta)} |A^{\leq 1/r}|$, provided the limit exists. Here we use $|\cdot|$ to denote the $d$-dimensional volume (Lebesgue measure) in $\R^d$.
We hope to return to these problems in a future work.

\subsection{An Alternate model}

While it is easy to derive from the definition of $\cL^\lambda$ that the limiting measure $\cL$ will be supported on $F$, it is not obvious that its support is actually on the smaller set $\pmR$. To handle this issue we will construct another random measure $\widetilde{\cL}(\kappa)$ supported on $\pmR$ for any $\kappa>0$ by utilizing exit measures and show that there is some constant $c(\kappa)>0$ such that $\cL=c(\kappa)\widetilde{\cL}(\kappa)$ a.s., thus proving that $\cL$ indeed lives on $\pmR$. We also feel that the construction of $\widetilde{\cL}(\kappa)$ may be of independent interest, given the central role exit measures have played in the study of the boundaries of the range.
%
%
%
 We first introduce the definition of exit measure. For $K_1,K_2$ non-empty, set 
$d(K_1,K_2)=\inf\{|x-y|: x\in K_1, y\in K_2\}.$
Define
\begin{align} \label{Gdef}
\cO_{X_0}\equiv &\{ \text{open sets }  D \text{ satisfying } d(D^c,S(X_0))>0\text{ and a Brownian }\nn\\
&\text{  path starting from any $x\in\partial D$ will exit $D$ immediately}\}.
\end{align}
In what follows we always assume that $G\in \cO_{X_0}$. The exit measure of SBM $X$ from an open set $G$ under $\P_{X_0}$ or $\N_{X_0}$ is denoted by $X_G$ (see Chp. V of \cite{Leg99} for the construction of the exit measure). Intuitively $X_{G}$ is a random finite measure supported on $\partial G$,  which corresponds to the mass started at $X_0$ which is stopped at the instant it leaves $G$. What follows may be found in Chp. V of \cite{Leg99} (see also Section 1 of \cite{HMP18}).
The Laplace functional of $X_G$ is given by 
\begin{align}\label{e7.1}
\E_{X_0}(e^{-X_G(g)})&=\exp\Bigl(-\N_{X_0}\big(1-e^{-X_G(g)}\big)\Bigr)= e^{- X_0(U^g)},
\end{align}
where $g:\partial G\to[0,\infty)$ is continuous and $U^g\ge 0$ is the unique continuous function on $\overline G$ which is $C^2$ on $G$ and solves
\begin{equation}\label{e7.2}
\Delta U^g=(U^g)^2\text{ on }G,\quad U^g=g\text{ on }\partial G.
\end{equation}
Define
$G_\varepsilon^{x_0}=G_\varepsilon(x_0)=\{x:|x-x_0|>\varepsilon\} \text{ and set } G_\veps=G_\veps(0).$
For $\veps>0$ and $\lambda\geq 0$, we let $U^{\lambda,\veps}$ denote the unique continuous function on $\{|x|\ge \veps\}$ such that (cf. \eqref{e7.2})
\begin{equation}\label{ev5.2}
\Delta U^{\lambda,\veps}=(U^{\lambda,\veps})^2\ \ \text{for }|x|>\veps,\ \ \text{ and }\  U^{\lambda,\veps}(x)=\lambda\ \ \text{for }|x|=\veps.
\end{equation}
Uniqueness of solutions implies the scaling property (see (3.3) of \cite{MP17})
\begin{equation}\label{ev5.3}
U^{\lambda,\veps}(x)=\veps^{-2} U^{\lambda \veps^2,1}(x/\veps)\quad\text{for all }|x|\ge\veps,
\end{equation}
and also shows $U^{\lambda, \veps}$ is radially symmetric, thus allowing us to write $U^{\lambda,\veps}(|x|)$ for the value at $x\in\R^d$. 
 By \eqref{e7.1} we have for any $X_0\in M_F(\R^d)$ satisfying $S(X_0)\subset G_\veps$,
\begin{equation}\label{ev5.4}
\E_{X_0}(e^{-\lambda X_{G_\veps}(1)})=\exp\Bigl(-\N_{X_0}\big(1-e^{-\lambda X_{G_\veps}(1)}\big)\Bigr)=e^{-X_0(U^{\lambda, \veps})}.
\end{equation}
Let $\lambda\uparrow\infty$ in the above to see that $U^{\lambda,\veps}\uparrow U^{\infty,\veps}$ on $G_\veps$ and 
\begin{equation}\label{ev5.5}
\P_{X_0}(X_{G_\veps}(1)=0)=\exp(-X_0(U^{\infty,\veps})).
\end{equation}
Proposition V.9(iii) of \cite{Leg99} readily implies (see also (3.5) and (3.6) of \cite{MP17})
\begin{align}\label{ev5.6}
U^{\infty,\veps}\text{ is }C^2\text{ and }&\Delta U^{\infty,\veps}=(U^{\infty,\veps})^2\text{ on }G_\veps,\\
 \nonumber &\lim_{|x|\to\veps,|x|>\veps}U^{\infty,\veps}(x)=+\infty,\ \lim_{|x|\to\infty}U^{\infty,\veps}(x)=0.
\end{align}
Theorem 1.1 of \cite{Hong20} gives a construction of the local time $L^x$ in terms of the local asymptotic behavior of the exit measures at $x$. If $\psi_0(\eps)=\pi^{-1}\log^{+}(1/\eps)$ in $d=2$ and $\psi_0(\eps)=1/(2\pi\eps)$ in $d=3$, then for any $x\neq 0$, we have
\begin{equation}\label{e20.1}
X_{G_\eps^x}(1)\psi_0(\eps) \to L^x \text{ in measure under }   \N_{0} \text{ or } \P_{\delta_0} \text{ as } \eps \downarrow 0.
\end{equation}
%
Motivated by the above, for any $\kappa, \eps>0$, under $\P_{\delta_0}$ and $\N_{0}$ we define a measure $\widetilde{\cL}(\kappa)^\eps$ by 
\begin{equation}\label{edef}
d\widetilde{\cL}(\kappa)^\eps(x)=\frac{X_{G_\eps^x}(1)}{\eps^p} \exp(-\kappa \frac{X_{G_\eps^x}(1)}{\eps^{2}} ) 1(X_{G_{\eps/2}^x}=0)1(|x|>\eps)dx
\end{equation}
It is easy to derive from the definition of $X_{G_\eps^x}(1)$ (see Proposition V.1 and Lemma V.2 of \cite{Leg99}) that for any fixed $\eps>0$, $(\omega, x)\mapsto X_{G_\eps^x}(1)(\omega)$ is \break $\cF\times \mathfrak{B}(\R^d)$ measurable and so $\widetilde{\cL}(\kappa)^\eps$ is well defined and $\cF$-measurable.%

We can deduce from \eqref{e20.1} that $\widetilde{\cL}(\kappa)^\eps$ is closely related to ${\cL}^\lambda$(as in \eqref{e7.3}): for example in $d=3$, we have $\psi_0(\eps)=1/(2\pi \eps)$ and so $X_{G_\eps^x}(1) \sim \eps L^x$ as $\eps \downarrow 0$ by \eqref{e20.1}. Hence if $\lambda=\kappa \eps^{-1}$,
\begin{align}\label{edef3}
\frac{X_{G_\eps^x}(1)}{\eps^p} \exp(-\kappa\frac{X_{G_\eps^x}(1)}{\eps^{2}})\sim \eps^{1-p} L^x e^{-\kappa \eps^{-1} L^x}\sim \lambda^{1+\alpha} L^x e^{-\lambda L^x}
\end{align}
as $\eps \downarrow 0$, where in the last approximation we have used the fact that $\alpha=p-2$ in $d=3$. In \eqref{edef}, the indicator function $1(|x|>\eps)$ is to ensure that $X_{G_\eps^x}$ is well defined and the extra indicator $1(X_{G_{\eps/2}^x}=0)$ is to ensure that the limiting measures will be supported on $\pmR$ rather than $F$. 
We will show below that they indeed differ only up to some constant.


\begin{theorem}\label{tl1}
Let $d=2$ or $3$. For any $\kappa>0$, under both $\N_{0}$ and $\P_{\delta_0}$, there exists a random measure $\widetilde{\cL}(\kappa) \in M_F(\R^d)$, supported on $\pmR$, such  that $\widetilde{\cL}(\kappa)^\eps \overset{P}{\rightarrow} \widetilde{\cL}(\kappa)$ as $\eps\downarrow 0$ and there is a sequence $\eps_n \downarrow 0$ such that  $\widetilde{\cL}(\kappa)^{\eps_n} \to \widetilde{\cL}(\kappa)$ a.s. as $n \to \infty$. Moreover, there is some constant $c_{\ref{tl1}}(\kappa)>0$ such that $\widetilde{\cL}(\kappa)=c_{\ref{tl1}}(\kappa) {\cL}$ a.s.
\end{theorem}

Turning to the $\P_{X_0}$ case, again we will restrict our interest in $S(X_0)^c$ as in \eqref{em7.3}. For any $\kappa, \eps>0$, under $\P_{X_0}$ we define a measure $\widetilde{\cL}(\kappa)^\eps$ supported on $S(X_0)^c$ by 
\begin{equation}\label{edef2}
d\widetilde{\cL}(\kappa)^\eps(x)=
\frac{X_{G_\eps^x}(1)}{\eps^p} \exp(-\kappa \frac{X_{G_\eps^x}(1)}{\eps^{2}} ) 1(X_{G_{\eps/2}^x}=0)1_{(x\in S(X_0)^{>\eps})} dx
\end{equation}

\begin{theorem}\label{tl1.0.0}
Let $d=2$ or $3$ and $X_0\in M_F$. For any $\kappa>0$, under $\P_{X_0}$ there exists a $\sigma$-finite random measure $\widetilde{\cL}(\kappa)$,  supported on $\pmR\cap S(X_0)^c$, such that for any $k\geq 1$, $\widetilde{\cL}(\kappa)^\eps|_{S(X_0)^{\geq 1/k}}  \overset{P}{\rightarrow} \widetilde{\cL}(\kappa)|_{S(X_0)^{\geq 1/k}} $ as $\eps\downarrow 0$ and there is a sequence $\eps_n \downarrow 0$ such that $\widetilde{\cL}(\kappa)^{\eps_n}|_{S(X_0)^{\geq 1/k}} \to \widetilde{\cL}(\kappa)|_{S(X_0)^{\geq 1/k}}, \forall k\geq 1$ a.s. as $n \to \infty$. Moreover, we have $\widetilde{\cL}(\kappa)=c_{\ref{tl1}}(\kappa) {\cL}$ a.s.
\end{theorem}

\no $\mathbf{Organization\ of\ the\ paper.}$  In Section \ref{s2} we give preliminary results on super-Brownian motion, the Brownian snake, exit measures and their special Markov property. In Section \ref{s3} we establish the convergence of the mean measures of $\cL^\lambda$ and $\widetilde \cL(\kappa)^\eps$ and give the proof of Proposition \ref{p4.1}. In Section \ref{s4} the second moment convergence results will be given in Propositions \ref{p3.1}, \ref{p3.2} and \ref{p3.3} while we defer their proofs to Sections \ref{s8} and \ref{s9}. Assuming the results from Section \ref{s4}, we will finish the proofs of our main results Theorems \ref{t0} and \ref{tl1} under $\N_0$ and $\P_{\delta_0}$ in Section \ref{s5} while we include the similar proof of Theorems \ref{t0.0.1} and \ref{tl1.0.0} under $\P_{X_0}$ for general initial condition $X_0$ in the Appendix.  In Section \ref{s5} we also give the proof for the first and second moment measures of $\cL$ stated as in Theorems \ref{p0.0} and \ref{p0.0_1}. In Section \ref{s6} the proof of Theorem \ref{tl1.1} will be finished by utilizing the shrinking ball arguments from \cite{HMP18}. 
In Section \ref{s7} a key proposition in terms of a change of measure method is given and finally in Sections \ref{s8} and \ref{s9} we finish the essential proofs of Propositions \ref{p3.1}, \ref{p3.2} and \ref{p3.3}.

\section*{Acknowledgements}
This work was done as part of the author's graduate studies at the University of British Columbia. I would like to thank my supervisor, Professor Edwin Perkins, for suggesting this problem and for the helpful discussions and suggestions throughout this work.

\section{Exit Measures and the Special Markov Property}\label{s2}

We will use Le Gall's Brownian snake construction of a SBM $X$, with initial condition $X_0\in M_F(\R^d)$. Set $\cW=\cup_{t\ge 0} C([0,t],\R^d)$ with the natural metric (see page 54 of \cite{Leg99}), and let $\zeta(w)=t$ be the lifetime of \mbox{$w\in C([0,t],\R^d)\subset\cW$.} The Brownian snake $W=(W_t,t\ge0)$ is a $\cW$-valued continuous strong Markov process and, abusing notation slightly, let $\N_x$ denote its excursion measure starting from the path at $x\in\R^d$ with lifetime zero.  As usual we let $\hat W(t)=W_t(\zeta(W_t))$ denote the tip of the snake at time $t$, and $\sigma(W)>0$ denote the length of the excursion path. We refer the reader to Ch. IV of \cite{Leg99} for the precise definitions.  The construction of super-Brownian motion, $X=X(W)$ under $\N_x$ or $\P_{X_0}$, may be found in Ch. IV of \cite{Leg99}. The ``law" of $X(W)$ under $\N_x$ is the canonical measure of SBM starting at $x$ described in the last Section (and also denoted by $\N_x$). If $\Xi=\sum_{j\in J}\delta_{W_j}$ is a Poisson point process on $\cW$ with intensity $\N_{X_0}(dW)=\int\N_x(dW)X_0(dx)$, then by Theorem~4 of Ch. IV of \cite{Leg99} (cf. \eqref{ev2.1})
\begin{equation}\label{Xtdec}
X_t(W)=\sum_{j\in J}X_t(W_j)=\int X_t(W)\Xi(dW)\text{ for }t>0
\end{equation}
defines a SBM with initial measure $X_0$. We will refer to this as the standard set-up for $X$ under $\P_{X_0}$. It follows that the total local time $L^x$ under $\P_{X_0}$ may also be decomposed as 
\begin{equation}\label{ae4.1}
L^x=\sum_{j\in J}L^x(W_j)=\int L^x(W)\Xi(dW).
\end{equation}

Recall $\mathcal{R}=\overline{\{x:L^x>0\}}$ is the range of the SBM $X$ under $\P_{X_0}$ or $\N_{X_0}$.    Under $\N_{X_0}$ we have (see (8) on p. 69 of \cite{Leg99})
\begin{equation}\label{ea0.0}
\mR=\{\hat W(s):s\in[0,\sigma]\}.
\end{equation}

Let $G\in\cO_{X_0}$ as in \eqref{Gdef}.   
Then 
\begin{equation}\label{ea1.1}
X_G \text{ is a finite random measure supported on }\mathcal{R}\cap \partial G\text{ a.s.}
\end{equation}
Under $\N_{X_0}$ this follows from the definition of $X_G$ on p. 77 of \cite{Leg99} and 
the ensuing discussion, and \eqref{ea0.0}.
  Although \cite{Leg99} works under $\N_x$ for $x\in G$ the above extends immediately to $\P_{X_0}$ because as in (2.23) of \cite{MP17}, 
  \begin{equation}\label{ea1.2}
X_G=\sum_{j\in J} X_G(W_j)=\int X_G(W)d\Xi(W),
\end{equation}
where $\Xi$ is a Poisson point process on $\cW$ with intensity $\N_{X_0}$.  

Working under $\N_{X_0}$ and following \cite{Leg95}, we define
\begin{align}\label{ae8.1}
S_G(W_u)&=\inf\{t\le \zeta_u: W_u(t)\notin G\}\ \ (\inf\emptyset=\infty),\nn\\
\eta_s^G(W)&=\inf\{t:\int _0^t1(\zeta_u\le S_G(W_u))\,du>s\},\nn\\
\cE_G&=\sigma(W_{\eta_s^G},s\ge 0)\vee\{\N_{X_0}-\text{null sets}\},
\end{align}
where $s\to W_{\eta^G_s}$ is continuous (see p. 401 of \cite{Leg95}).  Write the open set \mbox{$\{u:S_G(W_u)<\zeta_u\}$} as countable union of disjoint open intervals, $\cup_{i\in I}(a_i,b_i)$.
Clearly $S_G(W_u)=S^i_G<\infty$ for all $u\in [a_i,b_i]$ and we may define
\[W^i_s(t)=W_{(a_i+s)\wedge b_i}(S^i_G+t)\text{ for }0\le t\le \zeta_{(a_i+s)\wedge b_i}-S^i_G.\]
Therefore for $i\in I$, $W^i\in C(\R_+,\cW)$ are the excursions of $W$ outside $G$. Proposition 2.3 of \cite{Leg95} implies $X_G$ is $\cE_G$-measurable and Corollary~2.8 of the same reference implies
\begin{equation}\label{SMP1}
\left\{\begin{array}{l}
\text{Conditional on $\cE_G$, the point measure $\sum_{i\in I}\delta_{W^i}$ is a Poisson}\\
\text{point measure with intensity $\N_{X_G}$.}\end{array}\right.
\end{equation}
If $D$ is an open set in $\cO_{X_0}$ 
such that $\bar{G}\subset D$ and $d(D^c,\bar G)>0$, then the definition (and existence) of $X_D(W)$ applies equally well to each $X_D(W^i)$
and it is easy to check that 
\begin{equation}\label{Widecomp}
X_D(W)=\sum_{i\in I}X_D(W^i).
\end{equation}

If $U$ is an open subset of $S(X_0)^c$, then $L_U$, the
restriction of the local time $L^x$ to $U$, is in $C(U)$ which is the set of continuous functions on $U$.
\begin{proposition}\label{pv0.2}
Let $X_0 \in M_F(\R^d)$.\\
(i)  Let $G$ be an open set in $\cO_{X_0}$. Let $\psi_0$ be a  bounded measurable function on $C({\overline{G}}^c)$ and $\Phi_1$ be a bounded measurable function on $M_F(\R^d)^n$ for any $n\geq 1$. Let $D_i$ be open sets in $\cO_{X_0}$, such that $\overline{G}\subset D_i$ and $d(D_i^c,\bar G)>0$, $\forall 1\leq i\leq n$. Then 
\[\N_{X_0} \Big(\psi_0(L_{{\overline{G}}^c}) \Phi_1(X_{D_1},\dots, X_{D_n})|\cE_{G}\Big)= \E_{X_G}\Big( \psi_0(L_{{\overline{G}}^c}) \Phi_1(X_{D_1},\dots,X_{D_n})\Big).\]
(ii) Let $G_1, G_2$ be open sets in $\cO_{X_0}$ such that $\overline{G_1}\subset G_2$ and $d(G_2^c,\overline{G_1})>0$. If $\psi_2:\cK \to \R$ is Borel measurable, then we have 
\begin{equation*} 
\N_{X_0}(\psi_2(\mR\cap G_2^c)|\cE_{G_1})=\E_{X_{G_1}}(\psi_2(\mR\cap G_2^c)),
\end{equation*}
where $\cK$ is the space of compact subsets of $\R^d$ equipped with the Hausdorff metric (see, e.g., Section 2 of \cite{HMP18}).

\no (iii) Let $G_1, G_2$ be open sets in $\cO_{X_0}$ such that $\overline{G_1}\subset G_2$ and $d(G_2^c,\overline{G_1})>0$. If $\psi_3:\R\to \R$ is Borel measurable, then for any $\lambda>0$ we have 
\begin{equation*} 
\N_{X_0}(\psi_3(\cL^\lambda(G_2^c))|\cE_{G_1})=\E_{X_{G_1}}(\psi_3(\cL^\lambda(G_2^c))).
\end{equation*}
\end{proposition}
\begin{proof}
(ii) follows immediately from Proposition 2.2 in \cite{HMP18}. (i) and (iii) will follow in a similar way as Proposition 2.6(b) of \cite{MP17}. 
\end{proof}

We will need a version of the above under $\P_{X_0}$ as well, which is Proposition 2.3 in \cite{HMP18}.
\begin{proposition}\label{pv0.1}
Let $X_0 \in M_F(\R^d)$.\\
(i)  Let $G$ be an open set in $\cO_{X_0}$. Let $\psi_0$ be a  bounded measurable function on $C({\overline{G}}^c)$ and $\Phi_1$ be a bounded measurable function on $M_F(\R^d)^n$ for any $n\geq 1$. Let $D_i$ be open sets in $\cO_{X_0}$, such that $\overline{G}\subset D_i$ and $d(D_i^c,\bar G)>0$, $\forall 1\leq i\leq n$. Then 
\[\E_{X_0} \Big(\psi_0(L_{{\overline{G}}^c}) \Phi_1(X_{D_1},\dots, X_{D_n})|\cE_{G}\Big)= \E_{X_G}\Big( \psi_0(L_{{\overline{G}}^c}) \Phi_1(X_{D_1},\dots,X_{D_n})\Big).\]
(ii) Let $G_1, G_2$ be open sets in $\cO_{X_0}$ such that $\overline{G_1}\subset G_2$ and $d(G_2^c,\overline{G_1})>0$. If $\psi_2:\cK \to \R$ is Borel measurable, then we have 
\begin{equation*} 
\E_{X_0}(\psi_2(\mR\cap G_2^c)|\cE_{G_1})=\E_{X_{G_1}}(\psi_2(\mR\cap G_2^c)).
\end{equation*}
\end{proposition}

 \section{Convergence of the mean measure and proof of Proposition \ref{p4.1}}\label{s3}
In this section we will give the convergence of first moment measures of $\cL^\lambda$ and $\widetilde{\cL}(\kappa)^\eps$ and finish the proof of Proposition \ref{p4.1}.

  \subsection{Mean measure for local time}

Recall $V^\lambda(x)=\N_{0}(1-e^{-\lambda L^x})$ as in \eqref{ev1.0} and $V^\lambda$ is also the solution to \eqref{ev1.1}. Uniqueness of solutions implies the scaling property (see (2.13) of \cite{MP17})
\begin{equation}\label{e12.1}
V^{\lambda}(x)=r^{-2} V^{\lambda r^{4-d}}(x/r) \quad\text{ for all } x\neq 0, r>0,
\end{equation}
and also shows $V^{\lambda}$ is radially symmetric, thus allowing us to write $V^{\lambda}(|x|)$ for the value at $x\in\R^d$. 
Monotone convergence and the convexity of $e^{-ax}$ for $a,x>0$ allow us to differentiate $V^\lambda(x)=\N_{0}(1-e^{-\lambda L^x})$ with respect to $\lambda>0$ through the expectation so that for any $\lambda>0$ we can define 
\begin{align}\label{e12.2}
 V_1^\lambda(x):=\frac{\partial}{\partial \lambda} V^\lambda(x)=\N_0(L^x e^{-\lambda L^x}), \forall x\neq 0.
  \end{align}
  By differentiating both sides of \eqref{e12.1} with respect to $\lambda>0$, we obtain
 \begin{align}\label{ef4.4}
 V_1^\lambda(x)=r^{-2} V_1^{\lambda r^{4-d}}(x/r) r^{4-d}=r^{-2}\N_0( r^{4-d} L^{x/r} e^{-\lambda r^{4-d} L^{x/r}}),
   \end{align}
which is also a consequence of the scaling of Brownian snake under $\N_0$ (see, e.g., the proof of Proposition V.9 (i) of \cite{Leg99}). Before turning to the calculation of the mean measure of $\cL^\lambda$, we recall $\alpha$ as in \eqref{ep1.1} and give the following result from Proposition 5.5 of \cite{MP17}.
\begin{lemma}\label{l12}
There is some constant $c_{\ref{l12}}>0$, depending on $d$ so that
\[V^\infty(x)-V^\lambda(x)\leq c_{\ref{l12}} |x|^{-p} \lambda^{-\alpha}, \forall x\neq 0, \lambda>0.\]
\end{lemma}
\no The following is an easy consequence of the above lemma.
\begin{proposition}\label{p1.1}
There is some constant $c_{\ref{p1.1}}>0$, depending on $d$ so that
 \begin{equation}\label{e9.3.1}
 \N_{0}\Big(\lambda^{1+\alpha} L^x e^{-\lambda L^x}\Big)=\lambda^{1+\alpha} V_1^\lambda(x) \leq c_{\ref{p1.1}} |x|^{-p}, \forall x\neq 0, \lambda>0.
\end{equation}
\end{proposition}
 \begin{proof}
The first equality is immediate by definition \eqref{e12.2}. One can also conclude from \eqref{e12.2} that $\lambda\mapsto V_1^\lambda(x)$ is monotone decreasing and so for any $\lambda>0$,
\begin{align}
V_1^\lambda(x)&\leq \frac{2}{\lambda}\int_{\lambda/2}^\lambda V_1^{\lambda'}(x) d\lambda'=\frac{2}{\lambda} (V^\lambda(x)-V^{\lambda/2}(x))\nn\\
&\leq \frac{2}{\lambda}(V^\infty(x)-V^{\lambda/2}(x))\leq  \frac{2}{\lambda} c_{\ref{l12}} |x|^{-p} \lambda^{-\alpha} 2^\alpha,
\end{align}
where the last is from Lemma \ref{l12}. Let $c_{\ref{p1.1}}=2^{1+\alpha} c_{\ref{l12}}$ to finish the proof.
 \end{proof}

Let $(B_t)$ denote a $d$-dimentional Brownian motion starting from $x$ under $P_x$. Define $\tau_{r}=\inf\{t\geq 0: |B_t|\leq r\}$ for any $r>0$ and let $r_\lambda=\lambda_0 \lambda^{-\frac{1}{4-d}}$ where $\lambda_0$ will be chosen to be some fixed large constant below. In what follows we will always assume $0<r_\lambda<|x|$.
\begin{lemma}\label{l12.2}
Let $\lambda>0$ and $|x|>r_\lambda>0$. For any $t>0$, we have \[V_1^{{\lambda}}(x)=E_x\Big( V_1^{{\lambda}}(B({t\wedge \tau_{r_\lambda}}))\exp\big(-\int_0^{t\wedge \tau_{r_\lambda}} V^{{\lambda}}(B_s) ds\big)\Big).\]
\end{lemma}
\begin{proof}
It follows in a similar way to that of Lemma 9.4 in \cite{MP17}.
\end{proof}

For $\gamma\in\R$, we let $(\rho_t)$ denote a $\gamma$-dimensional Bessel process starting from $r>0$ under $P_{r}^{(\gamma)}$ and let $(\cF_t^\rho)$ denote the right-continuous filtration generated by the Bessel process. We slightly abuse the notation and define $\tau_R=\tau_R^\rho=\inf\{t\geq 0: \rho_t\leq R\}$ for $R>0$.
The following results (i) and (ii) are from Lemmas 5.2 and 5.3 of \cite{MP17} and the last one follows from (ii) and a simple application of Cauchy-Schwartz inequality.

\begin{lemma} \label{l13.5} 
Assume $0<2\gamma \le \nu^2$ and $q>2$. Then\\
(i) \[E_r^{(2+2\nu)}\Bigl(\exp\Bigl(\int_0^{\tau_1}\frac{\gamma}{\rho_s^2}\,ds\Bigr)\Bigr|\tau_1<\infty\Bigr)=r^{\nu-\sqrt{\nu^2-2\gamma}}, \forall r\geq 1.\]
(ii)
\[\sup_{r\ge 1}E_r^{(2+2\nu)}\Bigl(\exp\Bigl(\int_0^{\tau_1}\frac{\gamma}{\rho_s^q}\,ds\Bigr)\Bigr|\tau_1<\infty\Bigr)\le C_{\ref{l13.5}}(q,\nu)<\infty.\]
(iii)
\[\inf_{r\ge 1}E_r^{(2+2\nu)}\Bigl(\exp\Bigl(-\int_0^{\tau_1}\frac{\gamma}{\rho_s^q}\,ds\Bigr)\Bigr|\tau_1<\infty\Bigr)\ge
c_{\ref{l13.5}}(q,\nu)>0.\]
\end{lemma}
\begin{lemma} \label{c13.5} 
Let $r_\lambda=\lambda_0 \lambda^{-\frac{1}{4-d}}$. There is some constant $c_{\ref{c13.5}}>0$ such that for all $\lambda_0>c_{\ref{c13.5}}$, $0<\gamma\leq 2$, there is some constant $C_{\ref{c13.5}}(\lambda_0, \nu,\gamma)>0$ so that for all $x\neq 0$,
\begin{align}
&\sup_{\lambda>0} E_{|x|}^{(2+2\nu)}\Big(\exp\big(\gamma\int_0^{ \tau_{r_\lambda}}  (V^\infty-V^{{\lambda}})(\rho_s) ds\big)\Big|\tau_{r_\lambda}<\infty\Big) \nn\\
&=\lim_{\lambda\to \infty} E_{|x|}^{(2+2\nu)}\Big(\exp\big(\gamma\int_0^{ \tau_{r_\lambda}}  (V^\infty-V^{{\lambda}})(\rho_s) ds\big)\Big|\tau_{r_\lambda}<\infty\Big)\nn\\
&=C_{\ref{c13.5}}(\lambda_0, \nu,\gamma)<\infty.
\end{align}
\end{lemma}
\begin{proof}
It follows in a similar way to the proof of Lemma 4.5 in \cite{Hong20} (see more details for the proof in Appendix).
 \end{proof}

We also state a result on the application of Girsanov's theorem on Bessel process from \cite{Yor92} (see also Proposition 2.5 of \cite{MP17}).
\begin{lemma}\label{l12.4}
Let $\lambda\geq 0$, $\mu\in \R, r>0$ and $\nu=\sqrt{\lambda^2+\mu^2}$. If $\Phi_t\geq 0$ is $\cF_t^\rho$-adapted, then for all $R<r$, we have
\[E_r^{(2+2\mu)}\Big(\Phi_{t\wedge \tau_R} \exp\Big(-\frac{\lambda^2}{2} \int_0^{t\wedge \tau_R} \frac{1}{\rho_s^2} ds\Big)\Big)=r^{\nu-\mu} E_{r}^{(2+2\nu)} \Big((\rho_{t\wedge \tau_R})^{-\nu+\mu} \Phi_{t\wedge \tau_R}\Big).\]
\end{lemma}
The following result is an easy application of the above lemma and is proved in Proposition 4.3 of \cite{Hong20}.
\begin{proposition}\label{p20.1}
Let $x\in \R^d-\{0\}$ and $0<\eps<|x|$. For any Borel measurable function $g: \R^+ \to \R$ bounded on $[r_0, r_0^{-1}]$ for any $r_0>0$, we have
\begin{align}
E_{x}\Big(&1(\tau_\eps<\infty)  \exp\big(-\int_0^{\tau_{\eps}} g(|B_s|) ds\big)\Big)\nn \\
&=\eps^p |x|^{-p}{E}_{|x|}^{(2+2\nu)}\Big( \exp\big(-\int_0^{\tau_{\eps}} (g(\rho_s)-V^\infty(\rho_s)) ds\big)\Big|\tau_\eps<\infty\Big),
\end{align}
where $B$ is a $d$-dimensional Brownian motion under $P_x$ for $d\leq 3$ and $\nu$ is as in \eqref{ev1.5}.
\end{proposition}


\begin{proposition}\label{p12}
There is some constant $c_{\ref{p12}}=c_{\ref{p12}}(d)>0$ such that
 \begin{equation*}
\lim_{\lambda\to \infty} \N_{0}\Big(\lambda^{1+\alpha} L^x e^{-\lambda L^x}\Big)= \lim_{\lambda \to \infty} \lambda^{1+\alpha} V_1^{{\lambda}}(x)= c_{\ref{p12}} |x|^{-p}, \forall x\neq 0.
\end{equation*}
\end{proposition}
\begin{proof}
Recall $r_\lambda=\lambda_0 \lambda^{-\frac{1}{4-d}}$. We use Lemma \ref{l12.2}, and the facts that $V_1^\lambda(x) \to 0$ as $|x| \to \infty$ and $V_1^{{\lambda}}(B_{t\wedge \tau_{r_\lambda}})$ is uniformly bounded for all $t\geq 0$ by Proposition \ref{p1.1}, to see that
\begin{align*}
&\lambda^{1+\alpha}V_1^{{\lambda}}(x)=\lambda^{1+\alpha}\lim_{t\to \infty} E_x\Big( V_1^{{\lambda}}(B_{t\wedge \tau_{r_\lambda}})\exp\big(-\int_0^{t\wedge \tau_{r_\lambda}} V^{{\lambda}}(B_s) ds\big)\Big)\\
&=\lambda^{1+\alpha}E_x\Big(1(\tau_{r_\lambda}<\infty) V_1^{{\lambda}}(B_{\tau_{r_\lambda}})\exp\big(-\int_0^{\tau_{r_\lambda}} V^{{\lambda}}(B_s) ds\big)\Big)\\
&=\lambda^{1+\alpha}V_1^{{\lambda}}(r_\lambda) E_{x}\Big(1(\tau_{r_\lambda}<\infty) \exp\big(-\int_0^{\tau_{r_\lambda}} V^{{\lambda}}(|B_s|) ds\big)\Big)\\
&=\lambda^{1+\alpha}V_1^{{\lambda}}(r_\lambda)\frac{r_\lambda^p}{|x|^p}E_{|x|}^{(2+2\nu)}\Big( \exp\big(\int_0^{ \tau_{r_\lambda}} (V^\infty-V^{{\lambda}})(\rho_s) ds\big)\Big|\tau_{r_\lambda}<\infty\Big),
\end{align*}
where the third equality is by the radial symmetry of $V^\lambda$ and $V_1^\lambda$. The last equality follows from Proposition \ref{p20.1} with $g=V^\lambda$. Use the scaling of $V_1^\lambda$ from \eqref{ef4.4} to see that the right-hand side of the above equals 
 \begin{align*}
& \lambda^{1+\alpha}r_\lambda^{2-d}V_1^{\lambda r_\lambda^{4-d}}(1)   \frac{r_\lambda^p}{|x|^p}  E_{|x|}^{(2+2\nu)}\Big(\exp\big(\int_0^{ \tau_{r_\lambda}} (V^\infty-V^{{\lambda}})(\rho_s)  ds\big)\Big|\tau_{r_\lambda}<\infty\Big)\nn\\
&=|x|^{-p}  \lambda_0^{p+2-d}  V_1^{\lambda_0^{4-d}}(1)  E_{|x|}^{(2+2\nu)}\Big(\exp\big(\int_0^{ \tau_{r_\lambda}} (V^\infty-V^{{\lambda}})(\rho_s) ds\big) \Big|\tau_{r_\lambda}<\infty\Big),
\end{align*}
where we have used the definition of $r_\lambda$ and $\alpha$ in the last equality.
 Choose $\lambda_0>c_{ \ref{c13.5}}$ and then apply Lemma \ref{c13.5} with $\gamma=1$ to conclude
   \begin{align}\label{6.3}
 \lim_{\lambda \to \infty} \lambda^{1+\alpha} V_1^{{\lambda}}(x)=  \lambda_0^{p+2-d} V_1^{\lambda_0^{4-d}}(1) C_{ \ref{c13.5}}(\lambda_0, \nu, 1)|x|^{-p},
  \end{align}
and so the proof is complete.
 \end{proof}
\begin{corollary}\label{ca1.3}
For any $x\in S(X_0)^c$ we have
\begin{align*}
\lim_{\lambda \to \infty} \E_{X_0}\Big(\lambda^{1+\alpha} L^x e^{-\lambda L^x}\Big) =  e^{-\int V^\infty(y-x) X_0(dy)}c_{\ref{p12}}\int |y-x|^{-p} X_0(dy).
\end{align*}
\end{corollary}
\begin{proof}
For any $x\in S(X_0)^c$, we have $d(x,S(X_0))>0$. 
Monotone convergence and the convexity of $e^{-ax}$ for $a,x>0$ allow us to differentiate \eqref{ev5.4} to get
\begin{align}\label{e1.2.7}
\E_{X_0}\Big(L^x e^{-\lambda L^x} \Big)=   \int V_1^{\lambda}(y-x) X_0(dy)e^{-\int V^{\lambda}(y-x) X_0(dy)}.
\end{align}
By Proposition \ref{p1.1}  we have
$\lambda^{1+\alpha} V_1^\lambda(y-x) \leq c_{\ref{p1.1}} |y-x|^{-p}, \forall y\neq x, \lambda>0,$
and so by Proposition \ref{p12} we may apply Dominated Convergence to get
\begin{align}\label{e1.2.8}
\lim_{\lambda \to \infty}  \int \lambda^{1+\alpha} V_1^{\lambda }(y-x) X_0(dy)= \int  c_{\ref{p12}} |y-x|^{-p} X_0(dy).
\end{align}
Then it follows easily from \eqref{e1.2.7}, \eqref{e1.2.8}  and monotone convergence.
\end{proof}
\subsection{Left tail of the local time}

 \begin{proof}[Proof of Proposition \ref{p4.1}]
 First recall $V^\lambda$ and $V^\infty$ from \eqref{ev1.0} and \eqref{ev1.2} to see that for all $|x|>0$, we have
 \begin{align}\label{ae4.2}
\lambda^{\alpha} \N_0\Big(e^{-\lambda L^x}  1(L^x>0)\Big)=\lambda^{\alpha} (V^\infty(x)-V^\lambda(x)).
 \end{align}
Let $d^\lambda(x)=V^\infty(x)-V^\lambda(x)$ and $r_\lambda$ be as in Lemma \ref{l12.2}. By the Feyman-Kac formula for $d^\lambda$ (as in (5.2) of \cite{MP17}), we get for $|x|>r_\lambda>0$,
 \begin{align}\label{ae4.3}
d^\lambda(x)&=d^\lambda(r_\lambda)E_x\Big(1_{\{\tau_{r_\lambda}<\infty\}} \exp\Big(-\int_0^{\tau_{r_\lambda}} \frac{(V^\infty+V^\lambda)(B_s)}{2} ds\Big)\Big).
 \end{align}
 By the scaling of $V^\lambda$ and $V^\infty$ and the definition of $r_\lambda$, we have
 \begin{align*}
 d^\lambda(r_\lambda)=(V^\infty-V^\lambda)(r_\lambda)=r_\lambda^{-2} (V^\infty(1)-V^{\lambda_0^{4-d}}(1))=r_\lambda^{-2} d^{\lambda_0^{4-d}}(1).
  \end{align*}
Use the above and \eqref{ae4.3} to see that
  \begin{align*}
&\lambda^{\alpha}d^\lambda(x)=\lambda^{\alpha} r_\lambda^{-2} d^{\lambda_0^{4-d}}(1)  E_x\Big(1_{\{\tau_{r_\lambda}<\infty\}} \exp\Big(-\int_0^{\tau_{r_\lambda}} \frac{(V^\infty+V^\lambda)(B_s)}{2} ds\Big)\Big)\nn\\
=&\lambda^{\alpha} r_\lambda^{-2} d^{\lambda_0^{4-d}}(1)r_\lambda^p |x|^{-p} E_{|x|}^{(2+2\nu)}\Big(\exp\Big(\int_0^{\tau_{r_\lambda}} \frac{(V^\infty-V^\lambda)(\rho_s)}{2} ds\Big)\Big|\tau_{r_\lambda}<\infty\Big)\nn\\
=&|x|^{-p}d^{\lambda_0^{4-d}}(1) \lambda_0^{p-2} E_{|x|}^{(2+2\nu)}\Big(\exp\Big(\int_0^{\tau_{r_\lambda}} \frac{(V^\infty-V^\lambda)(\rho_s)}{2} ds\Big)\Big|\tau_{r_\lambda}<\infty\Big),
 \end{align*}
%
%
where the second equality is by Proposition \ref{p20.1} and in the last equality we have used the definitions of $r_\lambda$ and $\alpha$. Choose $\lambda_0>c_{ \ref{c13.5}}$ so that we can apply Lemma \ref{c13.5} with $\gamma=1/2$ to get
  \begin{align}\label{ae8.2}
\lim_{\lambda\to \infty} \lambda^{\alpha}d^\lambda(x)=|x|^{-p} d^{\lambda_0^{4-d}}(1)\lambda_0^{p-2} C_{\ref{c13.5}}(\lambda_0, \nu,1/2).
 \end{align}
 Recalling \eqref{ae4.2}, we apply Tauberian theorem (see, e.g., Theorem 5.1 and 5.3 of Chp. XIII of \cite{Fel71}) to get
\begin{align}
\lim_{\lambda \to \infty} \lambda^{\alpha} \N_0\Big(0<L^x<1/\lambda\Big)=c_{\ref{p4.1}} |x|^{-p},  
 \end{align}
 where $c_{\ref{p4.1}}=(\Gamma(\alpha+1))^{-1}d^{\lambda_0^{4-d}}(1) \lambda_0^{p-2} C_{\ref{c13.5}}(\lambda_0, \nu,1/2)$ and the proof of (i) is complete. Turning to (ii) for $\P_{\delta_0}$, we note that for all $|x|>0$, by \eqref{ev1.0} and \eqref{ev1.2} we have
 \begin{align*}
&\lambda^{\alpha} \E_{\delta_0}\Big(e^{-\lambda L^x}  1_{(L^x>0)}\Big)\\
&=\lambda^{\alpha} (e^{-V^\lambda(x)}-e^{-V^\infty(x)})=\lambda^{\alpha} e^{-V^\infty(x)}(e^{V^\infty(x)-V^\lambda(x)}-1)\\
& \to e^{-V^\infty(x)}|x|^{-p} d^{\lambda_0^{4-d}}(1) \lambda_0^{p-2} C_{\ref{c13.5}}(\lambda_0, \nu,1/2) \text{ as }\lambda \to\infty, 
 \end{align*}
 where the last line follows from \eqref{ae8.2}. Then an application of Tauberian theorem will give us (ii) and the proof is complete.
 \end{proof}
 
 \subsection{Mean measure for exit measure}
Now we will turn to the alternate model using exit measures. Recall from \eqref{ev5.4} that
\begin{align}\label{e1.2.5}
U^{\lambda \eps^{-2}, \eps}(x)=\N_{0}\Big(1-\exp(-\lambda \frac{X_{G_\eps^x}(1)}{\eps^{2}}) \Big), \forall |x|>\eps.
\end{align}
Similar to \eqref{e12.2}, we can differentiate the above with respect to $\lambda>0$ through the expectation so that for any $\lambda>0$ and for all $|x|>\eps$, we have
\begin{align}\label{er12.2}
 U_1^{\lambda \eps^{-2}, \eps}(x):=\frac{\partial}{\partial \lambda} U^{\lambda \eps^{-2}, \eps}(x)=\N_{0}\Big( \frac{X_{G_\eps^x}(1)}{\eps^{2}}\exp(-\lambda \frac{X_{G_\eps^x}(1)}{\eps^{2}}) \Big).
  \end{align}

By using Proposition \ref{pv0.2}(i), for any $|x|>\eps>0$ we have (more details can be found in the derivation of (4.2) in \cite{Hong20})
\begin{align}\label{er1.5}
&\N_0\Big(\frac{X_{G_\eps^x}(1)}{\eps^p} \exp(-\kappa \frac{X_{G_\eps^x}(1)}{\eps^{2}})1(X_{G_{\eps/2}^x}=0)\Big)\nn\\
=&\N_{0}\Big( \frac{X_{G_\eps^x}(1)}{\eps^{p}}\exp(-(\kappa+4U^{\infty,1}(2)) \frac{X_{G_\eps^x}(1)}{\eps^{2}}) \Big).
 \end{align}
The following result on the convergence of the mean measure of $\widetilde{\cL}(\kappa)$ is proved in Theorem 1.3 of \cite{Hong20}.
  \begin{proposition}\label{p1.0}
For any $\kappa>0$, there is some constant $C_{\ref{p1.0}}(\kappa)>0$ such that for all $x\neq 0$, 
 \begin{align*}
&\lim_{\eps \downarrow 0} \N_0\Big(\frac{X_{G_\eps^x}(1)}{\eps^p} \exp(-\kappa \frac{X_{G_\eps^x}(1)}{\eps^{2}})1(X_{G_{\eps/2}^x}=0)\Big)=C_{\ref{p1.0}}(\kappa) |x|^{-p},
\end{align*}
and 
and for any $x\in S(X_0)^c$,
\begin{align*}
&\lim_{\eps\downarrow 0} \E_{X_0}\Big(\frac{X_{G_\eps^x}(1)}{\eps^p} \exp(-\kappa \frac{X_{G_\eps^x}(1)}{\eps^{2}})1(X_{G_{\eps/2}^x}=0) \Big)\nn\\
&\quad \quad \quad =  e^{-\int V^\infty(y-x) X_0(dy)}\int C_{\ref{p1.0}}(\kappa) |y-x|^{-p} X_0(dy).
\end{align*}
 Moreover, for any $\kappa>0$ and $x\neq 0$, we have
 \begin{align}\label{ce9.1.0}
& \N_0\Big(\frac{X_{G_\eps^x}(1)}{\eps^p} \exp(-\kappa \frac{X_{G_\eps^x}(1)}{\eps^{2}})1(X_{G_{\eps/2}^x}=0) \Big)\leq |x|^{-p}, \ \forall 0<\eps<|x|.
\end{align}
\end{proposition}

\section{Second moment convergence}\label{s4}

\noindent One important step in proving the existence of the limiting measure in Theorems \ref{t0}, \ref{t0.0.1} and Theorems \ref{tl1}, \ref{tl1.0.0}  is the exact convergence of the second moment measures, that is to say for any $x_1\neq x_2$, the limits
\begin{align}\label{1.1}
\begin{cases}
&\lim_{\lambda_1, \lambda_2 \to \infty} \lambda_1^{1+\alpha} \lambda_2^{1+\alpha}\N_x\Big( L^{x_1} L^{x_2} \exp\Big(-\sum_{i=1}^2 \lambda_i L^{x_i}\Big)\Big)\\
&\lim_{\eps_1, \eps_2 \downarrow 0} \N_x\Big(\prod_{i=1}^2 \frac{X_{G_{\eps_i}^{x_i}}(1)}{\varepsilon_i^p}\exp\Big(-\kappa \frac{X_{G_{\eps_i}^{x_i}}(1)}{\varepsilon_i^2}\Big) 1(X_{G_{\eps_i/2}^{x_i}}=0)\Big)
\end{cases}
\end{align}
exist for all $x\neq x_1, x_2$. Similarly for any $x_1, x_2 \in S(X_0)^c$, the existence of the following limits is required for $\P_{\delta_0}$ and $\P_{X_0}$ case:
\begin{align}\label{m8.2}
\begin{cases}
& \lim_{\lambda_1, \lambda_2 \to \infty} \lambda_1^{1+\alpha} \lambda_2^{1+\alpha}\E_{X_0}\Big( L^{x_1} L^{x_2} \exp\Big(-\sum_{i=1}^2 \lambda_i L^{x_i}\Big)\Big)\\
&\lim_{\eps_1, \eps_2 \downarrow 0} \E_{X_0}\Big(\prod_{i=1}^2 \frac{X_{G_{\eps_i}^{x_i}}(1)}{\varepsilon_i^p}\exp\Big(-\kappa \frac{X_{G_{\eps_i}^{x_i}}(1)}{\varepsilon_i^2}\Big)1(X_{G_{\eps_i/2}^{x_i}}=0) \Big).
\end{cases}
\end{align}

We first introduce some notations. For $x_1\neq x_2$, we let $\vec{x}=(x_1,x_2)$ and $\vec{\lambda}=(\lambda_1,\lambda_2) \in [0, \infty)^2\backslash \{(0,0)\}$. 
Define $V^{\vec{\lambda},\vec{x}}\geq 0$ to be
\begin{align}\label{e9.1}
V^{\vec{\lambda},\vec{x}}(x) \equiv \N_x\Big(1-\exp\Big(-\sum_{i=1}^2 \lambda_i L^{x_i}\Big) \Big),\ \forall x\neq x_1, x_2,
\end{align} 
so that for any $X_0\in M_F$ with $d(x_i, S(X_0))>0, i=1,2,$ 
\begin{align}\label{e9.9.1}
\E_{X_0}\Big(\exp\Big(-\sum_{i=1}^2 \lambda_iL^{x_i}\Big) \Big)=\exp\Big(-X_0(V^{\vec{\lambda},\vec{x}})\Big),
\end{align} 
where \eqref{e9.9.1} follows by \eqref{ae4.1} (see also Lemma 9.1 of \cite{MP17}). Pick $\eps_1, \eps_2>0$ small enough so that $B(x_1, \eps_1) \cap B(x_2, \eps_2)=\emptyset$. Let $\vec{\eps}=(\eps_1,\eps_2)$ and $G=G_{\varepsilon_1}^{x_1} \cap G_{\varepsilon_2}^{x_2}$. Define $U^{\vec{\lambda},\vec{x},\vec{\varepsilon}}\geq 0$ to be
\begin{align}\label{e4.3}
U^{\vec{\lambda},\vec{x},\vec{\eps}}(x) \equiv \N_x\Big(1-\prod_{i=1}^2 \exp\Big(-\lambda_i \frac{X_{G_{\eps_i}^{x_i}}(1)}{\varepsilon_i^2}\Big)1(X_{G_{\eps_i/2}^{x_i}}=0) \Big),\ \forall x\in G
\end{align} 
so that for any $X_0\in M_F$ with $d(S(X_0), G^c)>0$,
\begin{align}\label{e4.3.1}
\E_{X_0}\Big(\prod_{i=1}^2\exp\Big(-\lambda_i \frac{X_{G_{\eps_i}^{x_i}}(1)}{\varepsilon_i^2}\Big)1(X_{G_{\eps_i/2}^{x_i}}=0) \Big)=\exp\Big(-X_0(U^{\vec{\lambda},\vec{x},\vec{\eps}})\Big).
\end{align} 
\no The proof of \eqref{e4.3.1} follows easily from a monotone convergence theorem:
\begin{align*}
&\E_{X_0}\Big(\prod_{i=1}^2\exp\Big(-\lambda_i \frac{X_{G_{\eps_i}^{x_i}}(1)}{\varepsilon_i^2}\Big)1(X_{G_{\eps_i/2}^{x_i}}=0) \Big)\\
=&\lim_{n\to \infty} \E_{X_0}\Big(\exp\Big(-\sum_{i=1}^2\lambda_i \frac{X_{G_{\eps_i}^{x_i}}(1)}{\varepsilon_i^2}-\sum_{i=1}^2 nX_{G_{\eps_i/2}^{x_i}}(1)\Big)\Big)\nn\\
=&\lim_{n\to \infty}\exp\Big(-\int \N_x\Big(1- \exp\Big(-\sum_{i=1}^2\lambda_i \frac{X_{G_{\eps_i}^{x_i}}(1)}{\varepsilon_i^2}-\sum_{i=1}^2 n X_{G_{\eps_i/2}^{x_i}}(1)\Big)\Big)X_0(dx)\Big)\nn\\
=&\exp\Big(-X_0(U^{\vec{\lambda},\vec{x},\vec{\eps}})\Big),\nn
\end{align*}
where the second equality follows from the Poisson decomposition \eqref{ea1.2}. 

Monotone convergence and the convexity of $e^{-ax}$ for $a,x>0$ allow us to differentiate  \eqref{e4.3} with respect to $\lambda_i>0$ and then further differentiate with respect to $\lambda_{3-i}>0$ to get
\begin{align}\label{e4.3.2}
&U_i^{\vec{\lambda},\vec{x},\vec{\eps}}(x):=\frac{d}{d\lambda_i} U^{\vec{\lambda},\vec{x},\vec{\eps}}(x)\nn\\
&=\N_x\Big(\frac{X_{G_{\eps_i}^{x_i}}(1)}{\varepsilon_i^2} \prod_{j=1}^2\exp\Big(-\lambda_j \frac{X_{G_{\eps_j}^{x_j}}(1)}{\varepsilon_j^2}\Big)1(X_{G_{\eps_j/2}^{x_j}}=0) \Big),\ i=1,2,
\end{align} 
and
\begin{align}\label{e3.5}
&U_{1,2}^{\vec{\lambda},\vec{x},\vec{\eps}}(x):=\frac{d^2}{d\lambda_1 d\lambda_2 } U^{\vec{\lambda},\vec{x},\vec{\eps}}(x)\nn\\
&=-\N_x\Big(\prod_{i=1}^2 \frac{X_{G_{\eps_i}^{x_i}}(1)}{\eps_i^2}\exp\Big(-\lambda_i \frac{X_{G_{\eps_i}^{x_i}}(1)}{\eps_i^2} \Big)1(X_{G_{\eps_i/2}^{x_i}}=0) \Big).
\end{align}
Similarly we can differentiate \eqref{e9.1} to get
\begin{align}\label{e4.4.1}
V_i^{\vec{\lambda},\vec{x}}(x):=\frac{d}{d\lambda_i} V^{\vec{\lambda},\vec{x}}(x)=\N_x\Big(L^{x_i}\exp\Big(-\sum_{j=1}^2 \lambda_j L^{x_j}\Big) \Big),\ i=1,2,
\end{align} 
and
\begin{align}\label{e4.4.2}
V_{1,2}^{\vec{\lambda},\vec{x}}(x):=\frac{d^2}{d\lambda_1 d\lambda_2 } V^{\vec{\lambda},\vec{x}}(x)=-\N_x\Big(L^{x_1} L^{x_2} \exp\Big(-\sum_{i=1}^2 \lambda_i L^{x_i}\Big) \Big).
\end{align}
For the general initial condition case, we can also differentiate \eqref{e9.9.1} and \eqref{e4.3.1} to get
\begin{align}\label{m8.0}
&\E_{X_0}\Big(\prod_{j=1}^2\frac{X_{G_{\eps_j}^{x_j}}(1)}{\varepsilon_j^2}\exp\Big(-\lambda_j \frac{X_{G_{\eps_j}^{x_j}}(1)}{\varepsilon_j^2}\Big)1(X_{G_{\eps_j/2}^{x_j}}=0)\Big)\nn\\
=&\exp\Big(-X_0(U^{\vec{\lambda},\vec{x},\vec{\eps}})\Big)\Big(X_0(U_1^{\vec{\lambda},\vec{x},\vec{\eps}})X_0(U_2^{\vec{\lambda},\vec{x},\vec{\eps}})-X_0(U_{1,2}^{\vec{\lambda},\vec{x},\vec{\eps}})\Big).
\end{align} 
and 
\begin{align}\label{m8.1}
&\E_{X_0}\Big(L^{x_1} L^{x_2} \exp\Big(-\sum_{i=1}^2 \lambda_i L^{x_i}\Big) \Big)\nn\\
=&\exp\Big(-X_0(V^{\vec{\lambda},\vec{x}})\Big)\Big(X_0(V_1^{\vec{\lambda},\vec{x}})X_0(V_2^{\vec{\lambda},\vec{x}})-X_0(V_{1,2}^{\vec{\lambda},\vec{x}})\Big).
\end{align} 
Hence one can see that it suffices to consider the convergence of $U_{i}^{\vec{\lambda},\vec{x},\vec{\eps}}(x)$, $V_{i}^{\vec{\lambda},\vec{x}}(x)$, $i=1, 2$ and $U_{1,2}^{\vec{\lambda},\vec{x},\vec{\eps}}(x)$, $V_{1,2}^{\vec{\lambda},\vec{x}}(x)$ for the proofs of \eqref{1.1} and \eqref{m8.2}. 

\begin{proposition}\label{p3.1}
Fix any $x_1\neq x_2$. 

\no (i) There exists some constant $K_{\ref{p3.1}}>0$ so that for all $x\neq x_1, x_2$, 
\begin{align*}
&\lim_{\lambda_1, \lambda_2 \to \infty}  \lambda_i^{1+\alpha} V_{i}^{\vec{\lambda},\vec{x}}(x)= K_{\ref{p3.1}}U_i^{\vec{\infty},\vec{x}}(x),\quad  i=1,2,
\end{align*}
where $U_i^{\vec{\infty},\vec{x}}$ is as in \eqref{eu1u2}. Moreover, $K_{\ref{p3.1}}=c_{\ref{p12}}$.
  
\no (ii) For any $\lambda_1, \lambda_2>0$, there exist some constants $C_{\ref{p3.1}}(\lambda_1), C_{\ref{p3.1}}(\lambda_2)>0$ such that  for all $x\neq x_1, x_2$, we have
\begin{align*}
\lim_{\eps_1, \eps_2 \downarrow 0} \frac{1}{\eps_{i}^{p-2}}U_i^{\vec{\lambda},\vec{x},\vec{\eps}}(x)=C_{\ref{p3.1}}(\lambda_i) U_i^{\vec{\infty},\vec{x}}(x), i=1,2.
\end{align*}
Moreover, the multiplicative constant $c_{\ref{tl1}}(\kappa)$ in Theorem \ref{tl1} is $C_{\ref{p3.1}}(\kappa)K_{\ref{p3.1}}^{-1}$.
\end{proposition}

\begin{proposition}\label{p3.2}

Fix any $x_1\neq x_2$. For all $x\neq x_1, x_2$, we have
\begin{align*}
(i) &\lim_{\lambda_1, \lambda_2 \to \infty}  \lambda_1^{1+\alpha} \lambda_2^{1+\alpha} (-V_{1,2}^{\vec{\lambda},\vec{x}}(x))= K_{\ref{p3.1}}^2(-U_{1,2}^{\vec{\infty},\vec{x}}(x)).\\
 (ii) &
\lim_{\eps_1, \eps_2 \downarrow 0} \frac{1}{\eps_1^{p-2}}\frac{1}{\eps_2^{p-2}} (-U_{1,2}^{\vec{\lambda},\vec{x},\vec{\eps}}(x))=C_{\ref{p3.1}}(\lambda_1) C_{\ref{p3.1}}(\lambda_2)(-U_{1,2}^{\vec{\infty},\vec{x}}(x)).
\end{align*}
Here $U_{1,2}^{\vec{\infty},\vec{x}}$ is as in \eqref{eu12} and there is some universal constant $c_{\ref{p3.2}}>0$ such that for all $x\neq x_1, x_2$,
\begin{align}\label{eu12b}
0\leq -U_{1,2}^{\vec{\infty},\vec{x}}(x)\leq c_{\ref{p3.2}} (|x-x_1|^{-p}+|x-x_2|^{-p}) |x_1-x_2|^{2-p}.
\end{align}
\end{proposition}
The proofs of Propositions \ref{p3.1} and \ref{p3.2} are long and involving and will be deferred to Sections \ref{s8} and \ref{s9}. 

In order to prove that $\widetilde \cL(\kappa)=c_{\ref{tl1}}(\kappa) \cL$ a.s., we implement ideas from the above: For any $x_1\neq x_2$, $\lambda_1, \lambda_2>0$ and $0<\eps<|x_2-x_1|/4$, we define $W^{\vec{\lambda},\vec{x},{\eps}}\geq 0$ for all $x\neq x_1$ and $|x-x_2|>\eps$ by
\begin{align}\label{e9.2}
W^{\vec{\lambda},\vec{x},{\eps}}(x) \equiv \N_x\Big(1-e^{-\lambda_1 L^{x_1}}\exp\Big(- \lambda_2\frac{X_{G_{\eps}^{x_2}}(1)}{\eps^2}\Big)1(X_{G_{\eps/2}^{x_2}}=0) \Big),
\end{align} 
so that for any $X_0\in M_F$ with $d(x_1, S(X_0))>0$ and $\overline{B(x_2, \eps)}\subset S(X_0)^c$,
\begin{align}\label{e10.01}
\E_{X_0}\Big(\exp\Big(-\lambda_1 L^{x_1}-\lambda_{2} \frac{X_{G_{\eps}^{x_{2}}}(1)}{\eps^2}\Big)1(X_{G_{\eps/2}^{x_2}}=0) \Big)=e^{-X_0(W^{\vec{\lambda},\vec{x},{\eps}})},
\end{align} 
where \eqref{e10.01} follows as in \eqref{e4.3.1}. 
Similar to \eqref{e4.3.2} and \eqref{e3.5}, we can differentiate \eqref{e9.2} with respect to $\lambda_i>0$ and then further differentiate with respect to $\lambda_{3-i}>0$ to get
\begin{align*}
\begin{cases}
&W_1^{\vec{\lambda},\vec{x},{\eps}}(x):=\frac{d}{d\lambda_1} W^{\vec{\lambda},\vec{x},{\eps}}(x)\\
&\quad \quad \quad \quad \quad=\N_x\Big(L^{x_1}e^{-\lambda_1 L^{x_1}}\exp\Big(- \lambda_2\frac{X_{G_{\eps}^{x_2}}(1)}{\eps^2}\Big)1(X_{G_{\eps/2}^{x_2}}=0) \Big),\\
& W_2^{\vec{\lambda},\vec{x},{\eps}}(x):=\frac{d}{d\lambda_2} W^{\vec{\lambda},\vec{x},{\eps}}(x)\\
&\quad \quad \quad \quad \quad =\N_x\Big(\frac{X_{G_{\eps}^{x_2}}(1)}{\eps^2}\exp\Big(- \lambda_2\frac{X_{G_{\eps}^{x_2}}(1)}{\eps^2}\Big)1(X_{G_{\eps/2}^{x_2}}=0) e^{-\lambda_1 L^{x_1}}\Big),
\end{cases}
\end{align*} 
 and
\begin{align}
&W_{1,2}^{\vec{\lambda},\vec{x},{\eps}}(x):=\frac{d^2}{d\lambda_1 d\lambda_2 } W^{\vec{\lambda},\vec{x},{\eps}}(x)\nn\\
=&-\N_x\Big(L^{x_1} e^{-\lambda_1 L^{x_1}}\frac{X_{G_{\eps}^{x_2}}(1)}{\eps^2}\exp\Big(-\lambda_2\frac{X_{G_{\eps}^{x_2}}(1)}{\eps^2}\Big)1(X_{G_{\eps/2}^{x_2}}=0) \Big).
\end{align}

\no For the general initial condition case, we can differentiate \eqref{e10.01} to get
\begin{align}\label{m8.00}
&\E_{X_0}\Big(L^{x_1}e^{-\lambda_1 L^{x_1}} \frac{X_{G_{\eps}^{x_2}}(1)}{\eps^2}\exp\Big(-\lambda_2\frac{X_{G_{\eps}^{x_2}}(1)}{\eps^2}\Big)1(X_{G_{\eps/2}^{x_2}}=0) \Big)\nn\\
& =\exp\Big(-X_0(W^{\vec{\lambda},\vec{x},\eps})\Big)\Big(X_0(W_{1}^{\vec{\lambda},\vec{x},\eps})X_0(W_{2}^{\vec{\lambda},\vec{x},\eps})-X_0(W_{1,2}^{\vec{\lambda},\vec{x},\eps})\Big).
\end{align} 

\no We will also need the following mixture of Propositions \ref{p3.1} and \ref{p3.2}.

\begin{proposition}\label{p3.3}
Fix any $x_1\neq x_2$. For all $x\neq x_1, x_2$, we have
 \begin{align*}
(i) &\begin{cases}
&\lim_{\lambda_1 \to \infty, \eps \downarrow 0} \lambda_1^{1+\alpha} W_1^{\vec{\lambda},\vec{x},\eps}(x)=K_{\ref{p3.1}} U_{1}^{\vec{\infty},\vec{x}}(x)\\
&\lim_{\lambda \to \infty, \eps \downarrow 0} \frac{1}{\eps^{p-2}}W_2^{\vec{\lambda},\vec{x},\eps}(x)=C_{\ref{p3.1}}(\lambda_2)U_{2}^{\vec{\infty},\vec{x}}(x).\\
\end{cases}\\
(ii) &\lim_{\lambda_1 \to \infty, \eps \downarrow 0} \lambda_1^{1+\alpha}\frac{1}{\eps^{p-2}}(-W_{1,2}^{\vec{\lambda},\vec{x},\eps}(x))=K_{\ref{p3.1}}C_{\ref{p3.1}}(\lambda_2)(-U_{1,2}^{\vec{\infty},\vec{x}}(x)).
\end{align*}
\end{proposition}
\no The proof of Proposition \ref{p3.3} follows in a similar way to the proofs of Proposition \ref{p3.1} and Proposition \ref{p3.2} and is deferred to Section \ref{s8} and Section \ref{s9}. We will first proceed to the proof of our main results.

\section{Proofs of Theorems \ref{t0} and \ref{tl1} and Theorems \ref{p0.0} and \ref{p0.0_1}}\label{s5}
In this section, we will finish the proofs of Theorems \ref{t0}, \ref{tl1} and Theorems \ref{p0.0}, \ref{p0.0_1} assuming Propositions \ref{p3.1}, \ref{p3.2} and \ref{p3.3}.

\subsection{Preliminaries}
Recall from last section the definitions of $V^{\vec{\lambda},\vec{x}}, U^{\vec{\lambda},\vec{x},\vec{\eps}}$ and $W^{\vec{\lambda},\vec{x},\eps}$ and their first and second derivatives and recall $V^{\vec{\infty},\vec{x}}$ from \eqref{ev}. Fix $x_1\neq x_2$. It is not hard to check that (see Lemma \ref{l7.5}) for all $x\neq x_1, x_2$,
\begin{align}\label{ae7.6}
\lim_{\eps_1, \eps_2 \downarrow 0} U^{\vec{\lambda},\vec{x},\vec{\eps}}(x)=\lim_{\lambda_1, \lambda_2 \to \infty}  V^{\vec{\lambda},\vec{x}}(x)=\lim_{\lambda_1 \to \infty, \eps \downarrow 0} W^{\vec{\lambda},\vec{x},\eps}(x)=V^{\vec{\infty},\vec{x}}(x).
\end{align}
Recall \eqref{e4.4.1} and use Proposition \ref{p1.1} to get for all $\lambda_1, \lambda_2>0$ and $x\neq x_1, x_2$,
 \begin{align}\label{m10.0}
 \lambda_i^{1+\alpha}V_i^{\vec{\lambda},\vec{x}}(x)\leq  \N_{x}\Big(\lambda_i^{1+\alpha} L^{x_i}\exp(-\lambda_i L^{x_i}) \Big)\leq c_{\ref{p1.1}} |x-x_i|^{-p},  i=1,2.
\end{align}
Recall \eqref{e4.3.2}. Similarly we can get for all $\eps_1,\eps_2, \lambda_1, \lambda_2>0$ and for all $x$ so that $|x-x_i|>\eps_i,$ for $i=1,2,$
\begin{align}\label{mc1.0.1}
 \frac{1}{\eps_i^{p-2}}U_i^{\vec{\lambda},\vec{x},\vec{\eps}}(x)\leq&  \N_{x}\Big(\frac{X_{G_{\eps_i}^{x_i}}(1)}{\eps_i^p}\exp(- \lambda_i \frac{X_{G_{\eps_i}^{x_i}}(1)}{\eps_i^{2}}) 1(X_{G_{\eps_i/2}^{x_i}}=0)\Big) \nn\\
 \leq & |x-x_i|^{-p},
\end{align}
where the last equality is by \eqref{ce9.1.0}.  Recall \eqref{m10.0} and \eqref{mc1.0.1}. It follows that for any $\lambda_1, \lambda_2, \eps>0$ and for all $x$ with $x\neq x_1$ and $|x-x_2|>\eps$, we have
 \begin{align}\label{m8.7}
\lambda_1^{1+\alpha} W_1^{\vec{\lambda},\vec{x},\eps}(x)\leq \lambda_1^{1+\alpha} \N_x\Big(L^{x_1}\exp\Big(-\lambda_1 L^{x_1}\Big) \Big)\leq c_{\ref{p1.1}}|x-x_1|^{-p}, 
\end{align} 
and 
 \begin{align}\label{m8.8}
 \frac{1}{\eps^{p-2}} W_2^{\vec{\lambda},\vec{x},\eps}(x)&\leq \N_x\Big(\frac{X_{G_{\eps}^{x_2}}(1)}{\eps^p}\exp\Big(-\lambda_2\frac{X_{G_{\eps}^{x_2}}(1)}{\eps^2}\Big)1(X_{G_{\eps/2}^{x_2}}=0) \Big)\nn\\
&\leq |x-x_2|^{-p}.
\end{align} 

The proof of Proposition 6.1 of \cite{MP17} readily implies that (note $U^{\vec{\lambda},\vec{x}}$ is used there to denote our $V_{1,2}^{\vec{\lambda},\vec{x}}$ here) for all $x_1\neq x_2$,  if $|x-x_1| \wedge |x-x_2|>\eps_0$ for some $\eps_0>0$, then there is some constant $C(\eps_0)>0$ so that
\begin{align}\label{eu12b2}
0\leq \lambda_{1}^{1+\alpha} \lambda_{2}^{1+\alpha}(-V_{1,2}^{\vec{\lambda},\vec{x}}(x)) \leq C(\eps_0)(1+|x_1-x_2|^{2-p}),\quad \forall \lambda_1, \lambda_2\geq 1.
\end{align}
Similarly one can show that (see Lemma \ref{l4.7}) for all $\eps_1, \eps_2>0$ small,
\begin{align}\label{eu12b1}
0\leq \frac{1}{\eps_{1}^{p-2}} \frac{1}{\eps_{2}^{p-2}}(-U_{1,2}^{\vec{\lambda},\vec{x},\vec{\eps}}(x)) \leq C(\eps_0)(1+|x_1-x_2|^{2-p}),
\end{align}
and for all $\lambda_1\geq 1$ large and $\eps>0$ small,
\begin{align}\label{eu12b3}
0\leq \lambda_1^{1+\alpha} \frac{1}{\eps^{p-2}}(-W_{1,2}^{\vec{\lambda},\vec{x},\eps}(x))) \leq C(\eps_0)(1+|x_1-x_2|^{2-p}).
\end{align}


\begin{theorem}\label{tl2}
For any bounded Borel function $h: \R^d \times \R^d \to \R$ supported on 
 \mbox{$\{(x_1,x_2): \eps_0 \leq |x_1|, |x_2|\leq \eps_0^{-1}\}$} for some $\eps_0>0$, we have
\begin{align*}
(a) &\lim_{\lambda_1,\lambda_2\to \infty} \N_0\Big((\cL^{\lambda_1} \times \cL^{\lambda_2})(h)\Big) =K_{\ref{p3.1}}^2 \int h(x_1,x_2) (-U_{1,2}^{\vec{\infty},\vec{x}}(0))  dx_1dx_2,\\
(b) &\lim_{\eps_1,\eps_2\downarrow 0} \N_0\Big((\widetilde{\cL}(\kappa)^{\eps_1} \times \widetilde{\cL}(\kappa)^{\eps_2})(h)\Big) =C_{\ref{p3.1}}(\kappa)^2\int h(x_1,x_2) (-U_{1,2}^{\vec{\infty},\vec{x}}(0))  dx_1dx_2.
\end{align*}

\end{theorem}

\begin{proof}
It suffices to consider nonnegative bounded Borel function $h$. By an application of Fubini's theorem, we have
\begin{align*}
&\N_0\Big((\cL^{\lambda_1} \times \cL^{\lambda_2})(h)\Big)\\
=& \int_{\eps_0 \leq |x_1|, |x_2|\leq \eps_0^{-1}} h(x_1, x_2) \N_0\Big( \lambda_1^{1+\alpha} \lambda_2^{1+\alpha} L^{x_1} L^{x_2}e^{-\lambda_1 L^{x_1}}e^{-\lambda_2 L^{x_2}}\Big)dx_1 dx_2\\
=& \int_{\eps_0 \leq |x_1|, |x_2|\leq \eps_0^{-1}} h(x_1, x_2) \lambda_1^{1+\alpha} \lambda_2^{1+\alpha} (-V_{1,2}^{\vec{\lambda},\vec{x}}(0)) dx_1 dx_2,
\end{align*}
where the second equality is by \eqref{e4.4.2}.
Since $h$ is bounded, in view of \eqref{eu12b2} we can see that the integrand has an integrable bound and so (a) will follow immediately by Dominated Convergence using Proposition \ref{p3.2}(i). Similarly (b) will follow from Proposition \ref{p3.2}(ii) and \eqref{eu12b1}.
\end{proof}

\begin{corollary}\label{c0}
For any bounded Borel function $\phi$ on $\R^d$ and for any $k\geq 1$, we have
${\cL}^\lambda(\phi \cdot 1(k^{-1}\leq |\cdot|\leq k) )$ converges in $L^2(\N_0)$ as $\lambda \to \infty$ and $\widetilde{\cL}(\kappa)^\eps(\phi \cdot 1(k^{-1}\leq |\cdot|\leq k) )$ converges in $L^2(\N_0)$ as $\eps \downarrow 0$.
\end{corollary}
\begin{proof}
For any bounded Borel function $\phi$ we let  
\begin{align}\label{be1.2}
 h(x_1,x_2)=\phi(x_1)  1(k^{-1}\leq |x_1|\leq k)  \cdot \phi(x_2) 1(k^{-1}\leq |x_2|\leq k),
\end{align}
 and apply Theorem \ref{tl2}(b) with the above $h$ to get 
\begin{align*}
&\lim_{\eps_1,\eps_2\downarrow 0} \N_0\Big(\Big(\widetilde{\cL}(\kappa)^{\eps_1}(\phi \cdot 1(k^{-1}\leq |\cdot|\leq k))-\widetilde{\cL}(\kappa)^{\eps_2}(\phi \cdot 1(k^{-1}\leq |\cdot|\leq k))\Big)^2\Big)\\ 
=&\lim_{\eps_1,\eps_2\downarrow 0} \N_0\Big((\widetilde{\cL}(\kappa)^{\eps_1} \times \widetilde{\cL}(\kappa)^{\eps_1})(h)\Big)-2 \N_0\Big((\widetilde{\cL}(\kappa)^{\eps_1} \times \widetilde{\cL}(\kappa)^{\eps_2})(h)\Big)\\
&\quad \quad+\N_0\Big((\widetilde{\cL}(\kappa)^{\eps_2} \times \widetilde{\cL}(\kappa)^{\eps_2})(h)\Big)=0.
\end{align*}
Therefore $\{\widetilde{\cL}(\kappa)^\eps(\phi \cdot 1(k^{-1}\leq |\cdot|\leq k) ): \eps>0\}$ is a Cauchy sequence in $L^2(\N_0)$ as $\eps \downarrow 0$ and so converges in $L^2(\N_0)$. The case for ${\cL}^\lambda$ is similar.
\end{proof}
\begin{corollary}\label{c3}
For any bounded Borel function $\phi$ on $\R^d$ and for any $k\geq 1$, we have
\mbox{$C_{\ref{p3.1}}(\kappa)\cL^\lambda (\phi \cdot 1_{(k^{-1}\leq |\cdot|\leq k)} )-K_{\ref{p3.1}}\widetilde \cL(\kappa)^\eps(\phi \cdot 1_{(k^{-1}\leq |\cdot|\leq k)} )$} converges to $0$ in $L^2(\N_0)$  as $\lambda \to \infty$ and $\eps \downarrow 0$.
\end{corollary}
\begin{proof}
 It suffices to prove for nonnegative $\phi\geq 0$. Let $h(x_1,x_2)$ be as in \eqref{be1.2} and use Fubini's theorem to get
 \begin{align*}
 &\N_0\Big(\cL^\lambda (\phi \cdot 1_{(k^{-1}\leq |\cdot|\leq k)} )\times \widetilde \cL(\kappa)^\eps(\phi \cdot 1_{(k^{-1}\leq |\cdot|\leq k)} )\Big)\\
 =&\int h(x_1, x_2) \N_0\Big(\lambda^{1+\alpha} L^{x_1} e^{-\lambda L^{x_1}}\frac{X_{G_{\eps}^{x_2}}(1)}{\eps^p}e^{-\kappa \frac{X_{G_{\eps}^{x_2}}(1)}{\eps^2}}1_{\{X_{G_{\eps/2}^{x_2}}=0\}} \Big) dx_1 dx_2\\
 =&\int_{k^{-1}\leq |x_1|, |x_2| \leq k, x_1\neq x_2} h(x_1, x_2) \lambda^{1+\alpha} \frac{1}{\eps^{p-2}} (-W_{1,2}^{\vec{\lambda},\vec{x},\eps}(0)) dx_1 dx_2,
 \end{align*}
 where $\vec{\lambda}=(\lambda, \kappa)$.
 Now apply Proposition \ref{p3.3}(ii), \eqref{eu12b3} and Dominated Convergence to conclude
 \begin{align}\label{a.0.0}
 \lim_{\lambda \to \infty, \eps\downarrow 0} &\N_0\Big(\cL^\lambda (\phi \cdot 1_{(k^{-1}\leq |\cdot|\leq k)} )\times \widetilde \cL(\kappa)^\eps(\phi \cdot 1_{(k^{-1}\leq |\cdot|\leq k)} )\Big) \nn\\
 &=K_{\ref{p3.1}}C_{\ref{p3.1}}(\kappa) \int_{k^{-1}\leq |x_1|, |x_2| \leq k, x_1\neq x_2} h(x_1, x_2) (-U_{1,2}^{\vec{\infty},\vec{x}}(0))dx_1 dx_2.
  \end{align}
Therefore
\begin{align*}
&\lim_{\lambda \to \infty, \eps\downarrow 0} \N_0\Big(\Big(C_{\ref{p3.1}}(\kappa){\cL}^{\lambda}(\phi \cdot 1_{(k^{-1}\leq |\cdot|\leq k)})-K_{\ref{p3.1}} \widetilde{\cL}(\kappa)^{\eps}(\phi \cdot 1_{(k^{-1}\leq |\cdot|\leq k)})\Big)^2\Big)\\ 
&=\lim_{\lambda \to \infty, \eps\downarrow 0} C_{\ref{p3.1}}(\kappa)^2\N_0\Big(({\cL}^{\lambda} \times {\cL}^{\lambda})(h)\Big)+K_{\ref{p3.1}}^2 \N_0\Big((\widetilde{\cL}(\kappa)^{\eps} \times \widetilde{\cL}(\kappa)^{\eps})(h)\Big)\\
&\quad -2K_{\ref{p3.1}} C_{\ref{p3.1}}(\kappa)\N_0\Big(\cL^\lambda (\phi \cdot 1_{(k^{-1}\leq |\cdot|\leq k)} )\times \widetilde \cL^\eps(\phi \cdot 1_{(k^{-1}\leq |\cdot|\leq k)} )\Big)=0,
\end{align*}
where we have used Theorem \ref{tl2} and \eqref{a.0.0} in the last equality.
\end{proof}

\noindent We continue to accommodate $\P_{X_0}$ for the general initial condition case.  

\begin{theorem}\label{tl3}
For any bounded Borel function $h: \R^d \times \R^d \to \R$ supported on $\{(x_1,x_2): x_i\in S(X_0)^{>\eps_0}\cap B(\eps_0^{-1}), i=1,2\}$ for some $\eps_0>0$, we have
 \begin{align}\label{ep1.1.2}
  \begin{cases}
&\lim_{\lambda_1,\lambda_2\to \infty} \E_{X_0}\Big((\cL^{\lambda_1} \times \cL^{\lambda_2})(h)\Big)=K_{\ref{p3.1}}^2 \int h(x_1,x_2)\\
&\quad \quad \quad e^{-X_0(V^{\vec{\infty},\vec{x}})} \Big(X_0(U_1^{\vec{\infty},\vec{x}}) X_0(U_2^{\vec{\infty},\vec{x}})-X_0(U_{1,2}^{\vec{\infty},\vec{x}})\Big)dx_1dx_2,\\
&\lim_{\eps_1,\eps_2 \downarrow 0} \E_{X_0}\Big((\widetilde{\cL}(\kappa)^{\eps_1} \times \widetilde{\cL}(\kappa)^{\eps_2})(h)\Big)=C_{\ref{p3.1}}(\kappa)^2\int h(x_1,x_2)\\
&\quad \quad \quad e^{-X_0(V^{\vec{\infty},\vec{x}})} \Big(X_0(U_1^{\vec{\infty},\vec{x}}) X_0(U_2^{\vec{\infty},\vec{x}})-X_0(U_{1,2}^{\vec{\infty},\vec{x}})\Big)dx_1dx_2.
\end{cases}
\end{align}
\end{theorem}
\begin{proof}
It suffices to prove for nonnegative $h$. By Fubini's theorem and \eqref{m8.1}, we have
\begin{align*}
& \E_{X_0}\Big((\cL^{\lambda_1} \times \cL^{\lambda_2})(h)\Big)=\int_{x_1, x_2 \in B(\eps_0^{-1})\cap S(X_0)^{>\eps_0} } h(x_1,x_2)\\
  &e^{-X_0(V^{\vec{\lambda},\vec{x}})}  \Big(\lambda_1^{1+\alpha} \lambda_2^{1+\alpha} X_0(V_1^{\vec{\lambda},\vec{x}}) X_0(V_2^{\vec{\lambda},\vec{x}})-\lambda_1^{1+\alpha} \lambda_2^{1+\alpha} X_0(V_{1,2}^{\vec{\lambda},\vec{x}})\Big)dx_1dx_2.
\end{align*}
The result follows by an application of Dominated Convergence using Proposition \ref{p3.1}(i), Proposition \ref{p3.2}(i), \eqref{ae7.6}, \eqref{m10.0}, \eqref{eu12b2},  and the assumption on $h$. The case for $\widetilde{\cL}(\kappa)^{\eps}$ follows in a similar way. 
\end{proof}

\begin{corollary}\label{c1}
For any bounded Borel function $\phi$ on $\R^d$ and for any $k\geq 1$, we have
$\cL^\lambda(\phi \cdot 1_{\{ B_k\cap  S(X_0)^{>1/k}\}})$ converges in $L^2(\P_{X_0})$  as $\lambda \to \infty$ and 
$\widetilde{\cL}(\kappa)^\eps(\phi \cdot 1_{\{ B_k\cap  S(X_0)^{>1/k}\}})$ converges in $L^2(\P_{X_0})$  as $\eps \downarrow 0$.
\end{corollary}
\begin{proof}
For any bounded Borel function $\phi$ we let
\begin{align*}
h(x_1,x_2)=\phi(x_1) & 1_{\{ B_k\cap  S(X_0)^{>1/k}\}}(x_1)\cdot \phi(x_2) 1_{\{ B_k\cap  S(X_0)^{>1/k}\}}(x_2).
\end{align*}
Then the proof follows in a similar way to that of Corollary \ref{c0} by applying Theorem \ref{tl3} with  the above $h$.
\end{proof}

\begin{corollary}\label{c4}
For any bounded Borel function $\phi$ on $\R^d$ and for any $k\geq 1$, we have
$C_{\ref{p3.1}}(\kappa) \cL^\lambda (\phi \cdot 1_{\{ B_k\cap  S(X_0)^{>1/k}\}} )-K_{\ref{p3.1}}\widetilde \cL(\kappa)^\eps(\phi \cdot 1_{\{ B_k\cap  S(X_0)^{>1/k}\}} )$ converges to $0$ in $L^2(\P_{X_0})$ as $\lambda \to \infty$ and $\eps \downarrow 0$.
\end{corollary}
\begin{proof}
The proof is similar to that of Corollary \ref{c3} by using \eqref{m8.00}, \eqref{ae7.6}, \eqref{m8.7}, \eqref{m8.8}, \eqref{eu12b3}, Proposition \ref{p3.3} and Theorem \ref{tl3}.
\end{proof}

\subsection{Proofs of Theorems \ref{t0} and \ref{tl1}}\label{s5.2}

\begin{proposition}\label{p4.0}
For any $k\geq 1$ and any sequence $\eps_n \downarrow 0$, we have $\N_{0}$-a.e. or $\P_{X_0}$-a.s. that $\widetilde{\cL}(\kappa)^{\eps_n}(\mR)=0$ for all $\eps_n>0$ and
$\widetilde{\cL}(\kappa)^{\eps_n}(\mR^{>1/k})=0$ for all $0<\eps_n<1/k$.
\end{proposition}
\begin{proof}
First for any $\eps>0$, 
\begin{align*}
&\N_{0}\Big(\widetilde{\cL}(\kappa)^{\eps}(\mR)\Big)\leq \N_{0}\Big(\int \frac{X_{G_\eps^x}(1)}{\eps^{p}} \exp\Big(-\kappa \frac{X_{G_{\eps}^{x}}(1)}{\eps^2}\Big)1_{(X_{G_{\eps/2}^{x}}=0)}1_{(x\in \mR)} dx \Big)\nn\\
=&\int\N_{0}\Big(\frac{X_{G_\eps^x}(1)}{\eps^{p}} \exp\Big(-\kappa \frac{X_{G_{\eps}^{x}}(1)}{\eps^2}\Big)1_{(X_{G_{\eps/2}^{x}}=0)}1_{(x\in \mR)} \Big)dx\nn\\
=&\int \N_{0}\Big(\frac{X_{G_\eps^x}(1)}{\eps^{p}} \exp\Big(-\kappa \frac{X_{G_{\eps}^{x}}(1)}{\eps^2}\Big)1_{(X_{G_{\eps/2}^{x}}=0)}\P_{X_{G_{\eps/2}^{x}}}(x\in \mR) \Big)dx=0.
\end{align*}
where the first equality is by Fubini's theorem and the second equality uses Proposition \ref{pv0.2}(ii). Hence $\widetilde{\cL}(\kappa)^{\eps}(\mR)=0,$ $\N_{0}$-a.e.

Next for all $x\in \mR^{>1/k}$ and $0<\eps<1/k$, we have $\overline{B_\eps(x)}\subset \mR^c$ and \eqref{ea1.1} will then imply $X_{G_\eps^x}(1)=0$. 
Thus if $0<\eps<1/k$,
\begin{align*}
&\N_{0}\Big(\widetilde{\cL}(\kappa)^{\eps}\big(\mR^{>1/k}\big)\Big)\leq \N_{0}\Big(\int \frac{X_{G_\eps^x}(1)}{\eps^{p}}\exp\Big(-\kappa \frac{X_{G_{\eps}^{x}}(1)}{\eps^2}\Big) 1_{\overline{B_\eps(x)} \subset \mR^c} dx \Big)\nn\\
&=\int\N_{0}\Big(\frac{X_{G_\eps^x}(1)}{\eps^{p}} \exp\Big(-\kappa \frac{X_{G_{\eps}^{x}}(1)}{\eps^2}\Big)1_{\overline{B_\eps(x)} \subset \mR^c} \Big)dx=0.
\end{align*}
Take a countable union of null sets to see that $\N_{0}$-a.e.  $\widetilde{\cL}(\kappa)^{\eps_n}(\mR)=0$ for all $\eps_n>0$ and $\widetilde{\cL}(\kappa)^{\eps_n}(\mR^{>1/k})=0$ for all $0<\eps_n<1/k$ and so the proof for $\N_0$ is complete. The proof for $\P_{X_0}$ follows in a similar way. 
\end{proof}

\begin{proof}[Proof of Theorems \ref{t0} and \ref{tl1}]
 We first give the convergence of $\cL^\lambda$ to $\cL$ and $\widetilde{\cL}(\kappa)^\eps$ to $\widetilde{\cL}(\kappa)$ and then find some constant $c_{\ref{tl1}}(\kappa)>0$ so that 
 $\widetilde{\cL}(\kappa)= c_{\ref{tl1}}(\kappa) \cL$ a.s. Next we show that the support of $\widetilde{\cL}(\kappa)$ is contained in $\pmR$ and it follows that the support of $\cL$ will also be on $\pmR$, thus finishing both proofs of Theorem \ref{t0} and Theorem \ref{tl1}. Since the proof for the convergence of $\cL^\lambda$ and $\widetilde{\cL}(\kappa)^\eps$ are similar, we will only give the proof for the latter. 
 
 We first deal with $\N_0$. Let $\{\phi_m\}_{m=1}^\infty$ be a countable determining class for $M_F(\R^d)$ consisting of bounded, continuous functions and we take $\phi_1=1$. Consider 
\begin{align}\label{e6.4}
\cC=\{\psi_{m,k}: \psi_{m,k}=\phi_m  \chi_k, m\geq 1,k\geq 1\},
\end{align}
where $\chi_k$ is a continuous modification of $1_{(k^{-1}\leq |x|\leq k)}$ so that $\chi_k(x)=1$ for all $k^{-1}\leq |x|\leq k$ and $\chi_k(x)=0$ for all $|x|\leq (2k)^{-1} \text{ or } |x|\geq k+1$.
 Corollary \ref{c0} implies that for any $\psi_{m,k}\in \cC$, we have $\widetilde{\cL}(\kappa)^\eps(\psi_{m,k})$ converges in $L^2(\N_0)$ to some $\widetilde{l}(\psi_{m,k})$ in $L^2(\N_0)$ and by taking a subsequence  we get almost sure convergence. Define subsequences iteratively and take a diagonal subsequence $\eps_n\downarrow 0$ (we may assume for all $n\geq 1$ that $0<\eps_n<1$) to get 
\begin{align}\label{e1.5.1}
    \widetilde{\cL}(\kappa)^{\eps_n}(\psi_{m,k}) \to \widetilde{l}(\psi_{m,k}) \text{ as } \eps_n \downarrow 0, \text{ for all } m,k\geq 1, \N_0-a.e.
\end{align}
Fix $\omega$ outside a null set such that \eqref{e1.5.1} hold. Choose $m=1$ in \eqref{e1.5.1} to see that $\widetilde{\cL}(\kappa)^{\eps_n}(\chi_k) \to \widetilde{l}(\chi_k)$ for all $k\geq 1$. Note we have $\widetilde{l}(\chi_k)<\infty$ by the choice of $\omega$ and so $\N_0$-a.e. we have
\begin{align}\label{e0.1.0}
\sup_{\eps_n>0} \widetilde{\cL}(\kappa)^{\eps_n}(\{x: k^{-1}\leq |x|\leq k\})\leq \sup_{\eps_n>0} \widetilde{\cL}(\kappa)^{\eps_n}(\chi_k)<\infty, \forall k\geq 1.
\end{align}
The proof of Theorem 1.5 in \cite{HMP18} implies that $\N_0$-a.e. $L^x$ is positive for $x$ near $0$,
and hence we have $\N_0$-a.e. that $\{x: |x|\leq k^{-1}\}\subset \mR$ for $k\geq 1$ large. Proposition \ref{p4.0} will then imply $\N_0$-a.e. for $k\geq 1$ large,
\begin{align}\label{e0.3}
\widetilde{\cL}(\kappa)^{\eps_n}(\{x: |x|\leq k^{-1}\})\leq \widetilde{\cL}(\kappa)^{\eps_n}(\mR)=0 \text{ for all } \eps_n>0.
\end{align} 
On the other hand, we know that the range of SBM $X$ is compact $\N_0$-a.e. by \eqref{ea0.0} and hence by Proposition \ref{p4.0} we have $\N_0$-a.e. that
\begin{align}\label{e0.1}
 \text{ for } k\geq 1 \text{ large,} \sup_{\eps_n>0} \widetilde{\cL}(\kappa)^{\eps_n}(\{x: |x|\geq k\})\leq \sup_{\eps_n>0} \widetilde{\cL}(\kappa)^{\eps_n}(\mR^{>1})=0.
\end{align} 
Combining \eqref{e0.1.0}, \eqref{e0.3} and \eqref{e0.1}, we get
\begin{align}\label{e0.2}
\sup_{\eps_n>0} \widetilde{\cL}(\kappa)^{\eps_n}(1)<\infty, \N_0-a.e.
\end{align}
Note \eqref{e0.1} also implies the tightness of $\{\widetilde{\cL}(\kappa)^{\eps_n}\}$ and together with \eqref{e0.2}, we get the relative compactness of $\{\widetilde{\cL}(\kappa)^{\eps_n}\}$ by Prohorov's theorem (see, e.g., Theorem 7.8.7 of \cite{AD00}). By relative compactness of $\{\widetilde{\cL}(\kappa)^{\eps_n}\}$, any subsequence admits a further sequence along which the measures converge to some $\widetilde{\cL}(\kappa)$ in the weak topology. It remains to check all limit point coincide which is easy to see by \eqref{e1.5.1} since $\cC$ is a determining class on $M_F(\R^d).$ In conclusion, for any sequence $\eps_k\downarrow 0$, we can find a subsequence $\eps_{k_n} \downarrow 0$ such that  $\N_0$-a.e. $\widetilde{\cL}(\kappa)^{\eps_{k_n}} \to \widetilde{\cL}(\kappa)$, which easily implies that $\widetilde{\cL}(\kappa)^\eps \overset{P}{\rightarrow} \widetilde{\cL}(\kappa)$ under $\N_0$. The case for ${\cL}^\lambda \overset{P}{\rightarrow} {\cL}$ under $\N_0$ is similar. 


After establishing the existence of $\cL$ and $\widetilde{\cL}(\kappa)$, we continue to show that they differ only up to some constant. 
It is easy to check that for any  $\eps, \lambda>0$ and any $\psi_{m,k} \in \cC$,
 \begin{align}\label{e9.4.1}
 &\N_0\Big(\Big(K_{\ref{p3.1}}\widetilde{\cL}(\kappa)(\psi_{m,k})-C_{\ref{p3.1}}(\kappa){\cL}(\psi_{m,k})\Big)^2 \Big)\nn\\
 &\leq 4 \N_0\Big(\Big(K_{\ref{p3.1}}\widetilde{\cL}(\kappa)(\psi_{m,k})-K_{\ref{p3.1}}\widetilde{\cL}(\kappa)^{\eps}(\psi_{m,k})\Big)^2 \Big)\nn\\
 &+ 4\N_0\Big(\Big(K_{\ref{p3.1}}\widetilde{\cL}(\kappa)^{\eps}(\psi_{m,k})-C_{\ref{p3.1}}(\kappa) {\cL}^{\lambda}(\psi_{m,k})\Big)^2 \Big)\nn\\
 &+4\N_0\Big(\Big(C_{\ref{p3.1}}(\kappa){\cL}^{\lambda}(\psi_{m,k})-C_{\ref{p3.1}}(\kappa){\cL}(\psi_{m,k})\Big)^2 \Big).
\end{align}
By letting $\lambda \to \infty$ and $\eps \downarrow 0$, we conclude by Corollary \ref{c0} and Corollary \ref{c3} that each term on the right-hand side of \eqref{e9.4.1} converges to $0$ and hence
\[K_{\ref{p3.1}}\widetilde{\cL}(\kappa)(\psi_{m,k})=C_{\ref{p3.1}}(\kappa) {\cL}(\psi_{m,k}), \N_0-a.e.\]
 Take a countable union of null sets to conclude that $\N_0$-a.e. for all $m,k\geq 1$, we have $C_{\ref{p3.1}}(\kappa) {\cL}(\psi_{m,k})=K_{\ref{p3.1}}\widetilde{\cL}(\kappa)(\psi_{m,k})$  and so $C_{\ref{p3.1}}(\kappa){\cL}=K_{\ref{p3.1}}\widetilde{\cL}(\kappa)$. Let $c_{\ref{tl1}}(\kappa)=C_{\ref{p3.1}}(\kappa) K_{\ref{p3.1}}^{-1}$ to see that $\N_0$-a.e. we have $\widetilde{\cL}(\kappa)=c_{\ref{tl1}}(\kappa) \cL$.

Finally we will show that $\widetilde{\cL}(\kappa)$ (and hence $\cL$) is supported on $\pmR$. Let $\{\eps_n\}_{n\geq 1}$ be any sequence such that $\N_{0}$-a.e. $\widetilde{\cL}(\kappa)^{\eps_n}\to \widetilde{\cL}(\kappa)$. By Proposition \ref{p4.0} we can fix $\omega$ outside a null set such that $\widetilde{\cL}(\kappa)^{\eps_n}\to \widetilde{\cL}(\kappa)$ and \mbox{$\widetilde{\cL}(\kappa)^{\eps_n}(\mR)\to 0$} hold. It follows that
\begin{align}\label{ce1.1}
\widetilde{\cL}(\kappa)(\text{Int}(\mR))\leq \liminf_{\eps_n \downarrow 0} \widetilde{\cL}(\kappa)^{\eps_n}(\text{Int}(\mR))\leq \liminf_{\eps_n \downarrow 0} \widetilde{\cL}(\kappa)^{\eps_n}(\mR)=0,
\end{align}
where the first inequality is by $\widetilde{\cL}(\kappa)^{\eps_n}\to \widetilde{\cL}(\kappa)$. 

Next by Proposition \ref{p4.0} we can take a countable union of null sets and fix $\omega$ outside a null set such that $\widetilde{\cL}(\kappa)^{\eps_n}\to \widetilde{\cL}(\kappa)$ and \mbox{$ \widetilde{\cL}(\kappa)^{\eps_n}(\mR^{>1/k})\to 0$} holds for all $k\geq 1$. Then we have
\begin{align*}
\widetilde{\cL}(\kappa)(\mR^c)=&\widetilde{\cL}(\kappa)\Big(\bigcup_{k =1}^\infty \mR^{>1/k}\Big)\leq \sum_{k=1}^\infty \widetilde{\cL}(\kappa)(\mR^{>1/k})\\
\leq&  \sum_{k=1}^\infty  \liminf_{\eps_n \downarrow 0} \widetilde{\cL}(\kappa)^{\eps_n}(\mR^{>1/k})=0,
\end{align*}
where the second inequality is by $\widetilde{\cL}(\kappa)^{\eps_n}\to \widetilde{\cL}(\kappa)$. Therefore we conclude the support of $\widetilde{\cL}(\kappa)$ is on $\pmR$ under $\N_0$. 

Turning to the case under $\P_{\delta_0}$, the above arguments work in an exactly same way as $\N_0$ and so we omit the details.
\end{proof}

\subsection{On the moments of the boundary local time measure}

In view of Theorems \ref{tl2}, \ref{tl3} and Corollaries \ref{c0}, \ref{c1}, we can get the moment measure formulas for $\cL$ and $\widetilde{\cL}(\kappa)$ and finish the proof of Theorems \ref{p0.0} and \ref{p0.0_1}. 
\begin{proof}[Proof of Theorem \ref{p0.0}]
(a) Let $\lambda_n$ be the sequence from Theorem \ref{t0} such that  $\cL^{\lambda_n} \to \cL$, $\N_0$-a.e. For any bounded continuous function $\phi\geq 0$ and any $k\geq 1$, we have $\cL^{\lambda_n}(\phi \cdot \chi_k) \to \cL(\phi \cdot \chi_k)$, $\N_0$-a.e., where $\chi_k$ is as in \eqref{e6.4}.
Corollary \ref{c0} will then give that $\cL^{\lambda_n}(\phi \cdot \chi_k)$ converges in $L^2(\N_0)$ to  $\cL(\phi \cdot \chi_k)$. In particular, by working with the finite measure $\N_0(\cdot \cap \{\mR \cap G_{1/4k} \neq \emptyset\})$, we have $\cL^{\lambda_n}(\phi \cdot \chi_k)$ converges in $L^1(\N_0)$ to $\cL(\phi \cdot \chi_k)$ and so
\begin{align}\label{e6.5}
&\N_0\Big(\cL(\phi \cdot \chi_k)\Big)=\lim_{n\to \infty} \N_0\Big(\cL^{\lambda_n}(\phi \cdot \chi_k)\Big)\\\nn
 =&\lim_{n\to \infty} \int \phi(x)\chi_k(x) \N_0(\lambda_n^{1+\alpha} L^x e^{-\lambda_n L^x}) dx= c_{\ref{p12}}\int |x|^{-p}  \phi(x)\chi_k(x) dx,
\end{align}
where the second equality is by Fubini's theorem and in the last equality we have used Dominated Convergence with Proposition \ref{p1.1} and Proposition \ref{p12}.  Let $k\to \infty$ and apply a Monotone Class Theorem to extend \eqref{e6.5} to any Borel measurable function $\phi$ and the proof follows by $K_{\ref{p3.1}}=c_{\ref{p12}}$. 

(b) By \eqref{eu12b}, \eqref{e6.3} follows immediately from \eqref{e6.2}. For the proof of \eqref{e6.2}, we let $\lambda_n$ be the sequence such that  $\cL^{\lambda_n} \to \cL$, $\N_0$-a.e. For any bounded continuous function $h\geq 0$ and any $k\geq 1$, we have $\N_0$-a.e. that 
\begin{align}\label{m9.0}
\begin{cases}
(i) &\lim_{n\to \infty} \cL^{\lambda_n}(\chi_k) \to \cL(\chi_k).\\
 (ii) &\lim_{n\to \infty} \int h(x_1, x_2)\chi_k(x_1) \chi_k(x_2)  d\cL^{\lambda_n}(x_1) d\cL^{\lambda_n}(x_2) \\
 &\quad\quad = \int h(x_1, x_2) \chi_k(x_1) \chi_k(x_2)  d\cL(x_1) d\cL(x_2).
  \end{cases}
\end{align}
Note $h\leq \|h\|_\infty$ and so
\begin{align}\label{e6.10}
  \Big|\int h(x_1, x_2)\chi_k(x_1) \chi_k(x_2)  d\cL^{\lambda_n}(x_1) d\cL^{\lambda_n}(x_2)\Big| \leq \|h\|_\infty (\cL^{\lambda_n}(\chi_k))^2.
  \end{align}
  Use Corollary \ref{c0} and \eqref{m9.0}(i) to get $\cL^{\lambda_n}(\chi_k)$ converges in $L^2(\N_0)$ to $\cL(\chi_k)$ and thus we get $\N_0((\cL^{\lambda_n}(\chi_k))^2)$ converges to $\N_0((\cL(\chi_k))^2)$.  Use \eqref{m9.0}(i) again and work under the finite measure \mbox{$\N_0(\cdot \cap \{\mR \cap G_{1/4k}\neq \emptyset\})$} to get $\{(\cL^{\lambda_n}(\chi_k))^2, n\geq 1\}$ is uniformly integrable. By \eqref{e6.10}, the left-hand side term of \eqref{m9.0}(ii) is also uniformly integrable and hence we conclude
\begin{align}\label{e6.8}
    \N_0\Big(&\int h(x_1, x_2) \chi_k(x_1) \chi_k(x_2)  d\cL(x_1) d\cL(x_2)\Big)\nonumber \\
    =&\lim_{n\to \infty} \N_0\Big(\int h(x_1, x_2)\chi_k(x_1) \chi_k(x_2)  d\cL^{\lambda_n}(x_1) d\cL^{\lambda_n}(x_2)\Big)\nonumber \\
    =&K_{\ref{p3.1}}^2 \int h(x_1, x_2)\chi_k(x_1) \chi_k(x_2)   (-U_{1,2}^{\vec{\infty},\vec{x}}(0)) dx_1 dx_2,
\end{align}
the last by Theorem \ref{tl2}. Let $k\to \infty$ and apply a Monotone Class Theorem to extend \eqref{e6.8} to any Borel measurable function.
\end{proof}
\begin{proof}[Proof of Theorem \ref{p0.0_1}]
 The proof of \eqref{ae8.4} and \eqref{ae8.5} follows in a similar way to the above proof of Theorem \ref{p0.0} by using Corollary \ref{ca1.3}, Theorem \ref{tl3}, Corollary \ref{c1}. \eqref{ae8.6} follows immediately from \eqref{ae8.5}, \eqref{eu12b} and the definitions of $U_i^{\vec{\infty},\vec{x}}$ from \eqref{eu1u2}.
\end{proof}

%

\section{Exit measures and zero-one law}\label{s6}

In this section we will give the proof of Theorem \ref{tl1.1}. Our approach to Theorem \ref{tl1.1} is similar to the proof of Theorem 1.2 in \cite{HMP18}; we utilize exit measures, which will be easy consequences of the following two results. The first result is proved below.

\begin{proposition}\label{p2.2}
 Let $x_1\in\R^d$ and $r_0>0$ satisfy $B_{2r_0}(x_1)\subset S(X_0)^c$.  If $0<r_1<r_0$, then $\N_{X_0}$-a.e.
\begin{equation*}
\begin{cases}
&X_{G_{r_1}^{x_1}}(1)=0\text{ and }X_{G_{r_0}^{x_1}}(1)>0\text{ imply }\\
&\cL(B_{r}(x_1))>0 \text{ for every } r>r_1\ s.t.\ X_{G_{r}^{x_1}}(1)>0.
\end{cases}
\end{equation*}
\end{proposition}

\begin{corollary}\label{p2.3}
 Let $x_1\in\R^d$ and $r_0>0$ satisfy $B_{2r_0}(x_1)\subset S(X_0)^c$.  If $0<r_1<r_0$, then $\P_{X_0}$-a.s.
\begin{equation*}
X_{G_{r_1}^{x_1}}(1)=0\text{ and }X_{G_{r_0}^{x_1}}(1)>0\text{ imply }\cL(B_{r_0}(x_1))>0.
\end{equation*}
\end{corollary}
 \begin{proof}
  It follows in a similar way to the proof of Proposition 1.6 assuming Proposition 1.7 in \cite{HMP18} by replacing $\text{dim}(\pmR\cap B_r)\ge d_f$ with \mbox{$\cL(B_r)>0$}.
\end{proof}

\begin{proof}[Proof of Theorem \ref{tl1.1}]
By using Proposition \ref{p2.2} and Corollary \ref{p2.3}, the proof of \eqref{er1.3} follows in a same way as the proof of Theorem 1.2 of \cite{HMP18}. \eqref{ae1.5} is immediate from \eqref{er1.3}. To see that with $\P_{X_0}$-probability one, $\text{Supp}(\cL)=S(X_0)^c\cap \pmR$, we pick any $x\in S(X_0)^c\cap \pmR$. There is some $\eps>0$ so that $B(x,r) \subset S(X_0)^c$ for all $0<r<\eps$  and $B(x,\eps)\cap \pmR\neq \emptyset$. Apply \eqref{er1.3} with $U=B(x,r)$ to see that $\cL(B(x,r))>0$ for all $0<r<\eps$ and so conclude $x\in \text{Supp}(\cL)$, giving $S(X_0)^c\cap \pmR \subset \text{Supp}(\cL)$. Together with Theorem \ref{t0.0.1} we have $\text{Supp}(\cL)=S(X_0)^c\cap \pmR$, $\P_{X_0}$-a.s. and the proof is complete.
\end{proof}
Now it remains to prove Proposition \ref{p2.2}. We first state a result that plays the role of Lemma 5.4 in \cite{HMP18}.


\begin{lemma}\label{l7.1}
There is a constant $q_{\ref{l7.1}}>0$ so that if $X_0'\in M_F(\R^d)$ is supported on
$\{|x|=r\}$ and $\delta=X_0'(1)$ satisfies $0<\delta\le r^2$, then
\begin{equation*}
\P_{X_0'}\Big(\cL\Big(B\Big(0,r-\frac{\sqrt{\delta}}{2}\Big)\Big)>0\Big)\ge q_{\ref{l7.1}}.
\end{equation*}
\end{lemma}
\begin{proof} Define $X_0^{(\delta)}(A)=\delta^{-1}X'_0(\sqrt\delta A)$,
so that $X_0^{(\delta)}$
 is supported on $\{|x|=r/\sqrt\delta\}$ and has total mass one.  
By scaling properties of SBM, we may
conclude that 
\begin{align}\label{ae9.0}
&\P_{X_0'}\Big(\cL\Big(B\Big(0,r-\frac{\sqrt{\delta}}{2}\Big)\Big)>0\Big)=\P_{X_0^{(\delta)}}\Big(\cL\Big(B\Big(0,\frac{r}{\sqrt{\delta}}-\frac{1}{2}\Big)\Big)>0\Big).
\end{align}
Now work in our standard set-up for SBM with initial law $X_0^{(\delta)}$ so that  \break
$X_t=\sum_{j\in J}X_t(W_j)=\int X_t(W)\Xi(dW)$ for all $t>0$, where $\Xi$ is a Poisson point process with intensity $\N_{X_0^{(\delta)}}$.
For $r\ge \sqrt\delta$ define
\begin{align*}
&\tau_\rho(W_j)=\inf\{t\ge 0:|\hat W_j(t)|\le \rho\},\\
&U_\rho(W_j)=\inf\{t\ge 0:|\hat W_j(t)-\hat W_j(0)|\ge\rho\},\\
&\text{and } N_1=\sum_{j\in J}1(\tau_{(r/\sqrt\delta)-(1/2)}(W_j)<\infty):=\#(I_{r,\delta}).
\end{align*}
Here as usual $\inf\emptyset=\infty$.  Then $N_1$ is Poisson with mean
\begin{align}\label{ae9.1}
m_{r,\delta}:=\N_{X_0^{(\delta)}}(\tau_{(r/\sqrt\delta)-(1/2)}<\infty)
&\le \N_{X_0^{(\delta)}}(U_{1/2}(W)<\infty)\\
\nn&=\N_{0}(U_{1/2}(W)<\infty):=\overline m<\infty,
\end{align}
where $X_0^{(\delta)}(1)=1$ and translation invariance are used in the equality, and the finiteness of $\bar m$ follows from Theorem~1 of \cite{Isc88}. 
We may assume (by additional randomization) that conditional on $I_{r,\delta}$, $\{W_j:j\in I_{r,\delta}\}$ are iid with law $\N_{X_0^{(\delta)}}(W\in\cdot|\tau_{(r/\sqrt\delta)-(1/2)}<\infty)$. Therefore the right-hand side of \eqref{ae9.0} is at least
\begin{align}\label{ae9.2}
\nn \P_{X_0^{(\delta)}}&(N_1=1)\N_{X_0^{(\delta)}}\Big(\cL\Big(B\Big(0,\frac{r}{\sqrt{\delta}}-\frac{1}{2}\Big)\Big)>0\Big|\tau_{\frac{r}{\sqrt\delta}-\frac{1}{2}}<\infty\Big)\\
&=\frac{m_{r,\delta}e^{-m_{r,\delta}}}{m_{r,\delta}}\N_{x_0}\Big(\cL\Big(B\Big(0,\frac{r}{\sqrt{\delta}}-\frac{1}{2}\Big)\Big)>0\Big),
\end{align}
where $x_0=(\frac{r}{\sqrt\delta})e_1$ and $e_1$ is the first unit basis vector. We also have used the facts that spherical symmetry shows we could have taken any $x_0$ on
the sphere of radius $r/\sqrt\delta$ and $\cL(B(0,\frac{r}{\sqrt{\delta}}-\frac{1}{2}))=0$ if $\tau_{\frac{r}{\sqrt\delta}-\frac{1}{2}}=\infty$ by the fact that $\text{Supp}(\cL)=\pmR, \N_{x_0}$-a.e. from Corollary \ref{c1.6} and translation variance.  Now again use translation invariance and spherical symmetry to see that the right side of \eqref{ae9.2} equals
\begin{align}\label{ae9.3}
\nn e^{-m_{r,\delta}}\N_{0}\Big(\cL\Big(B\Big(x_0,|x_0|-\frac{1}{2}\Big)\Big)>0\Big)&\geq e^{-\overline m}\N_{0}\Big(\cL(B(e_1,1/2))>0\Big)\\
&\ge e^{-\overline m} \frac{\Big(\N_{0}\Big(\cL(B(e_1,1/2))\Big)\Big)^2}{\N_{0}\Big(\Big(\cL(B(e_1,1/2))\Big)^2\Big)},
\end{align}
where the first inequality follows by \eqref{ae9.1} and $B(e_1,1/2) \subset B(x_0,|x_0|-\frac{1}{2})$ since $x_0=|x_0|e_1$ and $|x_0|\ge1$, and the last follows by the second moment method.
Now apply Theorem \ref{p0.0} (a) with $\phi(x)=1_{B(e_1,1/2)}(x)$ and Theorem \ref{p0.0} (b) with $h(x_1,x_2)=1_{B(e_1,1/2)}(x_1) 1_{B(e_1,1/2)}(x_2)$ to get
\begin{align*}
\N_{0}\Big(\cL(B(e_1,1/2))\Big)=K_{\ref{p3.1}} \int_{|x-e_1|<1/2} |x|^{-p} dx\geq K_{\ref{p3.1}} (\frac{3}{2})^{-p}|B(0,1/2)|>0.
\end{align*}
and
\begin{align*}
&\N_{0}\Big(\Big(\cL(B(e_1,1/2))\Big)^2\Big)\\
\leq& K_{\ref{p3.1}}^2\int_{|x_1-e_1|,|x_2-e_1|<1/2} c_{\ref{p3.2}}(2^{p}+2^{p})|x_1-x_2|^{2-p} dx_1dx_2 <\infty.
\end{align*}
Thus we have shown that the right-hand side of \eqref{ae9.3} has some lower bound $e^{-\overline m} c>0$ for some universal constant $c>0$, and so have proved the lemma with $q_{\ref{l7.1}}=e^{-\overline m}c$ . 
\end{proof}
   
Now we proceed to the proof of Proposition \ref{p2.2}. Using the setting from Proposition \ref{p2.2}, by translation invariance we may assume $x_1=0$ and fix $r_0>0$ such that
\begin{align}\label{ex3.1}
    B_{2r_0}\subset S(X_0)^c.
\end{align}

\no{\bf Notation.} We define $Y_r(\cdot)=X_{G_{r_0-r}}(\cdot)$ and $\cE_r=\cE_{G_{r_0-r}}\vee\{\N_{X_0}-\text{null sets}\}$ for $0\le r<r_0$.
It is not hard to show that $\cE_r$ is non-decreasing in $r$ (see Section 6 of \cite{HMP18}). Intuitively $\cE_r$ is the $\sigma$-field generated by the excursions of $W$ in $G_{r_0-r}$.
By Proposition 2.3 of \cite{Leg95}, $Y$ is $(\cE_r)$-adapted. Let $\cE^+_r=\cE_{r+}$ denote the associated right-continuous filtration. Note Proposition~6.2(b) in \cite{HMP18} gives a cadlag version of $Y_r(1)$ which has no negative jumps and is an $(\cE^+_r)$-supermartingale. In what follows we always work with this cadlag version of $Y_r(1)$.

In addition to $\N_{X_0}$, we will also work under the probability $Q_{X_0}(\cdot)=\N_{X_0}(\cdot|Y_0(1)>0)$, where \eqref{ex3.1} ensures that $\N_{X_0}(Y_0(1)>0)<\infty$.  Note that  
\begin{equation}\label{ae9.4}
\text{for any r.v. }Z\ge 0,\text{ and any }r\ge 0,\ Q_{X_0}(Z|\cE_r)=\N_{X_0}(Z|\cE_r)\ Q_{X_0}\text{-a.s.}
\end{equation}
because $\{Y_0(1)>0\}\in \cE_0$.  When conditioning on $\cE_r$ under $Q_{X_0}$, we are adding the slightly larger class of $Q_{X_0}$-null sets to $\cE_r$, but will not record this distinction in our notation. We write $Q_{x_0}$ for $Q_{\delta_{x_0}}$ as usual.

Let $W$ denote a generic Brownian snake under $\N_{X_0}$ or $Q_{X_0}$ with the associated ``tip process'' $\hat W(t)$ and excursion length $\sigma$.
Define
 \[T_{0}(W)=\inf\{r\in [0,r_0): Y_r(1)=0\} \in [0,r_0], \text{ where } \inf \emptyset=r_0,\] and 
 \begin{align*}
 \hat T_{0}(W)=\inf\{|\hat W(t)|: 0\leq t\leq \sigma\}=\inf\{|x|: x\in \mR\},
 \end{align*}
 where the last equality holds $\N_{X_0}$ by \eqref{ea0.0}. Clearly we have $Q_{X_0}(\cdot)=\N_{X_0}(\cdot|T_0>0)$. By Lemma 7.1 of \cite{HMP18}, we have 
  \begin{align}\label{be4.9.3}
 \N_{X_0}-a.e.\ \{T_0>0\}=\{\hat T_0<r_0\}, \text{ and on this set } \hat T_0=r_0-T_0.
  \end{align}
 Define a sequence of $(\cE^+_r)$-stopping times by
\[T_{n^{-1}}=\inf\{r\in [0,r_0): Y_r(1)\le 1/n\}\quad(\inf\emptyset = r_0).\] 
Then 
\begin{equation}\label{ae9.5}
\text{on $\{0<T_0\}$ (and so $Q_{X_0}$-a.s.) $T_{n^{-1}}\uparrow T_0$ and $T_{n^{-1}}<T_0$, }
\end{equation}
where the last inequality holds since $Y_r(1)$ has no negative jumps.  So under $Q_{X_0}$, $T_0$ is a predictable stopping time which is announced by $\{T_{n^{-1}}\}$ and so (see (12.9)(ii) in Chapter VI of \cite{RW94})
\begin{equation*}
\cE^+_{T_0-}=\vee_n\cE^+_{T_n}.
\end{equation*}
\begin{lemma}\label{ae9.6}
   If $X_0=\delta_{x_0}$ where $|x_0|\geq 2r_0$, then
   $\cL(B_{r_0})\in \cE_{T_0^{-}}^+.$
\end{lemma}   
\begin{proof}
Note Theorem \ref{t0} implies there is some $\lambda_n \to \infty$ such that $\cL^{\lambda_n}\to \cL$, $\N_{x_0}$-a.e. by translation invariance. On the other hand, by Theorem \ref{p0.0} we have $\N_{x_0}(\cL(\partial B_{r_0}))=0$ and so $\N_{x_0}$-a.e., $\cL(B_{r_0})=\lim_{n \to \infty}\cL^{\lambda_n}(B_{r_0})$. As is shown in the proof of Lemma 7.3 in \cite{HMP18}, we have $\psi(W)\in \cE_{T_0^{-}}^+$ for any Borel map $\psi$ on $C(\R^+,\cW)$. Then it follows that $L^x \in \cE_{T_0^{-}}^+$ for any $x\in B_{r_0}$ and so $\cL^{\lambda_n}(B_{r_0})\in \cE_{T_0^{-}}^+$ for any $\lambda_n>0$, thus proving $\cL(B_{r_0})\in \cE_{T_0^{-}}^+$.
\end{proof}

\begin{proof}[Proof of Proposition \ref{p2.2}] 
Clearly it suffices to fix $x_0\in S(X_0)$ and prove the result with $\N_{x_0}$  in place of $\N_{X_0}$. By translation invariance we may assume $x_1=0$, and so $|x_0|\ge 2r_0$.  
Fix $0<r_1<r_0$.   Assume $0\le r<r_0$ and $n\in\N$ is large enough so that $r+n^{-1}<r_0$.  Recall that conditional expectations with respect to $\cE_r$, under $\N_{x_0}$ and $Q_{x_0}$, agree $Q_{x_0}$-a.s.  Therefore up to $Q_{x_0}$-null sets, on 
$\{4n^{-2}\le Y_r(1)\le (r_0-r)^2\} (\in \cE_r)$ we  have
\begin{align}\label{e0.0.2} 
Q_{x_0}\Big(\cL(B_{r_0})>0\Big|\cE_r\Big)\geq& Q_{x_0}\Big(\cL(\bar{B}_{r_0-r-n^{-1}})>0\Big|\cE_r\Big)\nn\\
\geq& Q_{x_0}\Big(\limsup_{k\to \infty}\cL^{\lambda_k}(\bar{B}_{r_0-r-n^{-1}})>0\Big|\cE_r\Big)\nn\\
=&\lim_{m\to \infty} Q_{x_0}\Big(\limsup_{k\to \infty}\cL^{\lambda_k}(\bar{B}_{r_0-r-n^{-1}})>m^{-1}\Big|\cE_r\Big),
\end{align}
where the second inequality is by $\cL^{\lambda_k} \to \cL$ in $M_F$ with the $\{\lambda_k\}$ from Theorem \ref{t0}. The last equality uses monotone convergence. For each $m\geq 1$ we have
\begin{align} \label{ae10.0}
&Q_{x_0}\Big(\limsup_{k\to \infty}\cL^{\lambda_k}(\bar{B}_{r_0-r-n^{-1}})>m^{-1}\Big|\cE_r\Big)\nn\\
&\geq\liminf_{k \to \infty} Q_{x_0}\Big(\cL^{\lambda_k}(B_{r_0-r-n^{-1}})>m^{-1}\Big|\cE_r\Big)\nn\\
&=\liminf_{k \to \infty}\P_{Y_r}\Big(\cL^{\lambda_k}(B_{r_0-r-n^{-1}})>m^{-1}\Big)\nonumber\\
&\geq \P_{Y_r}\Bigl(\liminf_{k \to \infty}\cL^{\lambda_k}(B_{r_0-r-n^{-1}})>m^{-1}\Bigr)\quad(\text{by Fatou's Lemma})\nonumber\\
&\geq \P_{Y_r}\Bigl(\cL(B_{r_0-r-n^{-1}})>m^{-1}\Bigr)\quad(\text{by }\cL^{\lambda_{k}}|_{\bar{B}_{r_0-r-n^{-1}}} \to \cL|_{\bar{B}_{r_0-r-n^{-1}}})
\end{align}
where we have used Proposition~\ref{pv0.2}(iii) in the equality and the last inequality is by Theorem \ref{t0.0.1} and by replacing $\{\lambda_k\}$ with a  further subsequence which is still denoted by $\{\lambda_k\}$. Combining \eqref{e0.0.2} and \eqref{ae10.0}, we get up to $Q_{x_0}$-null sets, on 
$\{4n^{-2}\le Y_r(1)\le (r_0-r)^2\}$ (which is in $\cE_r$), we have
\begin{align} \label{e0.0.3} 
&Q_{x_0}\Big(\cL(B_{r_0})>0\Big|\cE_r\Big)\ge \liminf_{m\to \infty} \P_{Y_r}\Bigl(\cL(B_{r_0-r-n^{-1}})>m^{-1}\Bigr)\\
=&\P_{Y_r}\Bigl(\cL(B_{r_0-r-n^{-1}})>0\Bigr)\geq \P_{Y_r}\Bigl(\cL(B_{r_0-r-(\sqrt{Y_r(1)}/2)})>0\Bigr)\ge q_{\ref{l7.1}},\nn
\end{align}
where Lemma~\ref{l7.1} and the assumed upper bounds on $Y_r(1)$ are used in the last inequality, and the assumed lower bound on $Y_r(1)$ is used in the next to last inequality. 
Let $n\to\infty$ and take limits from above in $r\in \Q_+$ (recall $Y_r(1)$ is cadlag) to conclude that $Q_{x_0}$-a.s. $\forall r\in\Q\cap(0,r_0)$,
\begin{equation}\label{ae9.7}
M_r:=Q_{x_0}(\cL(B_{r_0})>0|\cE^+_{r})\ge q_{\ref{l7.1}}\text{ on }\{0<Y_r(1)<(r_0-r)^2\}.
\end{equation}
Here $M_r$ is a cadlag version of the bounded martingale on the left-hand side.  Using right-continuity one can strengthen \eqref{ae9.7} to $Q_{x_0}$-a.s. $\forall r\in(0,r_0)$,
\begin{equation}\label{ae9.8}
M_r=Q_{x_0}(\cL(B_{r_0})>0|\cE^+_{r})\ge q_{\ref{l7.1}}\text{ on }\{0<Y_r(1)<(r_0-r)^2\}.
\end{equation}
On $\{0<T_0\le r_0-r_1\}$, by \eqref{ae9.5} and the lack of negative jumps for $Y_r(1)$, we have  $Q_{x_0}$-a.s. that
\begin{equation}\label{ae9.9}
\text{for $n$ large, }T_{n^{-1}}\in(0,r_0-r_1)\text{ and }Y_{T_{n^{-1}}}(1)=n^{-1}<(r_0-T_{1/n})^2.
\end{equation}
By Corollary (17.10) in Chapter VI of \cite{RW94}, \eqref{ae9.8}, and \eqref{ae9.9},  we have $Q_{x_0}$-a.s. on
$\{0<T_0\le r_0-r_1\}\in \cE^+_{T_0-}$,
\begin{equation}\label{ae9.10}
Q_{x_0}(\cL(B_{r_0})>0|\cE^+_{T_0-})=\lim_{n\to\infty}M(T_{n^{-1}})\ge q_{\ref{l7.1}}.
\end{equation}
Multiplying the above by $1(\{0<T_0\le r_0-r_1\})$, we see from Lemma~\ref{ae9.6} that $Q_{x_0}$-a.s.,
\[1(\{\cL(B_{r_0})>0\}\cap\{0<T_0\le r_0-r_1\})\ge q_{\ref{l7.1}}1(\{0<T_0\le r_0-r_1\}),\]
and therefore by \eqref{be4.9.3},
\[  r_1\leq \hat T_0<r_0 \text{ implies } \cL(B_{r_0})>0\quad Q_{x_0}-\text{a.s.}\]
This remains true if we replace $r_0$ by any $r\in(r_1,r_0]$ since we still have $B_{2r}\subset S(X_0)^c$.  Therefore we may fix $\omega$ outside a $Q_{x_0}$-null set so that for any $r\in(r_1,r_0]\cap\Q$, $r_1\leq \hat T_0<r$ implies $\cL(B_{r})>0$. By monotonicity of the conclusion in $r$ this means that $\{r_1\leq \hat T_0<r_0\}$ implies $\cL(B_{r})>0$ for all $r>\hat T_0$. This gives Proposition~\ref{p2.2} under $Q_{x_0}$.  The result under $\N_{x_0}$ is now immediate from the definition of $Q_{x_0}$, and $\{Y_0(1)>0\}=\{ \hat T_0<r_0\}$ $\N_{x_0}$-a.e. by \eqref{be4.9.3}.
\end{proof}

\section{Change of Measure}\label{s7}
Before turning to the proof of Propositions \ref{p3.1}, \ref{p3.2} and \ref{p3.3}, we state a result on the change of measure that plays a central role in the proof. This result is a generalization of Proposition \ref{p20.1} where only radially symmetric functions are considered. We implement the ideas there and prove stronger results to deal with non-radial functions.

Let $Y=(Y_s, s\geq 0)$ denote the coordinate variables on  $C([0,\infty), \R^d)$ and set $(\cY_t)$ to be the right continuous filtration generated by $Y$. Under the law $P_x$ (Wiener measure), $Y$ is a standard $d$-dimensional Brownian motion starting from $x$. Recall $\mu, \nu$ as in \eqref{ev1.5} and recall $\Hat{P}_{x}^{(2-2\nu)}$ is the law under which, $Y$ is the unique solution of
  \begin{align}\label{eb7.3}
       \begin{cases}
        &Y_t=x+\Hat{B}_t+\int_0^t (-\nu-\mu)\frac{Y_s}{|Y_s|^2}ds, \quad t< \tau_0,\\
        &Y_t=0,\quad t\geq \tau_0,
        \end{cases}
  \end{align}
  where $\tau_\eps=\tau_\eps^{Y}=\inf\{t\geq 0: |Y_t|\leq \eps\}$ and $\Hat{B}$ is a standard $d$-dimensional Brownian motion under $\Hat{P}_{x}^{(2-2\nu)}$.  The upper index $2-2\nu<0$ on $\Hat{P}_{x}^{(2-2\nu)}$ is to remind us that under $\Hat{P}_{x}^{(2-2\nu)}$, the radial process $\{|Y_s|, s\geq 0\}$, as we will show later,  is a $(2-2\nu)$-dimensional Bessel process stopped at $0$. Now we proceed to the key proposition for proving the convergence of the second moments.

\begin{proposition}\label{p8.1}
Let $x\in \R^d-\{0\}$ and $0<\eps<|x|$. If $\Phi_{t} \geq 0$ is $\cY_{t}$-adapted, then for any Borel measurable function $g: \R^d \to \R$ such that $P_x$-a.s. $\int_0^{\tau_{\eps}} |g(Y_s)| ds<\infty$, we have
\begin{align}
E_{x}\Big(&1(\tau_\eps<\infty) \Phi_{\tau_\eps} \exp\big(-\int_0^{\tau_{\eps}} g(Y_s) ds\big)\Big)\nn \\
&=\eps^p |x|^{-p}\Hat{E}_{x}^{(2-2\nu)}\Big( \Phi_{\tau_\eps} \exp\big(-\int_0^{\tau_{\eps}} (g(Y_s)-V^\infty(Y_s)) ds\big)\Big).
\end{align}
\end{proposition}

\begin{proof}
By monotone convergence theorem we have
\begin{align}\label{ed1.0.1}
I:=&E_{x}\Big(1(\tau_\eps<\infty) \Phi_{\tau_\eps} \exp\big(-\int_0^{\tau_{\eps}} g(Y_s) ds\big)\Big)\nn \\
=&\lim_{t\to \infty}E_{x}\Big(1(\tau_\eps\leq t) \Phi_{\tau_\eps} \exp\big(-\int_0^{\tau_{\eps}} g(Y_s) ds\big)\Big)\nn\\
=&\lim_{t\to \infty}E_{x}\Big(1(\tau_\eps\leq \tau_\eps\wedge t) \Phi_{\tau_\eps \wedge t} \exp\big(-\int_0^{\tau_{\eps}\wedge t} g(Y_s) ds\big)\Big)
\end{align}
Use Ito's lemma to see that under $P_x$,
\begin{align}\label{e2.1}
   \log |Y_{\tau_{\eps} \wedge t}|=\log |Y_0|+\int_0^{\tau_{\eps} \wedge t} \frac{Y_s}{|Y_s|^2} \cdot dY_s+\frac{1}{2}\int_0^{\tau_{\eps} \wedge t} \frac{d-2}{|Y_s|^2} ds, \forall t\geq 0.
 \end{align}
  Recall $\mu, \nu$ as in \eqref{ev1.5} and consider 
\begin{align}\label{e2.2}
  M_\eps(t)=\exp\Big(\int_0^{t\wedge \tau_\eps} (\nu-\mu)\frac{Y_s}{|Y_s|^2}\cdot dY_s-\frac{1}{2}\int_0^{t\wedge \tau_\eps} \frac{(\nu-\mu)^2}{|Y_s|^2}  ds \Big).
 \end{align}
  As one can easily check, $M_\eps$ is a martingale under $P_x$. Moreover by using \eqref{e2.1} we can get
\begin{align}\label{e2.3}
  M_\eps(t)=\frac{|Y_{\tau_{\eps} \wedge t}|^{\nu-\mu}}{|Y_0|^{\nu-\mu}}\exp\Big(-\int_0^{\tau_{\eps} \wedge t} \frac{2(4-d)}{|Y_s|^2} ds \Big).
 \end{align}
An application of Girsanov's theorem (see, e.g., Chapter IV.4 of \cite{IW79}) implies there is a unique probability $\widetilde{P}_{\eps, x}^{(2+2\nu)}$ on $C([0,\infty), \R^d)$ so that for any $t\geq 0$,
  \begin{align}\label{e2.4}
  d\widetilde{P}_{\eps, x}^{(2+2\nu)}\Big|_{\cY_{t}}=\frac{|Y_{{\tau_{\eps} \wedge t}}|^{\nu-\mu}}{|x|^{\nu-\mu}}\exp\Big(-\int_0^{{\tau_{\eps} \wedge t}} \frac{2(4-d)}{|Y_s|^2} ds \Big) dP_x\Big|_{\cY_{t}},
   \end{align}
   and under $\widetilde{P}_{\eps, x}^{(2+2\nu)}$, $Y$ is the unique solution of
  \begin{align}
       Y_t=x+\widetilde{B}_{t}+\int_0^{\tau_{\eps} \wedge t} (\nu-\mu)\frac{Y_s}{|Y_s|^2}ds,
  \end{align}
  (so the drift is stopped when $Y$ hits the ball $\overline{B(0,\eps)}$). Here $\widetilde{B}$ is a standard $d$-dimensional Brownian motion with respect to $\widetilde{P}_{\eps, x}^{(2+2\nu)}$. 
   The upper index $2+2\nu$ on $\widetilde{P}^{(2+2\nu)}_{\eps,x}$ is to indicate that the radial process $\{|Y_{s\wedge \tau_\eps}|, s\geq 0\}$ is a $(2\nu+2)$-dimensional Bessel process stopped when it hits $\eps>0$:
 \begin{align}\label{ev2.5}
  |Y_{\tau_{\eps} \wedge t}|^2&=|x|^2+\int_0^{\tau_{\eps} \wedge t} 2Y_s \cdot (d\widetilde{B}_s+(\nu-\mu)\frac{Y_s}{|Y_s|^2}ds)+d(\tau_{\eps} \wedge t)\nonumber\\
  &=|x|^2+\int_0^{\tau_{\eps} \wedge t}2|Y_s| \sum_{i=1}^d \frac{Y_s^i}{|Y_s|} d\widetilde{B}_s^i +(2\nu+2)(\tau_{\eps} \wedge t)\nonumber\\
  &=|x|^2+\int_0^{\tau_{\eps} \wedge t} 2|Y_s| d\widetilde{\beta}_s+(2\nu+2)(\tau_{\eps} \wedge t),
  \end{align} 
 where the last follows since $\widetilde{\beta}_t=\sum_{i=1}^d \frac{Y_t^i}{|Y_t|} \widetilde{B}_t^i$ is a one-dimensional Brownian motion under $\widetilde{P}_{\eps,x}^{(2+2\nu)}$. Therefore $\{|Y_{s\wedge \tau_\eps}|^2, s\geq 0\}$ satisfies the SDE of a stopped square Bessel process of dimension $2+2\nu$ and so $\{|Y_{s\wedge \tau_\eps}|, s\geq 0\}$ is a stopped $(2+2\nu)$-dimensional Bessel process (see Chp. XI of \cite{RY94} for the definition of square Bessel process and its connection with Bessel process). It follows that
 \begin{align}\label{ec1.0.4}
  \widetilde{P}_{\eps, x}^{(2+2\nu)}(\tau_{\eps}<\infty)=\frac{\eps^{2\nu}}{|x|^{2\nu}}.
  \end{align}
  Now apply \eqref{e2.4} to see that \eqref{ed1.0.1} becomes
   \begin{align}\label{ed1.0.2}
I&=\lim_{t\to \infty} \widetilde{E}_{\eps, x}^{(2+2\nu)}\Big(1_{(\tau_\eps\leq \tau_\eps\wedge t)} \Phi_{\tau_\eps \wedge t} \exp\Big(-\int_0^{\tau_{\eps}\wedge t } (g(Y_s)-V^\infty(Y_s)) ds\Big)\frac{|Y_{\tau_{\eps} \wedge t}|^{\mu-\nu}}{|x|^{\mu-\nu}}\Big)\\
&=\lim_{t\to \infty} \widetilde{E}_{\eps, x}^{(2+2\nu)}\Big(1(\tau_\eps\leq t) \Phi_{\tau_\eps} \exp\Big(-\int_0^{\tau_{\eps}} (g(Y_s)-V^\infty(Y_s)) ds\Big)\frac{|Y_{\tau_{\eps}}|^{\mu-\nu}}{|x|^{\mu-\nu}}\Big)\nn\\
  &=\frac{\eps^{\mu-\nu}}{|x|^{\mu-\nu}}\widetilde{E}_{\eps, x}^{(2+2\nu)}\Big(1(\tau_\eps<\infty) \Phi_{\tau_\eps} \exp\Big(-\int_0^{\tau_{\eps}} (g(Y_s)-V^\infty(Y_s)) ds\Big)\Big)\nn\\
    &=\frac{\eps^{p}}{|x|^{p}}\widetilde{E}_{\eps, x}^{(2+2\nu)}\Big(\Phi_{\tau_\eps} \exp\Big(-\int_0^{\tau_{\eps}} (g(Y_s)-V^\infty(Y_s)) ds\Big)\Big|\tau_\eps<\infty\Big),\nn
  \end{align}
where we have used monotone convergence in the next to last equality and the last equality follows from \eqref{ec1.0.4} and $p=\mu+\nu$. 
 
   We interrupt the proof of the proposition for another auxiliary result.
\begin{lemma}\label{l1.5}
 For any $\eps>0$ and $|x|>\eps$, we have the law of $\{Y_{s\wedge \tau_{\eps}}, s \geq 0\}$ conditioning on $\{\tau_{\eps}<\infty\}$ under $\widetilde{P}^{(2+2\nu)}_{\eps, x}$ is equal to the law of $\{Y_{s\wedge \tau_{\eps}}, s \geq 0\}$ under $\Hat{P}^{(2-2\nu)}_x$ defined as in \eqref{eb7.3}.
  \end{lemma}
  \begin{proof}
For any $0<t_1<\cdots<t_n$ and any bounded Borel functions $\phi_i: \R^d \to \R, 1\leq i\leq n$, we use \eqref{ec1.0.4} to get
 \begin{align}\label{e2.7}
 J&:=\widetilde{E}^{(2+2\nu)}_{\eps, x}\Big(\prod_{i=1}^n \phi_i(Y_{t_i\wedge \tau_\eps}) \Big|\tau_\eps<\infty\Big)\nn\\
 &=  \widetilde{E}^{(2+2\nu)}_{\eps, x}\Big(\prod_{i=1}^n \phi_i(Y_{t_i\wedge \tau_\eps}) 1_{\{\tau_\eps<\infty\}}\Big) \cdot \frac{|x|^{2\nu}}{\eps^{2\nu}}\nn\\
 &= \widetilde{E}^{(2+2\nu)}_{\eps, x}\Big(\prod_{i=1}^n \phi_i(Y_{t_i\wedge \tau_\eps})  \widetilde{P}^{(2+2\nu)}_{\eps, Y_{t_n\wedge \tau_\eps}}({\tau_\eps<\infty})\Big) \cdot \frac{|x|^{2\nu}}{\eps^{2\nu}}\nn\\
 &= \widetilde{E}^{(2+2\nu)}_{\eps, x}\Big(\prod_{i=1}^n \phi_i(Y_{t_i\wedge \tau_\eps}) \frac{|x|^{2\nu}}{|Y_{t_n\wedge \tau_\eps}|^{2\nu}} \Big),
 \end{align}
 where the second last equality is by the strong Markov property of $Y$. Similar to the derivation of \eqref{e2.4} using \eqref{e2.1}, \eqref{e2.2} and \eqref{e2.3}, by replacing $\nu$ with $-\nu$ in \eqref{e2.2} and \eqref{e2.3}, another application of Girsanov's theorem implies there is a unique probability $\Hat{P}_{\eps, x}^{(2-2\nu)}$ on $C([0,\infty), \R^d)$ so that for any $t\geq 0$,
  \begin{align}\label{e2.5}
  d\Hat{P}_{\eps, x}^{(2-2\nu)}\Big|_{\cY_{t}}=\frac{|Y_{{\tau_\eps\wedge t}}|^{-\nu-\mu}}{|x|^{-\nu-\mu}}\exp\Big(-\int_0^{\tau_\eps\wedge t} \frac{2(4-d)}{|Y_s|^2} ds \Big) dP_x\Big|_{\cY_{t}},
   \end{align}
    and under $\Hat{P}_{\eps, x}^{(2-2\nu)}$, $Y$ is the unique solution of
  \begin{align}\label{ebaa7.3}
       Y_t=x+\Hat{B}_{t}+\int_0^{\tau_{\eps} \wedge t} (-\nu-\mu)\frac{Y_s}{|Y_s|^2}ds,
  \end{align}
  (so again the drift is stopped when $Y$ hits the ball $\overline{B(0,\eps)}$). Here $\Hat{B}$ is a standard $d$-dimensional Brownian motion with respect to $\Hat{P}_{\eps, x}^{(2-2\nu)}$. 
   Combining \eqref{e2.4} and \eqref{e2.5}, we can get
    \begin{align}\label{e2.6}
  \Hat{P}_{\eps, x}^{(2-2\nu)}\Big|_{\cY_{t}}=\frac{|x|^{2\nu}}{|Y_{{\tau_\eps\wedge t}}|^{2\nu}}\widetilde{P}^{(2+2\nu)}_{\eps, x}\Big|_{\cY_{t}}.
   \end{align}
Now apply \eqref{e2.6} in \eqref{e2.7} to see that
 \begin{align*}
 J= \Hat{E}^{(2-2\nu)}_{\eps, x}\Big(\prod_{i=1}^n \phi_i(Y_{t_i\wedge \tau_\eps}) \Big)=\Hat{E}^{(2-2\nu)}_{x}\Big(\prod_{i=1}^n \phi_i(Y_{t_i\wedge \tau_\eps}) \Big),
 \end{align*}
where the last equality follows since one can easily check that $\{Y_{t\wedge \tau_\eps}, t\geq 0\}$ under $\Hat{P}_{\eps, x}^{(2-2\nu)}$ is equal in law to that under $\Hat{P}^{(2-2\nu)}_x$ (see \eqref{eb7.3} and \eqref{ebaa7.3}). So the proof is complete.
  \end{proof}
Returning to the proof of Proposition \ref{p8.1}, we apply the above lemma in \eqref{ed1.0.2} to conclude
  \begin{align}\label{ed1.0.3}
  I=\frac{\eps^{p}}{|x|^{p}}\Hat{E}_{x}^{(2-2\nu)}\Big(\Phi_{\tau_\eps} \exp\big(-\int_0^{\tau_{\eps}} (g(Y_s)-V^\infty(Y_s)) ds\big)\Big),
  \end{align}
  and the proof is complete.
\end{proof}  
  One can show (as for \eqref{ev2.5}) that the radial process $\{|Y_{s\wedge \tau_0}|, s\geq 0\}$ under $\Hat{P}_{x}^{(2-2\nu)}$ is a $(2-2\nu)$-dimensional Bessel process stopped at $0$.  By applying Lemma \ref{l1.5} to the radial process $\{|Y_{s \wedge \tau_{\eps}}|, s \geq 0 \}$, we can get following ``well-known'' result on Bessel process (see Corollary 2.3 of Lawler \cite{Law18}).
\begin{corollary}\label{c1.4}
 For $\delta \in\R$, let $(\rho_t)$ denote a $\delta$-dimensional Bessel process starting from $r>0$ under $P_{r}^{(\delta)}$. For any $\gamma>0$ and any $\eps>0$ such that $r>\eps$, we have the law of $\{\rho_{s \wedge \tau_{\eps}}, s \geq 0 \}$ conditioning on $\{\tau_{\eps}<\infty\}$ under $P_{r}^{(2+2\gamma)}$ is equal to the law of $\{\rho_{s \wedge \tau_{\eps}}, s \geq 0 \}$ under $P_{r}^{(2-2\gamma)}$. 
  \end{corollary}

\section{Proof of Proposition \ref{p3.1} and Proposition \ref{p3.3}(i)}\label{s8}

In this section we will give the proof of Proposition \ref{p3.1} and Proposition \ref{p3.3}(i). Recall the definitions of $U^{\vec{\lambda},\vec{x},\vec{\eps}}$, $V^{\vec{\lambda},\vec{x}}$ and  $W^{\vec{\lambda},\vec{x},\eps}$ from Section \ref{s4}.

\noindent $\mathbf{Throughout\ the\ rest\ of\ this\ paper\, we\ note}$ when dealing with $U^{\vec{\lambda},\vec{x},\vec{\eps}}$, we will fix $\lambda_1,\lambda_2>0$ and let $\eps_1, \eps_2$ converge to $0$. For $V^{\vec{\lambda},\vec{x}}$ we will let $\lambda_1$, $\lambda_2$ converge to infinity; for $W^{\vec{\lambda},\vec{x},\eps}$ we will fix $\lambda_2>0$ and let $\lambda_1$ converge to infinity and $\eps$ converge to $0$.

\subsection{Preliminaries}

 \begin{lemma}\label{l7.5}
 For any $x_1\neq x_2$ and $x\neq x_1, x_2,$, we have
 \[    
    \lim_{\eps_1, \eps_2 \downarrow 0} U^{\vec{\lambda},\vec{x},\vec{\eps}}(x)=
     \lim_{\lambda_1, \lambda_2 \to \infty } V^{\vec{\lambda},\vec{x}}(x)=
      \lim_{\lambda_1 \to \infty, \eps \downarrow 0 } W^{\vec{\lambda},\vec{x},\eps}(x)=V^{\infty,\vec{x}}(x),
      \]
  where $V^{\infty,\vec{x}}(x)$ is as in \eqref{ev}.
  \end{lemma}
  \begin{proof}
 This result follows intuitively from \eqref{e20.1} and
 more details for the proof can be found 
in Appendix.
\end{proof}

Use $1-ab\leq (1-a)+(1-b)$ for all $0\leq a,b\leq 1$ to see that for all $x$ so that $|x-x_i|>\eps_i, i=1,2$,
\begin{align}\label{eb1.1.4}
U^{\vec{\lambda},\vec{x},\vec{\eps}}(x)\leq & \sum_{i=1}^{2} \N_{x} \Big(1-\exp\Big(- \lambda_i \frac{X_{G_{\eps_i}^{x_i}}(1)}{\varepsilon_i^2}\Big)1(X_{G_{\eps_i/2}^{x_i}}=0)\Big)\nn\\
= &\sum_{i=1}^{2} \N_{x} \Big(1-\exp\Big(- (\lambda_i+4U^{\infty,1}(2)) \frac{X_{G_{\eps_i}^{x_i}}(1)}{\varepsilon_i^2}\Big)\Big)\nn\\
=&U^{\widetilde{\lambda}_1\eps_1^{-2},\eps_1} (x-x_1) + U^{\widetilde{\lambda}_2\eps_2^{-2},\eps_2} (x-x_2), 
\end{align}
where the first equality follows in a similar way to the derivation of \eqref{er1.5} and the last is by \eqref{e1.2.5} and by letting
 \begin{align}\label{ae2.4}
\widetilde{\lambda}_i:=\lambda_i+4U^{\infty,1}(2), i=1,2.
\end{align}
Next we apply $1-ab\geq (1-a)\vee (1-b)$, $\forall 0\leq a,b\leq 1$ to see that for all $x$ so that $|x-x_i|>\eps_i, i=1,2$,
\begin{align}\label{ec1.1.3}
&U^{\vec{\lambda},\vec{x},\vec{\eps}}(x)\geq U^{\widetilde{\lambda}_1\eps_1^{-2},\eps_1} (x-x_1) \vee U^{\widetilde{\lambda}_2\eps_2^{-2},\eps_2} (x-x_2).
\end{align}
Similar to the above derivations, one can also show that for all $x\neq x_1, x_2$,
\begin{align}\label{e10.2}
V^{\lambda_1}(x-x_1)\vee V^{\lambda_2}(x-x_2) \leq V^{\vec{\lambda},\vec{x}}(x)\leq V^{\lambda_1}(x-x_1)+V^{\lambda_2}(x-x_2),
\end{align}
and for all $x\neq x_1$ and $|x-x_2|>\eps,$
\begin{align}\label{8.1}
\begin{cases}
&W^{\vec{\lambda},\vec{x},\eps}(x) \leq V^{\lambda_1}(x-x_1)+U^{\widetilde{\lambda}_2\eps^{-2},\eps}(x-x_2),\\
&W^{\vec{\lambda},\vec{x},\eps}(x)\geq V^{\lambda_1}(x-x_1)\vee U^{\widetilde{\lambda}_2\eps^{-2},\eps}(x-x_2).
\end{cases}
\end{align}
 By (4.1) of \cite{HMP18} we have $4U^{\infty,1}(2)\geq 4V^\infty(2)=\lambda_d$ and so $\widetilde{\lambda}_i\geq \lambda_d$. Then it follows from (4.17) of \cite{Hong20} that 
\begin{align}\label{ae8.0}
U^{\widetilde{\lambda}_i\eps_i^{-2}, \eps_i}(x)\geq V^{\infty}(x), \text{ for all } |x|\geq \eps_i \text{ for } i=1,2.
\end{align}
Together with \eqref{ec1.1.3}, we have for all $x \text{ so that } |x-x_i|\geq \eps_i, i=1,2$, 
 \begin{align}\label{e4.11}
 U^{\vec{\lambda},\vec{x},\vec{\eps}}(x) \geq V^{\infty}(x-x_1)\vee  V^{\infty}(x-x_2),
 \end{align}
and by \eqref{8.1} we have for all $x\neq x_1 \text{ and } |x-x_2|>\eps$,
\begin{align}\label{ae7.4}
W^{\vec{\lambda},\vec{x},\eps}(x)\geq V^{\lambda_1}(x-x_1) \vee V^\infty(x-x_2).
\end{align}

 Fix $x_{1}\neq x_{2}$ and $x\neq x_{1}, x_{2}$. Let $(B_t)$ denote a $d$-dimentional Brownian motion starting from $x$ under $P_x$. Let $r_{\lambda_i}=\lambda_0 \lambda_i^{-\frac{1}{4-d}}, i=1,2,$ where $\lambda_0$ will be chosen to be some fixed large constant below. Set $T_{r_{\lambda}}=T_{r_{\lambda_1}}^1 \wedge T_{r_{\lambda_2}}^2$ where $T_{r_{\lambda_i}}^i=\inf\{t\geq 0: |B_t-x_i|\leq r_{\lambda_i}\}, i=1,2$. Let $\lambda_1, \lambda_2>0$ be large so that
\begin{align}\label{5.1}
0<4(r_{\lambda_1} \vee r_{\lambda_2})\leq \min\{|x-x_1|, |x-x_2|, |x_1-x_2|\}.
\end{align}
 The following result is from Lemma 9.4 of \cite{MP17}.

\begin{lemma}\label{l10.3}
For any $t>0$, we have for $i=1,2$, \[V_i^{\vec{\lambda},\vec{x}}(x)=E_x\Big( V_i^{\vec{\lambda},\vec{x}}(B(t\wedge T_{r_{\lambda}} ))\exp\big(-\int_0^{t\wedge T_{r_{\lambda}} } V^{\vec{\lambda},\vec{x}}(B_s) ds\big)\Big).\]
\end{lemma}

 \begin{lemma}\label{al4.0}
Let $G=G_{\eps_1}^{x_1} \cap G_{\eps_2}^{x_2}$. Then $U^{\vec{\lambda},\vec{x},\vec{\eps}}$ is a $C^2$ function on $G$ and solves
\begin{align}
\Delta U^{\vec{\lambda},\vec{x},\vec{\eps}}=(U^{\vec{\lambda},\vec{x},\vec{\eps}})^2 \text{ on } G.
\end{align}
\end{lemma}
\begin{proof}
The proof follows in a similar way to that of Lemma S.1.1 of \cite{HMP18_supp} and will be given in Appendix.
\end{proof}

Set $T_{2\eps_i}^i=\inf\{t\geq 0: |B_t-x_i|\leq 2\eps_i \},  i=1, 2$ and $T_{\eps}=T_{2\eps_1}^1 \wedge T_{2\eps_2}^2$. Let $\eps_1, \eps_2>0$ be small so that
$0<4(\eps_1\vee \eps_2)<\min\{|x_1-x|, |x_2-x|, |x_1-x_2|\}.$
 \begin{lemma}\label{l4.3}
For any $t>0$, we have for $i=1,2$, \[U_i^{\vec{\lambda},\vec{x},\vec{\eps}}(x)=E_x\Big( U_i^{\vec{\lambda},\vec{x},\vec{\eps}}(B({t\wedge T_{\eps}}))\exp\big(-\int_0^{t\wedge T_{\eps}} U^{\vec{\lambda},\vec{x},\vec{\eps}}(B_s) ds\big)\Big).\] 
\end{lemma}
\begin{proof}
By using Lemma \ref{al4.0}, the proof is similar to the derivation of Lemma \ref{l10.3}.
\end{proof}

\begin{lemma}\label{8.5}
Let $G=\{x: x\neq x_1\} \cap G_\eps^{x_2}$. Then $W^{\vec{\lambda},\vec{x},\varepsilon}$ is a $C^2$ function on $G$ and solves
\begin{align}\label{8.6}
\Delta W^{\vec{\lambda},\vec{x},\varepsilon}=(W^{\vec{\lambda},\vec{x},\varepsilon})^2 \text{ on } G.
\end{align}
\end{lemma}
\begin{proof}
It follows in a similar manner to the proof of Lemma \ref{al4.0}.
\end{proof}

Let $r_{\lambda_1}=\lambda_0 \lambda_1^{-\frac{1}{4-d}}$ where $\lambda_0$ will be chosen to be some fixed large constant below. Set $T_{\lambda_1, \eps}=T_{r_{\lambda_1}}^1 \wedge T_{2\eps}^2$ where $T_{r_{\lambda_1}}^1=\inf\{t\geq 0: |B_t-x_1|\leq r_{\lambda_1} \}$ and $T_{2\eps}^2=\inf\{t\geq 0: |B_t-x_2|\leq 2\eps \}$. Let $\lambda_1>0$ large and $\eps>0$ small such that  $0<4(r_{\lambda_1}\vee \eps)<\min\{|x_1-x|, |x_2-x|,  |x_1-x_2|\}$.   
\begin{lemma}\label{8.9}
For any $t>0$, we have for $i=1,2$, \[W_i^{\vec{\lambda},\vec{x},\eps}(x)=E_x\Big( W_i^{\vec{\lambda},\vec{x},\eps}(B({t\wedge T_{\lambda_1, \eps}}))\exp\big(-\int_0^{t\wedge T_{\lambda_1, \eps}} W^{\vec{\lambda},\vec{x},\eps}(B_s) ds\big)\Big).\]
\end{lemma}
\begin{proof}
By using Lemma \ref{8.5}, the proof follows in a similar way to that of Lemma \ref{l10.3}.
\end{proof}

\subsection{Proof of Proposition \ref{p3.1} and Proposition \ref{p3.3}(i)}
Given the similarities of the proofs of Propositions \ref{p3.1}(i),  \ref{p3.1}(ii) and  \ref{p3.3}(i), we will only give the proof of Proposition \ref{p3.1}(ii) here and  other proofs can be found in Appendix. 
\begin{proof}[Proof of Proposition \ref{p3.1}(ii)]
By symmetry it suffices to consider the case $i=1$. Recall Lemma \ref{l4.3} to get
\begin{align*}
\frac{1}{\eps_1^{p-2}}U_1^{\vec{\lambda},\vec{x},\vec{\eps}}(x)=\frac{1}{\eps_1^{p-2}}\lim_{t\to \infty}E_x\Big( U_1^{\vec{\lambda},\vec{x},\vec{\eps}}(B({t\wedge T_{\eps}}))\exp\big(-\int_0^{t\wedge T_{\eps}} U^{\vec{\lambda},\vec{x},\vec{\eps}}(B_s) ds\big)\Big),
\end{align*}
where $T_\eps=T_{2\eps_1}^1\wedge T_{2\eps_2}^2$ and $T_{2\eps_i}^i=\inf\{t\geq 0: |B_t-x_i|\leq 2\eps_i\}$ for $i=1,2$. 
By \eqref{mc1.0.1}, we have $U_1^{\vec{\lambda},\vec{x},\vec{\eps}}(x) \to 0$ as $|x|\to \infty$ and $U_1^{\vec{\lambda},\vec{x},\vec{\eps}}(B({t\wedge T_{\eps}}))$ is uniformly bounded for all $t\geq 0$. Apply Dominated Convergence to see that
   \begin{align}\label{m9.3}
 \frac{1}{\eps_1^{p-2}}U_1^{\vec{\lambda},\vec{x},\vec{\eps}}(x) &=\frac{1}{\eps_1^{p-2}} E_x\Big(1_{\{T_\eps<\infty\}} U_1^{\vec{\lambda},\vec{x},\vec{\eps}}(B({T_{\eps}}))\exp\big(-\int_0^{T_{\eps}} U^{\vec{\lambda},\vec{x},\vec{\eps}}(B_s) ds\big)\Big)\nn\\
 &=\sum_{i=1}^2  E_x\Big(1_{\{T_{2\eps_i}^i<\infty\}} 1_{\{T_{2\eps_i}^i<T_{2\eps_{3-i}}^{3-i}\}} \frac{1}{\eps_1^{p-2}}U_1^{\vec{\lambda},\vec{x},\vec{\eps}}(B({T_{2\eps_i}^i})) \nn\\
 &\quad \quad \quad \quad \quad \quad  \exp\big(-\int_0^{T_{2\eps_i}^i} U^{\vec{\lambda},\vec{x},\vec{\eps}}(B_s) ds\big)\Big):=I_1+I_2.
  \end{align}
  We first deal with $I_2$. Note in the integrand of $I_2$ we may assume that  $|B({T_{2\eps_2}^2})-x_2|=2\eps_2$ and so for $\eps_2>0$ small we have $|x_1-B({T_{2\eps_2}^2})|>\Delta/2$ where $\Delta=|x_1-x_2|$. Apply \eqref{mc1.0.1} with $x=B({T_{2\eps_2}^2})$ to get
\begin{align}\label{eb1.2.1}
  &\frac{1}{\eps_1^{p-2}}U_1^{\vec{\lambda},\vec{x},\vec{\eps}}(B({T_{2\eps_2}^2}))\leq  |B({T_{2\eps_2}^2})-x_1|^{-p}\leq\Delta^{-p} 2^p.
    \end{align}
    Let $\tau_r=\inf\{t\geq 0: |B_t|\leq r\}$ and use the above and \eqref{e4.11} to see that  $I_2$ becomes
   \begin{align}\label{m9.4}
   I_2&\leq  2^p \Delta^{-p} E_x\Big(1_{\{T_{2\eps_2}^2<\infty\}} 1_{\{T_{2\eps_2}^2<T_{2\eps_{1}}^{1}\}} \exp\big(-\int_0^{T_{2\eps_2}^2} U^{\vec{\lambda},\vec{x},\vec{\eps}}(B_s) ds\big)\Big)\\
      &\leq  2^p \Delta^{-p} E_{x-x_2}\Big(1_{\{\tau_{2\eps_2}<\infty\}}  \exp\big(-\int_0^{\tau_{2\eps_2}} V^\infty(B_s) ds\big)\Big)\nn\\
      &= 2^p \Delta^{-p} (2\eps_2/|x-x_2|)^{p} \to 0 \text{ as } \eps_2 \downarrow 0,\nn
   \end{align}
   where in the last equality we have used Proposition \ref{p8.1} with $g=V^\infty$.
   
   Now we will turn to $I_1$. 
   Let $(Y_t, t\geq 0)$ be the $d$-dimensional coordinate process under Wiener measure, $P_x$. By slightly abusing the notation, we set 
$\tau_r=\tau_r^Y=\inf\{t\geq 0: |Y_t|\leq r\}$ for any $r>0$, and set
      \begin{align}\label{eb4.1}
     T^{'}_{2\eps_2}= T^{',Y}_{2\eps_2}=\inf\{t\geq 0: |Y_t-(x_2-x_1)|\leq 2\eps_2\}.
     \end{align}
Then use translation invariance of $Y$ to get
\begin{align*}
I_1&=E_{x-x_1}\Big(1_{\{\tau_{2\eps_1}<\infty\}}  1_{\{\tau_{2\eps_1}<T^{'}_{2\eps_{2}}\}}  \frac{1}{\eps_1^{p-2}}U_1^{\vec{\lambda},\vec{x},\vec{\eps}}(Y_{\tau_{2\eps_1}}+x_1)\\
 &\quad \quad \quad \quad \quad \quad\exp\big(-\int_0^{\tau_{2\eps_1}} U^{\vec{\lambda},\vec{x},\vec{\eps}}(Y_s+x_1) ds\big)\Big).
\end{align*}            
 Recall that $\Hat{P}_{x}^{(2-2\nu)}$ is the law of $Y$ starting from $x$ such that $Y$ satisfies the SDE as in \eqref{eb7.3}.
  Apply Proposition \ref{p8.1} with $g(\cdot)=U^{\vec{\lambda},\vec{x},\vec{\eps}}(\cdot+x_1)$ in the above to get
     \begin{align}\label{ec1.0.3}
   I_1=&\frac{(2\eps_1)^{p}}{|x-x_1|^{p}} \Hat{E}^{(2-2\nu)}_{x-x_1}\Big(1_{\{\tau_{2\eps_1}<T^{'}_{2\eps_{2}}\}} \frac{1}{\eps_1^{p-2}}U_1^{\vec{\lambda},\vec{x},\vec{\eps}}(Y_{\tau_{2\eps_1}}+x_1)\nn \\
   &\quad \quad \quad \times \exp\Big(-\int_0^{\tau_{2\eps_1}} (U^{\vec{\lambda},\vec{x},\vec{\eps}}(Y_s+x_1)-V^\infty(Y_s)) ds\Big)\Big)\nn\\
 =&\frac{2^{p}}{|x-x_1|^{p}} \Hat{E}^{(2-2\nu)}_{x-x_1}\Big([1_{\{\tau_{2\eps_1}<T^{'}_{2\eps_{2}}\}}] \times [\eps_1^{2}U_1^{\vec{\lambda},\vec{x},\vec{\eps}}(Y_{\tau_{2\eps_1}}+x_1)]\nn \\
   &\quad \quad \quad \times \Big[\exp\Big(-\int_0^{\tau_{2\eps_1}} (U^{\vec{\lambda},\vec{x},\vec{\eps}}(Y_s+x_1)-U^{\widetilde{\lambda}_1 \eps_1^{-2}, \eps_1}(Y_s)) ds\Big)\Big]\nn\\
   &\quad \quad \quad \times \Big[\exp\Big(-\int_0^{\tau_{2\eps_1}} (U^{\widetilde{\lambda}_1 \eps_1^{-2}, \eps_1}(Y_s)-V^\infty(Y_s)) ds\Big)\Big]\Big)\nn\\
   :=&\frac{2^p}{|x-x_1|^{p}} \Hat{E}^{(2-2\nu)}_{x-x_1} ([J_1] [J_2] [J_3][J_4]),
   \end{align}
   where $\widetilde{\lambda}_1$ is as in \eqref{ae2.4} and we have ordered the fours terms in square brackets as $J_1, \dots, J_4$.
   
We first consider  $J_2$. Recall \eqref{e4.3.2} and use translation invariance to get 
 \begin{align*}
 J_2=&\N_{Y_{\tau_{2\eps_1}}+x_1}\Bigg({X_{G_{\eps_1}^{x_1}}(1)}\prod_{i=1}^2\exp\Big(-\lambda_i \frac{X_{G_{\eps_i}^{x_i}}(1)}{\eps_i^{2}}\Big)1(X_{G_{\eps_i/2}^{x_i}}=0) \Bigg) \\
 =&\N_{Y_{\tau_{2\eps_1}}}\Bigg({X_{G_{\eps_1}}(1)}  \exp\Big(- \lambda_1 \frac{X_{G_{\eps_1}}(1)}{\eps_1^{2}}\Big)1(X_{G_{\eps_1/2}}=0)\nn\\
 &\quad \quad \quad \times \exp\Big(- \lambda_2 \frac{X_{G_{\eps_2}^{x_2-x_1}}(1)}{\eps_2^{2}}\Big)1(X_{G_{\eps_2/2}^{x_2-x_1}}=0) \Bigg).
   \end{align*} 
By the scaling of Brownian snake and its exit measure under the excursion measure $\N_x$ (see, e.g., the proof of Proposition V.9 in \cite{Leg99}), we have
 \begin{align}\label{eb6.5}
 J_2=& \eps_1^{-2} \N_{Y_{\tau_{2\eps_1}}/\eps_1}\Bigg(\eps_1^2 {X_{G_{1}}(1)}\exp\Big(- \lambda_1 {X_{G_{1}}(1)}\Big)1(X_{G_{1/2}}=0) \nn\\
 &\quad \quad \quad \times\exp\Big(- \lambda_2 \frac{X_{G_{\eps_2/\eps_1}^{(x_2-x_1)/\eps_1}}(1)}{(\eps_2/\eps_1)^{2}}\Big)1\Big(X_{G_{\eps_2/2\eps_1}^{(x_2-x_1)/\eps_1}}=0\Big) \Bigg)\nn \\
\overset{\text{law}}{=}& \N_{Y_{\tau_{2}}}\Bigg( {X_{G_{1}}(1)}\exp\Big(- \lambda_1 {X_{G_{1}}(1)}\Big)1(X_{G_{1/2}}=0) \nn\\
 &\quad \quad \quad \times\exp\Big(- \lambda_2 \frac{X_{G_{\eps_2/\eps_1}^{(x_2-x_1)/\eps_1}}(1)}{(\eps_2/\eps_1)^{2}}\Big)1\Big(X_{G_{\eps_2/2\eps_1}^{(x_2-x_1)/\eps_1}}=0\Big)\Bigg),
   \end{align} 
   where the last equality is by the scaling of $Y$.
Note for any $K>0$, we have \[\Big|\frac{x_2-x_1}{\eps_1}\Big|-\frac{\eps_2}{\eps_1}\geq \frac{|x_2-x_1|/2}{\eps_1}>K \text{ for } \eps_1, \eps_2 \text{ small enough, }\] and so by \eqref{ea0.0} and \eqref{ea1.1} we conclude $\N_{Y_{\tau_{2}}}$-a.e.
\[X_{G_{\eps_2/\eps_1}^{(x_2-x_1)/\eps_1}}(1)=X_{G_{\eps_2/2\eps_1}^{(x_2-x_1)/\eps_1}}(1)=0 \text{ for } \eps_1, \eps_2 \text{ small enough. }\]
Therefore an application of Dominated Convergence will give us
  \begin{align}\label{eb1.3.6}
 &\lim_{\eps_{1}, \eps_{2} \downarrow 0}  \N_{Y_{\tau_{2}}}\Bigg( {X_{G_{1}}(1)}\exp\Big(- \lambda_1 {X_{G_{1}}(1)}\Big)1(X_{G_{1/2}}=0)\nn\\
 &\quad \quad \quad \times \exp\Big(- \lambda_2 \frac{X_{G_{\eps_2/\eps_1}^{(x_2-x_1)/\eps_1}}(1)}{(\eps_2/\eps_1)^{2}}\Big)1\Big(X_{G_{\eps_2/2\eps_1}^{(x_2-x_1)/\eps_1}}=0\Big)\Bigg)\nn \\
 =& \N_{Y_{\tau_{2}}}\Big( {X_{G_{1}}(1)}e^{- \lambda_1 {X_{G_{1}}(1)}}1(X_{G_{1/2}}=0)\Big)\nn\\
 =& \N_{2e_1}\Big( {X_{G_{1}}(1)}e^{- \lambda_1 {X_{G_{1}}(1)}}1(X_{G_{1/2}}=0)\Big) =U_1^{\widetilde{\lambda}_1,1}(2),
     \end{align} 
     where the next to last equality is by spherical symmetry and $e_1$ is the first unit basis vector. In the last equality we have used \eqref{er1.5}, \eqref{er12.2} with $\eps=1$, $x=2e_1$ and  \eqref{ae2.4}.
       In view of \eqref{eb6.5} and \eqref{eb1.3.6}, we have proved 
\begin{align*}
J_2=\eps_1^{2}\ U_1^{\vec{\lambda},\vec{x},\vec{\eps}}(Y_{\tau_{2\eps_1}}+x_1) \to U_1^{\widetilde{\lambda}_1,1}(2) \text{ in distribution}\text{ as  } \eps_1, \eps_2 \downarrow 0.
\end{align*}
Since $U_1^{\widetilde{\lambda}_1,1}(2)$ is a constant, we conclude that under $\Hat{P}_{x-x_1}^{(2-2\nu)}$,
\begin{align}\label{eb6.4}
J_2=\eps_1^{2}\ U_1^{\vec{\lambda},\vec{x},\vec{\eps}}(Y_{\tau_{2\eps_1}}+x_1) \to U_1^{\widetilde{\lambda}_1,1}(2) \text{ in probability}\text{ as } \eps_1, \eps_2 \downarrow 0.
\end{align}

We continue to show that with  $\Hat{P}^{(2-2\nu)}_{x-x_1}$-probability one,
   \begin{align}\label{ec2.1.0}
   J_1=1_{\{\tau_{2\eps_1}<T^{'}_{2\eps_{2}}\}} \to 1 \text{ as } \eps_1, \eps_2 \downarrow 0,
   \end{align}
   and 
      \begin{align}\label{ec2.1.1}
   J_3=\exp\big(-&\int_0^{\tau_{2\eps_1}} (U^{\vec{\lambda},\vec{x},\vec{\eps}}(Y_s+x_1)-U^{\widetilde{\lambda}_1\eps_1^{-2},\eps_1}(Y_s)) ds\big)\nn \\
     &\to \exp\big(-\int_0^{\tau_{0}} (V^{\vec{\infty},\vec{x}}(Y_s+x_1)-V^\infty(Y_s)) ds\big) \text{ as } \eps_1, \eps_2 \downarrow 0.
     \end{align}

  Since the drift of $\{Y_t, t\geq 0\}$ as in \eqref{eb7.3} is bounded up to time $\tau_{\eps}$ for any $\eps>0$ and since Brownian motion in $d\geq 2$ won't hit points, we conclude by Girsanov's theorem (recall \eqref{e2.5}) that $\{Y_t, t\geq 0\}$ won't hit the point $x_1-x_2\neq 0$ and so with $\Hat{P}_{x-x_1}^{(2-2\nu)}$ probability one,
  \begin{align}\label{eb1.1.0}
\exists \delta(\omega)>0 \text{ so that } |Y_s-(x_2-x_1)|>\delta \text{ for all } 0\leq s\leq \tau_0,
    \end{align}
     which implies (recall \eqref{eb4.1})
    \begin{align}\label{eb4.2}
  T^{'}_{2\eps_{2}}=\infty \text{ for all } 0<\eps_2<\delta(\omega)/2, \quad  \Hat{P}_{x-x_1}^{(2-2\nu)}-a.s.
    \end{align}
Since $\tau_0$ under $\Hat{P}_{x-x_1}^{(2-2\nu)}$ is the hitting time $\tau_0$ of a $(2-2\nu)$-dimensional Bessel process, it follows that with $\Hat{P}_{x-x_1}^{(2-2\nu)}$-probability one (see, e.g., Exercise (1.33) in Chp. XI of \cite{RY94})
  \begin{align}\label{eb1.1.1}
\tau_0<\infty,\quad  \Hat{P}_{x-x_1}^{(2-2\nu)}-a.s.
    \end{align}
    Therefore by \eqref{eb4.2} we have \eqref{ec2.1.0}. 
    
  Fix $\omega$ outside a null set such that both \eqref{eb1.1.0} and \eqref{eb1.1.1} holds. For all $0<\eps_{1}, \eps_{2}<\delta(\omega)/2$, we have   \begin{align}\label{eb1.1.2}
|Y_{s}|\geq 2\eps_1 \text{ and } |Y_s-(x_2-x_1)|>\delta>2\eps_2, \text{ for all } 0<s<\tau_{2\eps_{1}}.
    \end{align}
Now apply \eqref{eb1.1.4} with $Y_{s}+x_1$ in place of $x$  to get 
  \begin{align}\label{ae6.3}
&(U^{\vec{\lambda},\vec{x},\vec{\eps}}(Y_s+x_1)-U^{\widetilde{\lambda}_1\eps_1^{-2},\eps_1}(Y_s)) 1_{\{0<s<\tau_{2\eps_{1}}\}}\nn\\
 &\leq U^{\widetilde{\lambda}_2\eps_2^{-2},\eps_2}(Y_s-(x_2-x_1)) 1_{\{0<s<\tau_{2\eps_{1}}\}}\leq U^{\widetilde{\lambda}_2\eps_2^{-2},\eps_2}(\delta) 1_{\{0<s<\tau_{0}\}},
 \end{align}
 where in the last inequality we have used \eqref{eb1.1.2} and the fact that \break $r\mapsto U^{\lambda,\eps}(r)$ is decreasing from Lemma 3.2(b) of \cite{MP17}.
 Corollary 4.3 of \cite{HMP18} gives us 
\[U^{\widetilde{\lambda}_2,1}(x)\leq U^{\infty,1}(x)\leq 3(4-d)|x|^{-2}, \forall |x|>1 \text{ large}.\]
By scaling of $U^{\lambda,\eps}$ from \eqref{ev5.3}, we have for $\eps_2>0$ small,
\begin{align}\label{aea6.3}
U^{\widetilde{\lambda}_2\eps_2^{-2},\eps_2}(\delta) =\eps_2^{-2} U^{\widetilde{\lambda}_2,1}(\delta/\eps_2) \leq 3(4-d)\delta^{-2}\leq 6\delta^{-2}.
\end{align}
Combining \eqref{ae6.3} and \eqref{aea6.3}, we have for $\eps_2>0$ small,
  \begin{align}\label{eb1.1.3}
  (U^{\vec{\lambda},\vec{x},\vec{\eps}}(Y_s+x_1)-U^{\widetilde{\lambda}_1\eps_1^{-2},\eps_1}(Y_s)) 1_{\{0<s<\tau_{2\eps_{1}}\}}\leq 6\delta^{-2}1_{\{0<s<\tau_{0}\}}.
   \end{align}
 Since we have $\tau_0(\omega)<\infty$ by \eqref{eb1.1.1}, we conclude the left-hand side term of \eqref{eb1.1.3} is bounded by an integrable bound.
By (4.38) of \cite{Hong20} we have
\begin{align}\label{ae1.4}
\lim_{\eps \downarrow 0} U^{\lambda \eps^{-2}, \eps}(x)=V^\infty(x), \forall x\neq 0, \text{ for any }\lambda>0.
\end{align}
Now use Dominated Convergence with Lemma \ref{l7.5}, \eqref{eb1.1.3} and \eqref{ae1.4} to see that with $\Hat{P}_{x-x_1}^{(2-2\nu)}$-probability one,
  \begin{align*}
 \lim_{\eps_{1}, \eps_{2} \downarrow 0}\int_0^{\tau_{2\eps_1}}& (U^{\vec{\lambda},\vec{x},\vec{\eps}}(Y_s+x_1)-U^{\widetilde{\lambda}_1\eps_1^{-2},\eps_1}(Y_s))ds\\
&= \int_0^{\tau_{0}} (V^{\vec{\infty},\vec{x}}(Y_s+x_1)-V^{\infty}(Y_s))ds,
 \end{align*}
 thus proving \eqref{ec2.1.1} holds.

Combine \eqref{eb6.4}, \eqref{ec2.1.0} and \eqref{ec2.1.1} to see that under $ \Hat{P}_{x-x_1}^{(2-2\nu)}$, we have
\begin{align}\label{ce1.3}
J_1 J_2 J_3 \to U_1^{\widetilde{\lambda}_1,1}(2) & \exp\big(-\int_0^{\tau_{0}} (V^{\vec{\infty},\vec{x}}(Y_s+x_1)-V^\infty(Y_s)) ds\big)\nn\\
&  \text{ in probability} \text{ as  } \eps_1, \eps_2 \downarrow 0.
\end{align}
Recall \eqref{mc1.0.1} to see that
  \begin{align}
 J_2=\eps_1^{2}U_1^{\vec{\lambda},\vec{x},\vec{\eps}}(Y_{\tau_{2\eps_1}}+x_1)\leq \eps_1^p |Y_{\tau_{2\eps_1}}|^{-p}=2^{-p},
  \end{align} 
 and together with \eqref{ec1.1.3} we have
   $0\leq J_1J_2J_3\leq 2^{-p},  \Hat{P}_{x-x_1}^{(2-2\nu)}$-a.s.
Recalling $J_4$ as in \eqref{ec1.0.3}, 
we have $0\leq J_4\leq 1$ by \eqref{ae8.0}. By \eqref{ev1.2} and the definition of $V^{\vec{\infty},\vec{x}}$ as in \eqref{ev}, we have
\begin{align}\label{ae8.10}
V^{\vec{\infty},\vec{x}}(x)\geq V^\infty(x-x_1) \vee V^\infty(x-x_2),  \quad \forall x \neq x_1, x_2.
\end{align}
Now use \eqref{ce1.3} and bounded convergence theorem to see that
\begin{align*}
&\Big|\Hat{E}_{x-x_1}^{(2-2\nu)}(J_1J_2J_3J_4)-\Hat{E}_{x-x_1}^{(2-2\nu)}\Big(U_1^{\widetilde{\lambda}_1,1}(2)\\
&\quad \quad  \quad \quad \quad  \quad \times \exp\big(-\int_0^{\tau_{0}} (V^{\vec{\infty},\vec{x}}(Y_s+x_1)-V^\infty(Y_s)) ds\big) \times J_4\Big)\Big| \nn \\
&\leq\Hat{E}_{|x-x_1|}^{(2-2\nu)}\Big(\Big|J_1J_2J_3-U_1^{\widetilde{\lambda}_1,1}(2) \exp\big(-\int_0^{\tau_{0}} (V^{\vec{\infty},\vec{x}}(Y_s+x_1)-V^\infty(Y_s)) ds\big)\Big|\Big)\\
& \to 0\text{ as } \eps_1, \eps_2 \downarrow 0.
\end{align*} 
In view of \eqref{ec1.0.3}, we conclude  
\begin{align}\label{ae2.5}
 &\lim_{\eps_1, \eps_2 \downarrow 0} I_1= \frac{2^pU_1^{\widetilde{\lambda}_1,1}(2)}{|x-x_1|^p}\nn\\
 & \times \lim_{\eps_1, \eps_2 \downarrow 0} \Hat{E}_{x-x_1}^{(2-2\nu)}\Big( \exp\big(-\int_0^{\tau_{0}} (V^{\vec{\infty},\vec{x}}(Y_s+x_1)-V^\infty(Y_s)) ds\big) \times J_4\Big),
\end{align} 
providing we can show the limit on the right-hand side exists.
We claim that there is some constant $C(\widetilde{\lambda}_1)>0$ such that
\begin{align}\label{ae5.9}
 &\lim_{\eps_1 \downarrow 0} \Hat{E}_{x-x_1}^{(2-2\nu)}\Big( \exp\big(-\int_0^{\tau_{0}} (V^{\vec{\infty},\vec{x}}(Y_s+x_1)-V^\infty(Y_s)) ds\big) \times J_4\Big)\nn\\
     &=C(\widetilde{\lambda}_1) \Hat{E}_{x-x_1}^{(2-2\nu)}\Big(\exp\big(-\int_0^{\tau_{0}} (V^{\vec{\infty},\vec{x}}(Y_s+x_1)-V^\infty(Y_s)) ds\big)\Big).
     \end{align}
   It will then follow from \eqref{m9.3}, \eqref{m9.4}, \eqref{ae2.5} and \eqref{ae5.9} that
\begin{align*}
     &\lim_{\eps_1, \eps_2 \downarrow 0} \frac{1}{\eps_1^{p-2}} U_1^{\vec{\lambda},\vec{x},\vec{\eps}}(x)=2^pU_1^{\widetilde{\lambda}_1,1}(2)C(\widetilde{\lambda}_1)\nn\\
     &\quad \quad \quad |x-x_1|^{-p} \Hat{E}_{x-x_1}^{(2-2\nu)}\Big(\exp\big(-\int_0^{\tau_{0}} (V^{\vec{\infty},\vec{x}}(Y_s+x_1)-V^\infty(Y_s)) ds\big)\Big),
\end{align*}
  and the proof is complete by letting 
  $C_{\ref{p3.1}}(\lambda_1)=2^pU_1^{\widetilde{\lambda}_1,1}(2)C(\widetilde{\lambda}_1)$.
 
  It remains to prove \eqref{ae5.9}. First by monotone convergence theorem and \eqref{ae8.10}, we have
\begin{align}\label{am1.3}
\lim_{\delta\downarrow 0} \Hat{E}_{x-x_1}^{(2-2\nu)}\Big(&\exp\big(-\int_0^{\tau_{\delta}} (V^{\vec{\infty},\vec{x}}(Y_s+x_1)-V^\infty(Y_s)) ds\big)\nn\\
&-\exp\big(-\int_0^{\tau_{0}} (V^{\vec{\infty},\vec{x}}(Y_s+x_1)-V^\infty(Y_s)) ds\big)\Big)=0.
\end{align} 
Since $0\leq J_4\leq 1$ for any $\eps_1>0$,  it follows from monotonicity and \eqref{ae8.10} that
\begin{align}\label{am1.4}
& \Big|\Hat{E}_{x-x_1}^{(2-2\nu)}\Big(\exp\big(-\int_0^{\tau_{\delta}} (V^{\vec{\infty},\vec{x}}(Y_s+x_1)-V^\infty(Y_s)) ds\big)\times J_4\Big)\\
&-\Hat{E}_{x-x_1}^{(2-2\nu)}\Big(\exp\big(-\int_0^{\tau_{0}} (V^{\vec{\infty},\vec{x}}(Y_s+x_1)-V^\infty(Y_s)) ds\big)\times J_4\Big)\Big|\nn\\
\leq&\Hat{E}_{x-x_1}^{(2-2\nu)}\Big(\exp\big(-\int_0^{\tau_{\delta}} (V^{\vec{\infty},\vec{x}}(Y_s+x_1)-V^\infty(Y_s)) ds\big)\nn\\
&\quad \quad -\exp\big(-\int_0^{\tau_{0}} (V^{\vec{\infty},\vec{x}}(Y_s+x_1)-V^\infty(Y_s)) ds\big)\Big) \to 0 \text{ as } \delta \downarrow 0\nn
\end{align} 
uniformly for all $\eps_1>0$ by \eqref{am1.3}.
Fixing any $\delta>0$, we will show that
\begin{align}\label{a6.0}
&\lim_{\eps_1 \downarrow 0} \Hat{E}_{x-x_1}^{(2-2\nu)}\Big(\exp\big(-\int_0^{\tau_{\delta}} (V^{\vec{\infty},\vec{x}}(Y_s+x_1)-V^\infty(Y_s)) ds\big) \times J_4\Big)\nn
\\
&= C(\widetilde{\lambda}_1) \Hat{E}_{x-x_1}^{(2-2\nu)}\Big(\exp\big(-\int_0^{\tau_{\delta}} (V^{\vec{\infty},\vec{x}}(Y_s+x_1)-V^\infty(Y_s)) ds\big)\Big).
\end{align}
Then one can easily conclude from \eqref{am1.3}, \eqref{am1.4} and \eqref{a6.0} that \eqref{ae5.9} holds.

It remains to prove \eqref{a6.0}. Recall $(\rho_s)$ is a $\gamma$-dimensional Bessel process starting from $r>0$ under $P_r^{(\gamma)}$ and let $\tau_\eps=\tau_\eps^\rho=\inf\{t\geq 0: \rho_t\leq \eps\}$. Lemma 4.5 of \cite{Hong20} implies that
for any $\lambda>0$, there is some constant $0<C_{\ref{ae13.5}}(\lambda)<\infty$ so that for all $x\neq 0$,
\begin{align}\label{ae13.5}
&\lim_{\eps\downarrow 0} E_{|x|}^{(2+2\nu)}\Big(e^{-\int_0^{ \tau_{2\eps}}  (U^{\lambda \eps^{-2}, \eps}-V^\infty)(\rho_s) ds}\Big|\tau_{2\eps}<\infty\Big)=C_{\ref{ae13.5}}(\lambda).
\end{align}
For $0<\eps_1<\delta/2$ we apply the strong Markov property of $(Y_s, s\geq 0)$ to get 
\begin{align}\label{am2.1}
&\Hat{E}_{x-x_1}^{(2-2\nu)}\Big(\exp\big(-\int_0^{\tau_{\delta}} (V^{\vec{\infty},\vec{x}}(Y_s+x_1)-V^\infty(Y_s)) ds\big)\nn\\
&\quad \quad \quad \times \exp\big(-\int_0^{\tau_{2\eps_1}} (U^{\widetilde{\lambda}_1 \eps_1^{-2}, \eps_1}(Y_s)-V^\infty(Y_s)) ds\big)\Big)\nn\\
=&\Hat{E}_{x-x_1}^{(2-2\nu)}\Bigg(\exp\big(-\int_0^{\tau_{\delta}} (V^{\vec{\infty},\vec{x}}(Y_s+x_1)-V^\infty(Y_s)) ds\big)\nn\\
&\quad \quad \quad \times\exp\big(-\int_0^{\tau_{\delta}} (U^{\widetilde{\lambda}_1 \eps_1^{-2}, \eps_1}(Y_s)-V^\infty(Y_s)) ds\big)\nn\\
&\quad \quad \quad \times \Hat{E}_{Y_{\tau_\delta}}^{(2-2\nu)} \Big(\exp\big(-\int_0^{\tau_{2\eps_1}} (U^{\widetilde{\lambda}_1 \eps_1^{-2}, \eps_1}(Y_s)-V^\infty(Y_s))ds\big) \Big)\Bigg).
\end{align} 
For the last term on the right-hand side of \eqref{am2.1}, we can use the fact that  under $\Hat{P}_{x}^{(2-2\nu)}$, $\{|Y_{s\wedge \tau_{2\eps_1}}|, s\geq 0\}$ is a stopped $(2-2\nu)$-dimensional Bessel process and then use Corollary \ref{c1.4} to get 
\begin{align}\label{am2.2}
 &\Hat{E}_{Y_{\tau_\delta}}^{(2-2\nu)} \Big(\exp\big(-\int_0^{\tau_{2\eps_1}} (U^{\widetilde{\lambda}_1 \eps_1^{-2}, \eps_1}(Y_s)-V^\infty(Y_s))ds \Big)\Big)\nn\\
 &={E}_{|Y_{\tau_\delta}|}^{(2-2\nu)}\Big(\exp\big(-\int_0^{\tau_{2\eps_1}} (U^{\widetilde{\lambda}_1 \eps_1^{-2}, \eps_1}(\rho_s)-V^\infty(\rho_s))ds\big) \Big)\nn\\
  &={E}_{\delta}^{(2+2\nu)}\Big(\exp\big(-\int_0^{\tau_{2\eps_1}} (U^{\widetilde{\lambda}_1 \eps_1^{-2}, \eps_1}(\rho_s)-V^\infty(\rho_s))ds\big) \Big|\tau_{2\eps_1}<\infty \Big)\nn\\
  &\to C_{\ref{ae13.5}}(\widetilde{\lambda}_1) \text{ as } \eps_1 \downarrow 0,
\end{align} 
 where the last is by \eqref{ae13.5}.
Next since $\delta>0$ is fixed, by \eqref{ae1.4} it follows that with $\Hat{P}_{x-x_1}^{(2-2\nu)}$-probability one,
\begin{align}\label{am2.3}
\lim_{\eps_1 \downarrow 0} \exp\big(-\int_0^{\tau_{\delta}} (U^{\widetilde{\lambda}_1 \eps_1^{-2}, \eps_1}(Y_s)-V^\infty(Y_s)) ds\big) = 1.
\end{align} 
In view of \eqref{ae8.10} and \eqref{ae8.0}, with $\Hat{P}_{x-x_1}^{(2-2\nu)}$-probability one, for any $\eps_1>0$ we have
\begin{align}\label{am2.4}
&\exp\big(-\int_0^{\tau_{\delta}} (V^{\vec{\infty},\vec{x}}(Y_s+x_1)-V^\infty(Y_s)) ds\big)\nn\\
& \times \exp\big(-\int_0^{\tau_{\delta}} (U^{\widetilde{\lambda}_1 \eps_1^{-2}, \eps_1}(Y_s)-V^\infty(Y_s)) ds\big)\nn\\
& \times\Hat{E}_{Y_{\tau_\delta}}^{(2-2\nu)} \Big(\exp\big(-\int_0^{\tau_{2\eps_1}} (U^{\widetilde{\lambda}_1 \eps_1^{-2}, \eps_1}(Y_s)-V^\infty(Y_s))ds\big) \Big)\leq 1.
\end{align} 
Combine \eqref{am2.2}, \eqref{am2.3} and \eqref{am2.4} to see that the integrand in \eqref{am2.1} converges pointwise a.s. as $\eps_1 \downarrow 0$ and is bounded by $1$. Therefore we apply Bounded Convergence Theorem to conclude 
\begin{align*}
&\lim_{\eps_1 \downarrow 0} \Hat{E}_{x-x_1}^{(2-2\nu)}\Big(\exp\big(-\int_0^{\tau_{\delta}} (V^{\vec{\infty},\vec{x}}(Y_s+x_1)-V^\infty(Y_s)) ds\big) \times J_4\Big)\nn\\
&= C_{\ref{ae13.5}}(\widetilde{\lambda}_1) \Hat{E}_{x-x_1}^{(2-2\nu)}\Big(\exp\big(-\int_0^{\tau_{\delta}} (V^{\vec{\infty},\vec{x}}(Y_s+x_1)-V^\infty(Y_s)) ds\big)\Big),
\end{align*}
and the proof of \eqref{a6.0} is complete.
\end{proof}


\section{Convergence of the second moments} \label{s9}

In this section we will give the proofs of Proposition \ref{p3.2} and Proposition \ref{p3.3}(ii).
\subsection{Preliminaries}
\begin{lemma}\label{l4.2}
For any $\lambda_1, \lambda_2, \eps_1, \eps_2, \eps>0$, the following holds for all $x$ so that $|x-x_1| \wedge |x-x_2| >(\eps_1 \vee \eps_2 \vee \eps)$:
 \begin{align*}
 \begin{cases}
 -U_{1,2}^{\vec{\lambda},\vec{x},\vec{\eps}}(x)\leq \min\{ 2\lambda_{2}^{-1} |x-x_1|^{-p} \eps_1^{p-2},& 2\lambda_{1}^{-1} |x-x_2|^{-p} \eps_2^{p-2}\},\\
 -V_{1,2}^{\vec{\lambda},\vec{x}}(x)\leq  \min\{2\lambda_{2}^{-1}c_{\ref{p1.1}}\lambda_1^{-(1+\alpha)} |x-x_1|^{-p},& 2\lambda_{1}^{-1}c_{\ref{p1.1}}\lambda_2^{-(1+\alpha)} |x-x_2|^{-p}\},\\
 -W_{1,2}^{\vec{\lambda},\vec{x},\eps}(x)\leq \min\{ 2\lambda_{2}^{-1}c_{\ref{p1.1}}\lambda_1^{-(1+\alpha)} |x-x_1|^{-p},& 2\lambda_{1}^{-1} |x-x_2|^{-p} \eps^{p-2}\}.
\end{cases}
\end{align*}
\end{lemma}
\begin{proof}
Similar to the derivation of Lemma S.1.2 in \cite{HMP18_supp}, it is easy to conclude from the definition (see \eqref{e3.5}) that $-U_{1,2}^{\vec{\lambda},\vec{x},\vec{\eps}}(x)$ is strictly decreasing in $\vec{\lambda}\in (0,\infty)^2$. So we can use this monotonicity and $U_2^{\vec{\lambda},\vec{x},\vec{\eps}}\geq 0$ (see \eqref{e4.3.2}) to get
\begin{align*}
-U_{1,2}^{\vec{\lambda},\vec{x},\vec{\eps}}(x)&\leq \frac{2}{\lambda_1} \int_{\lambda_1/2}^{\lambda_1} -\frac{\partial}{\partial \lambda_1^{'}} U_2^{(\lambda_1^{'},\lambda_2),\vec{x},\vec{\eps}}(x) d\lambda_1^{'}\\
&\leq \frac{2}{\lambda_1}  U_2^{(\lambda_1/2, \lambda_2),\vec{x},\vec{\eps}}(x) \leq \frac{2}{\lambda_1}  |x-x_2|^{-p} \eps_2^{p-2},
\end{align*}
the last by \eqref{mc1.0.1}. The result for $-U_{1,2}^{\vec{\lambda},\vec{x},\vec{\eps}}$ follows by symmetry. The proofs for $-V_{1,2}^{\vec{\lambda},\vec{x}}$ and $-W_{1,2}^{\vec{\lambda},\vec{x},\eps}$ will follow in a similar way by using \eqref{m10.0}, \eqref{m8.7} and \eqref{m8.8}.
\end{proof}
%

Fix $x_{1}\neq x_{2}$ and $x\neq x_{1}, x_{2}$.
Let $P_x$ denote the law of $d$-dimensional Brownian motion $B$ starting from $x$. Recall $r_{\lambda_1}, r_{\lambda_2}$ and $T_{r_{\lambda}}$ as in Lemma \ref{l10.3}.
The following result is from Lemma 9.5 of \cite{MP17}.
\begin{lemma}\label{l4.5}
For all $\lambda_1,\lambda_2>0$ large,
\begin{align*}
 -V_{1,2}^{\vec{\lambda},\vec{x}}(x)
 =&E_x\Big( \int_0^{T_{r_\lambda}} \prod_{i=1}^2 V_i^{\vec{\lambda},\vec{x}}(B_t)\exp\Big(-\int_0^{t} V^{\vec{\lambda},\vec{x}}(B_s) ds\Big)dt\Big)\\
 &+E_x\Big( \exp\Big(-\int_0^{ T_{r_\lambda}} V^{\vec{\lambda},\vec{x}}(B_s) ds\Big)1(T_{r_\lambda}<\infty) (-V_{1,2}^{\vec{\lambda},\vec{x}}(B_{T_{r_\lambda}}))\Big).
 \end{align*}
\end{lemma}

\begin{lemma}\label{l4.4}
For all $\eps_1, \eps_2>0$ small, we have
\begin{align*}
 -U_{1,2}^{\vec{\lambda},\vec{x},\vec{\eps}}(x)
 =&E_x\Big( \int_0^{T_\eps} \prod_{i=1}^2 U_i^{\vec{\lambda},\vec{x},\vec{\eps}}(B_t)\exp\Big(-\int_0^{t} U^{\vec{\lambda},\vec{x},\vec{\eps}}(B_s) ds\Big)dt\Big)\\
 &+E_x\Big( \exp\Big(-\int_0^{ T_{\eps}} U^{\vec{\lambda},\vec{x},\vec{\eps}}(B_s) ds\Big)1(T_{\eps}<\infty) (-U_{1,2}^{\vec{\lambda},\vec{x},\vec{\eps}}(B_{T_{\eps}}))\Big),
 \end{align*}
 where $T_{\eps}$ is as in Lemma \ref{l4.3}.
\end{lemma}
\begin{proof}
In view of Lemma \ref{l4.3}, it follows in a similar manner to the proof of  Lemma \ref{l4.5}.
\end{proof}
\begin{lemma}\label{l4.6}
For all $\lambda_1>0$ large and $\eps>0$ small,
\begin{align*}
 -&W_{1,2}^{\vec{\lambda},\vec{x},{\eps}}(x)
 =E_x\Big( \int_0^{T_{\lambda_1, \eps}} \prod_{i=1}^2 W_i^{\vec{\lambda},\vec{x},\vec{\eps}}(B_t)\exp\Big(-\int_0^{t} W^{\vec{\lambda},\vec{x},\vec{\eps}}(B_s) ds\Big)dt\Big)\\
 &+E_x\Big( \exp\Big(-\int_0^{ T_{\lambda_1, \eps}} W^{\vec{\lambda},\vec{x},\vec{\eps}}(B_s) ds\Big)1(T_{\lambda_1, \eps}<\infty) (-W_{1,2}^{\vec{\lambda},\vec{x},\vec{\eps}}(B_{T_{\lambda_1, \eps}}))\Big),
 \end{align*}
  where $T_{\lambda_1,\eps}$ is as in Lemma \ref{8.9}.
\end{lemma}
\begin{proof}
In view of Lemma \ref{8.9}, it follows in a similar manner to the proof of Lemma \ref{l4.5}.
\end{proof}

\begin{lemma}\label{l4.7}
For any $x_1\neq x_2$, if $|x-x_1| \wedge |x-x_2|>\eps_0$ for some $\eps_0>0$, then there is some constant $C_{\ref{l4.7}}(\eps_0)>0$ so that for all  $\eps_1, \eps_2>0$ small,
\begin{align}
0\leq  \frac{1}{\eps_1^{p-2}} \frac{1}{\eps_2^{p-2}}(-U_{1,2}^{\vec{\lambda},\vec{x},\vec{\eps}}(x))) \leq C_{\ref{l4.7}}(\eps_0)(1+|x_1-x_2|^{2-p}),
\end{align}
and for all $\lambda_1\geq 1$ large and $\eps>0$ small,
\begin{align}
0\leq \frac{ \lambda_1^{1+\alpha}}{\eps^{p-2}}(-W_{1,2}^{\vec{\lambda},\vec{x},\eps}(x))) \leq C_{\ref{l4.7}}(\eps_0)(1+|x_1-x_2|^{2-p}).
\end{align}
\end{lemma}
\begin{proof}
In view of Lemma \ref{l4.4} and Lemma \ref{l4.6}, it follows in a similar manner to the proof of Proposition 6.1 of \cite{MP17} and Proposition 5.1 of \cite{HMP18}.
\end{proof}

\subsection{Proof of Proposition \ref{p3.2} and \ref{p3.3}(ii)}
Given the similarities of the proofs of Propositions \ref{p3.2}(i),  \ref{p3.2}(ii) and  \ref{p3.3}(ii), we will only give the proof of \ref{p3.2}(ii) here and  other proofs can be found in Appendix.
\begin{proof}[Proof of Proposition \ref{p3.2}(ii)]
 For any  $x_1\neq x_2$, we fix any $x\neq x_1, x_2$. In order to find the limit of $\eps_1^{-(p-2)} \eps_2^{-(p-2)}(-U_{1,2}^{\vec{\lambda},\vec{x},\vec{\eps}}(x))$ as $\eps_1, \eps_2\downarrow 0,$ by Lemma \ref{l4.4}, it suffices to calculate the limits of following as $\eps_1, \eps_2\downarrow 0.$
\begin{align}\label{m5.1}
 K_1+K_2&:=\frac{1}{\eps_1^{p-2}} \frac{1}{\eps_2^{p-2}} E_x\Big( \int_0^{T_\eps} \prod_{i=1}^2 U_i^{\vec{\lambda},\vec{x},\vec{\eps}}(B_t)\exp\Big(-\int_0^{t} U^{\vec{\lambda},\vec{x},\vec{\eps}}(B_s) ds\Big)dt\Big)\nn\\
 & +\frac{1}{\eps_1^{p-2}} \frac{1}{\eps_2^{p-2}}E_x\Big( \exp\Big(-\int_0^{ T_{\eps}} U^{\vec{\lambda},\vec{x},\vec{\eps}}(B_s) ds\Big)1_{(T_{\eps}<\infty)} (-U_{1,2}^{\vec{\lambda},\vec{x},\vec{\eps}}(B_{T_{\eps}}))\Big).
\end{align}
 Recall $T_\eps=T_{2\eps_i}^i \wedge T_{2\eps_{3-i}}^{3-i}$ where $T_{2\eps_i}^i=\inf\{t\geq 0: |B_t|\leq 2\eps_i\}, i=1,2$. Let $\eps_1, \eps_2>0$ be small so that $0<4(\eps_1\vee \eps_2)<|x_1-x_2|$. 
 
 We first consider $K_2$. On $\{T_\eps<\infty\}$, by considering $T_\eps=T_{2\eps_i}^i<T_{2\eps_{3-i}}^{3-i}$ we may set $B_{T_\eps}=B_{T_{2\eps_i}^i}$ so that $|B_{T_{2\eps_i}^i}-x_i|=2\eps_i$ and \mbox{$|x_{3-i}-B_{T_{2\eps_i}^i}|\geq \Delta/2$} where $\Delta=|x_1-x_2|$.  Lemma \ref{l4.2} and the above imply
\[
-U_{1,2}^{\vec{\lambda},\vec{x},\vec{\eps}}(B_{T_\eps})=-U_{1,2}^{\vec{\lambda},\vec{x},\vec{\eps}}(B_{T_{2\eps_i}^i})\leq \frac{2}{\lambda_i} \Delta^{-p} 2^p \varepsilon_{3-i}^{p-2} \leq c\Delta^{-p} \varepsilon_{3-i}^{p-2}.
\]
 This shows that 
 \begin{align}\label{e4.5}
 K_2&\leq \frac{1}{\eps_1^{p-2}}\frac{1}{\eps_2^{p-2}} \sum_{i=1}^2 c\Delta^{-p} \varepsilon_{3-i}^{p-2} \nn\\
&\quad \quad \times E_x\Big( 1_{\{T_{2\eps_i}^i<\infty\}}1_{\{T_{2\eps_i}^i<T_{2\eps_{3-i}}^{3-i}\}}\exp\Big(-\int_0^{ T_{2\eps_i}^i} U^{\vec{\lambda},\vec{x},\vec{\eps}}(B_s) ds\Big) \Big).
 \end{align}
From \eqref{m9.4} we have for $i=1,2,$
  \begin{align*}
 &E_x\Big(1_{\{T_{2\eps_i}^i<\infty\}}1_{\{T_{2\eps_i}^i<T_{2\eps_{3-i}}^{3-i}\}}\exp\Big(-\int_0^{ T_{2\eps_i}^i} U^{\vec{\lambda},\vec{x},\vec{\eps}}(B_s) ds\Big) \Big)\leq  (2\eps_i)^p|x-x_i|^{-p},
  \end{align*}
and so \eqref{e4.5} becomes 
 \begin{align}\label{e4.7}
K_2&\leq \frac{1}{\eps_1^{p-2}}\frac{1}{\eps_2^{p-2}} c\Delta^{-p}\sum_{i=1}^2 \varepsilon_{3-i}^{p-2} (2\eps_i)^p|x-x_i|^{-p}\nn\\
&\leq C\Delta^{-p}(\eps_1^2+\eps_2^2) \sum_{i=1}^2  |x-x_i|^{-p} \to 0 \text{ as } \eps_1, \eps_2 \downarrow 0. 
 \end{align}
Turning to $K_1$, we first recall
 \begin{align}\label{ce5.4.2}
K_1&=\int   \int_0^\infty  \frac{1}{\eps_1^{p-2}}  \frac{1}{\eps_2^{p-2}} U_1^{\vec{\lambda},\vec{x},\vec{\eps}}(B_t) U_2^{\vec{\lambda},\vec{x},\vec{\eps}}(B_t)\nn \\
&\quad \quad \quad \quad \exp\Big(-\int_0^{t} U^{\vec{\lambda},\vec{x},\vec{\eps}}(B_s) ds\Big)1_{(t\leq T_\eps)}  dt dP_x.
 \end{align}
  We know $T_\eps \to \infty$ as $\eps_1, \eps_2 \downarrow 0$ since Brownian motion in  $d\geq 2$ doesn't hit points. By Proposition \ref{p3.1} and Lemma \ref{l7.5}, for ${Leb}\times P_x$-a.e. $(t,\omega)$, we have
 \begin{align}\label{ae10.2}
\lim_{\eps_1, \eps_2 \downarrow 0} &\frac{1}{\eps_1^{p-2}}  \frac{1}{\eps_2^{p-2}} U_1^{\vec{\lambda},\vec{x},\vec{\eps}}(B_t) U_2^{\vec{\lambda},\vec{x},\vec{\eps}}(B_t) \exp\Big(-\int_0^{t} U^{\vec{\lambda},\vec{x},\vec{\eps}}(B_s) ds\Big)1(t\leq T_\eps)\nn\\
&=C_{\ref{p3.1}}(\lambda_1) C_{\ref{p3.1}}(\lambda_2)U_1^{\vec{\infty},\vec{x}}(B_t)U_2^{\vec{\infty},\vec{x}}(B_t)\exp\Big(-\int_0^{t} V^{\vec{\infty},\vec{x}}(B_s) ds\Big).
 \end{align}
Recall the definition of $U_{1,2}^{\vec{\infty},\vec{x}}(x)$ as in \eqref{eu12}. If we can find an integrable bound for the left-hand side term of \eqref{ae10.2},  by Dominated Convergence we can conclude from \eqref{ce5.4.2} and \eqref{ae10.2} that 
 \begin{align}\label{ae10.2.9}
\lim_{\eps_1, \eps_2 \downarrow 0} K_1=C_{\ref{p3.1}}(\lambda_1) C_{\ref{p3.1}}(\lambda_2) (-U_{1,2}^{\vec{\infty},\vec{x}}(x)),
 \end{align}
  and the proof will be finished by Lemma \ref{l4.4}, \eqref{m5.1}, \eqref{e4.7} and \eqref{ae10.2.9}. 

 It suffices to find an integrable bound for the left-hand side term of \eqref{ae10.2}. Recall \eqref{mc1.0.1}
and \eqref{e4.11}  to see that
\begin{align}\label{e4.14}
& \frac{1}{\eps_1^{p-2}}  \frac{1}{\eps_2^{p-2}} U_1^{\vec{\lambda},\vec{x},\vec{\eps}}(B_t) U_2^{\vec{\lambda},\vec{x},\vec{\eps}}(B_t)\exp\Big(-\int_0^{t} U^{\vec{\lambda},\vec{x},\vec{\eps}}(B_s) ds\Big)1(t\leq T_\eps) \\
 \leq & |B_t-x_1|^{-p}|B_t-x_2|^{-p}  \exp\Big(-\int_0^{t} U^{\vec{\lambda},\vec{x},\vec{\eps}}(B_s) ds\Big)1(t\leq T_\eps) \nn \\
 \leq &\sum_{i=1}^2 |B_t-x_i|^{-p}|B_t-x_{3-i}|^{-p}  1(|B_t-x_i|\leq |B_t-x_{3-i}|) \nn\\
 &\quad \quad \quad \quad \quad \quad   \exp\Big(-\int_0^{t} V^\infty(B_s-x_i) ds\Big) \nn \\
  \leq &2^p\sum_{i=1}^2  |B_t-x_i|^{-p} (|B_t-x_i|^{-p}\wedge \Delta^{-p})  \exp\Big(-\int_0^{t} V^\infty(B_s-x_i) ds\Big).\nn
  \end{align}
  where in the last inequality we have used $|B_t-x_{3-i}|\geq (|B_t-x_i|\vee (\Delta/2))$ on $\{|B_t-x_i|\leq |B_t-x_{3-i}|\}$. 
  
  It remains to show that for $i=1,2$,
 \begin{align}\label{2.2}
I_i:&=\int   \int_0^\infty |B_t-x_i|^{-p}(|B_t-x_i|^{-p}\wedge \Delta^{-p})\nn\\
&\quad \quad \quad  \quad \quad \quad \exp\Big(-\int_0^{t} V^\infty(B_s-x_i) ds\Big) dt dP_x<\infty.
  \end{align}
%
 Let $r_\eps=2\eps$. For $i=1,2,$ by translation invariance and monotone convergence we have
    \begin{align}\label{1.3}
 I_i=&E_{x-x_i}\Big(\int_0^\infty |B_t|^{-p}(|B_t|^{-p}\wedge \Delta^{-p}) \exp\Big(-\int_0^{t} V^\infty(B_s) ds\Big) dt\Big)\nn\\
 =&\lim_{\eps\downarrow 0} E_{x-x_i}\Big(\int_0^{\tau_{r_\eps}} |B_t|^{-p}(|B_t|^{-p}\wedge \Delta^{-p}) \exp\Big(-\int_0^{t} V^\infty(B_s) ds\Big) dt\Big).
   \end{align}
By (S.18) and (S.20) of \cite{HMP18_supp}, we have
\begin{align}\label{e20.3}
E_{x-x_i}\Big(\int_0^{\tau_{r_\eps}} &|B_t|^{-p}(|B_t|^{-p}\wedge \Delta^{-p}) \exp\Big(\int_0^{t} \frac{2^p D^\lambda(2) \eps^{p-2}}{|B_s|^p} ds\Big)\nn\\
& \exp\Big(-\int_0^{t} \frac{2(4-d)}{|B_s|^2} ds\Big) dt\Big) \leq C\Delta^{2-p} |x-x_i|^{-p}, \forall \eps>0\text{ small, }
 \end{align}
where $D^\lambda(2)=U^{\infty,1}(2)-U^{\lambda,1}(2)$ with $\lambda>0$ large. Therefore we conclude from \eqref{1.3} and \eqref{e20.3} that $I_i<\infty$ and the proof of \eqref{2.2} is complete. It remains to prove \eqref{eu12b}. Recall the definition of $(-U_{1,2}^{\vec{\infty},\vec{x}}(x))$ as in \eqref{eu12}. By \eqref{eu1u2}, \eqref{ae8.10} and \eqref{e4.14}, it follows immediately from \eqref{1.3} and \eqref{e20.3}.
 \end{proof}

\bibliographystyle{plain}
\def\cprime{$'$}

\clearpage

\appendix

\section{Proof of Theorems \ref{t0.0.1} and \ref{tl1.0.0} under $\P_{X_0}$}\label{aa}

We deal with the case $\P_{X_0}$ for the general initial condition $X_0$ and recall $S(X_0)$ is the closed support of $X_0$. Recall $S(X_0)^{\geq \delta}=\{x: d(x,S(X_0))\geq \delta\}$ for any $\delta>0$, where $d(x, S(X_0))=\inf\{|x-y|: y\in S(X_0)\}$. Similarly we define $S(X_0)^{>\delta}, S(X_0)^{\leq \delta}$ and $S(X_0)^{<\delta}$.

We first give the convergence of $\cL^\lambda$ to $\cL$ and $\widetilde{\cL}(\kappa)^\eps$ to $\widetilde{\cL}(\kappa)$ and then find some constant $c_{\ref{tl1}}(\kappa)>0$ so that 
 $\widetilde{\cL}(\kappa)= c_{\ref{tl1}}(\kappa) \cL$ a.s. Next we show that the support of $\widetilde{\cL}(\kappa)$ is contained in $\pmR \cap S(X_0)^c$ and it follows that the support of $\cL$ will also be on $\pmR\cap S(X_0)^c$, thus finishing both proofs of Theorem \ref{t0.0.1} and Theorem \ref{tl1.0.0}. Since the proof for the convergence of $\cL^\lambda$ and $\widetilde{\cL}(\kappa)^\eps$ are similar, we will only give the proof for the latter. 

 Let $\{\phi_m\}_{m=1}^\infty$ be a countable determining class for $M_F(\R^d)$ consisting of bounded, continuous functions and we take $\phi_1=1$. 
Define  
\[\cC_{X_0}=\{\psi_{m,k}^{X_0}: \psi_{m,k}^{X_0}=\phi_m \cdot \chi_{k}^{X_0}, m\geq 1, k\geq 1\},\]
where $\chi_{k}^{X_0}$ is a continuous modification of $1_{B_k\cap  S(X_0)^{>1/k}}$ so that \mbox{$\chi_{k}^{X_0}(x)=1$} for all $x\in B_k\cap  S(X_0)^{>1/k}$ and $\chi_{k}^{X_0}(x)=0$ for all $x\in  S(X_0)^{<1/(2k)}$ or $|x|\geq k+1$ .
Corollary \ref{c1} implies that for any $\psi_{m,k}^{X_0}\in \cC_{X_0}$, we have $\widetilde{\cL}(\kappa)^\eps(\psi_{m,k}^{X_0})$ converges in $L^2(\P_{X_0})$ to some $\widetilde{l}(\psi_{m,k}^{X_0})$ and by taking a subsequence  we get almost sure convergence. Define subsequences iteratively and take a diagonal subsequence $\eps_n \downarrow 0$ (we may assume $0<\eps_n<1$ for all $n\geq 1$) to get 
\begin{align}\label{e1.5.5}
    \widetilde{\cL}(\kappa)^{\eps_n}(\psi_{m,k}^{X_0}) \to \widetilde{l}(\psi_{m,k}^{X_0}) \text{ as } \eps_n \downarrow 0, \text{ for all } m,k\geq 1,\ \P_{X_0}\text{-a.s.}
\end{align}


For any fixed $0<\delta<1$ we will consider the restriction of $\{\widetilde{\cL}(\kappa)^{\eps_n}\}$ to $S(X_0)^{\geq \delta}$ and we write $\widetilde{\cL}(\kappa)^{\eps_n}_\delta\equiv \widetilde{\cL}(\kappa)^{\eps_n}|_{S(X_0)^{\geq\delta}}$ (recall $\mu|_K(\cdot)=\mu(\cdot \cap K)$). 

First we use Corollary III.1.5 of \cite{Per02} to see that  with $\P_{X_0}$-probability one there is some $\beta'(\omega) \in (0,1] $ such that for all $0<t<\beta'$, the closed support of $X_t$ is within the region $\{x: d(x,S(X_0))<3(t\log (1/t))^{1/2}\}$. Pick $0<\beta<\beta'$ small enough so that $3(\beta\log (1/\beta))^{1/2}<\delta$ and hence 
\begin{align*}
\mR \cap S(X_0)^{\geq \delta}\subset \bigcup_{t\geq \beta} \text{Supp}(X_t). 
\end{align*}
By Corollary III.1.7 of \cite{Per02} we conclude from the above that for any $\delta>0$,
\begin{align}\label{de1.0.3}
\mR \cap S(X_0)^{\geq \delta} \text{ is bounded}, \quad \P_{X_0}-a.s. 
\end{align}

Next we claim that for any $0<\delta<1$ and any $\eps_n\downarrow 0$, with $\P_{X_0}$-probability one we have 
\begin{align}\label{e1.0.99}
\sup_{0<\eps_n<\delta/2} \widetilde{\cL}(\kappa)^{\eps_n}\Big((\mR\cap S(X_0)^{\geq \delta/2})^{>1}\cap S(X_0)^{\geq \delta}\Big)=0.
\end{align}
To see this, we fix $\eps<\delta/2<1$. For all $x\in (\mR\cap S(X_0)^{\geq \delta/2})^{>1}$, we have $\overline{B_\eps(x)}\subset  (\mR\cap S(X_0)^{\geq \delta/2})^c$. Next for all $x\in S(X_0)^{\geq \delta}$, we have $\overline{B_\eps(x)}\subset  S(X_0)^{\geq \delta/2}$ and in particular $\overline{B_\eps(x)}\subset  (\mR\cap S(X_0)^{\leq \delta/2})^c$. Therefore we conclude $x\in (\mR\cap S(X_0)^{\geq \delta/2})^{>1}\cap S(X_0)^{\geq \delta}$ would imply $\overline{B_\eps(x)} \subset \mR^c$, and by \eqref{ea1.1} we have $X_{G_\eps^x}(1)=0$. It then follows that for all $0<\eps_n<\delta/2$,
\begin{align*}
&\E_{X_0}\Big(\widetilde{\cL}(\kappa)^{\eps_n}\Big((\mR\cap S(X_0)^{\geq \delta/2})^{>1}\cap S(X_0)^{\geq \delta}\Big)\Big)\\
&\leq \E_{X_0}\Big(\int \frac{X_{G_\eps^x}(1)}{\eps^{p}}\exp\Big(-\kappa \frac{X_{G_{\eps}^{x}}(1)}{\eps^2}\Big) 1(\overline{B_\eps(x)} \subset \mR^c) dx \Big)\nn\\
&=\int\E_{X_0}\Big(\frac{X_{G_\eps^x}(1)}{\eps^{p}} \exp\Big(-\kappa \frac{X_{G_{\eps}^{x}}(1)}{\eps^2}\Big)1(\overline{B_\eps(x)} \subset \mR^c) \Big)dx=0,
\end{align*}
Thus we get \eqref{e1.0.99} by taking a countable union of null sets. 

Now use \eqref{de1.0.3} and \eqref{e1.0.99}  to see that with $\P_{X_0}$-probability one, for $M\geq 1$ large we have
\begin{align}\label{e1.5.4}
\sup_{0<\eps_n<\delta/2} \widetilde{\cL}(\kappa)^{\eps_n}_\delta &\Big(\{x: |x|\geq M\}\Big)\leq \sup_{0<\eps_n<\delta/2} \widetilde{\cL}(\kappa)^{\eps_n}_\delta\Big((\mR\cap S(X_0)^{\geq \delta/2})^{>1}\Big)\nn\\
=& \sup_{0<\eps_n<\delta/2} \widetilde{\cL}(\kappa)^{\eps_n}\Big((\mR\cap S(X_0)^{\geq \delta/2})^{>1}\cap S(X_0)^{\geq \delta}\Big)=0.
\end{align}
For any $M>1$, by using \eqref{e1.5.5} with $m=1$, we conclude with $\P_{X_0}$-probability one, for $k\geq 1$ large, we have
\begin{align*}
\sup_{0<\eps_n<\delta/2}\widetilde{\cL}(\kappa)^{\eps_n}_\delta(B_M)&= \sup_{0<\eps_n<\delta/2}\widetilde{\cL}(\kappa)^{\eps_n}(S(X_0)^{\geq \delta} \cap B_M)\\
&\leq \sup_{0<\eps_n<\delta/2} \widetilde{\cL}(\kappa)^{\eps_n}(\chi_{k}^{X_0})<\infty,
\end{align*}
Together with \eqref{e1.5.4}, we have
\begin{align}\label{e1.5.2}
\sup_{0<\eps_n<\delta/2} \widetilde{\cL}(\kappa)^{\eps_n}_\delta(1)&\leq \sup_{0<\eps_n<\delta/2} \widetilde{\cL}(\kappa)^{\eps_n}_\delta(B_M)\nn\\
&\quad \quad  +\sup_{0<\eps_n<\delta/2} \widetilde{\cL}(\kappa)^{\eps_n}_\delta \Big(\{x: |x|\geq M\}\Big)<\infty.
\end{align}
Note \eqref{e1.5.4} also implies the tightness of $\{\widetilde{\cL}(\kappa)^{\eps_n}_\delta, 0<\eps_n<\delta/2\}$ and together with \eqref{e1.5.2}, we get the relative compactness of $\{\widetilde{\cL}(\kappa)^{\eps_n}_\delta, 0<\eps_n<\delta/2\}$ by Prohorov's theorem (see, e.g., Theorem 7.8.7 of \cite{AD00}). Therefore any subsequence admits a further sequence along which the measures converge to some $\widetilde{l}(\kappa)_\delta$ supported on $S(X_0)^{\geq \delta}$ in the weak topology. It remains to check all limit point coincide which is easy to see by \eqref{e1.5.5} since $\cC_{X_0}$ is a determining class on $M_F(S(X_0)^{\geq \delta}).$ Therefore for any $\delta>0$, under $\P_{X_0}$ we have $\widetilde{\cL}(\kappa)^{\eps_n}_\delta \overset{P}{\rightarrow} \widetilde{l}(\kappa)_\delta$ as $\eps \downarrow 0$.

Note by definition, $\widetilde{l}(\kappa)_{\delta'}$ and $\widetilde{l}(\kappa)_\delta$ agree on $S(X_0)^{\geq \delta}$ for all $\delta\geq \delta'>0$. Take $\delta=1/k$ and define a $\sigma$-finite measure $\widetilde{\cL}(\kappa)$ on $S(X_0)^{c}$ by 
\begin{align}
\widetilde{\cL}(\kappa)|_{S(X_0)^{\geq 1/k}}\equiv \widetilde{l}(\kappa)_{1/k}, \forall k\geq 1.
\end{align}
Thus we conclude $\widetilde{\cL}(\kappa)^\eps|_{S(X_0)^{\geq 1/k}} \overset{P}{\rightarrow} \widetilde{\cL}(\kappa)|_{S(X_0)^{\geq 1/k}}$ as $\eps\downarrow 0$ under $\P_{X_0}$ for all $k\geq 1$ and by taking a diagonal subsequence, we can find some sequence $\eps_n \downarrow 0$ so that $\widetilde{\cL}(\kappa)^{\eps_n}|_{S(X_0)^{\geq 1/k}} \to \widetilde{\cL}(\kappa)|_{S(X_0)^{\geq 1/k}}, \forall k\geq 1$ a.s. as $n \to \infty$.

With the construction of $\widetilde{\cL}(\kappa)$, and by a similar argument for the construction of $\cL$ complete under $\P_{X_0}$, we now show $\P_{X_0}$-a.s. that $\widetilde{\cL}(\kappa)=c_{\ref{tl1}}(\kappa){\cL}$. By the above construction, it suffices to show that for any $k\geq 1$, we have $\P_{X_0}$-a.s. that $ \widetilde{\cL}(\kappa)|_{S(X_0)^{\geq 1/k}}=c_{\ref{tl1}}(\kappa){\cL}|_{S(X_0)^{\geq 1/k}}$.

 Similar to the derivation of \eqref{e9.4.1}, by Corollary \ref{c1} and Corollary \ref{c4}, we can get  $\P_{X_0}$-a.s. that $C_{\ref{p3.1}}(\kappa){\cL}(\psi_{m,k}^{X_0})=K_{\ref{p3.1}}\widetilde{\cL}(\kappa)(\psi_{m,k}^{X_0})$ for all $m,k\geq 1$ and so we have $C_{\ref{p3.1}}(\kappa){\cL}|_{S(X_0)^{\geq 1/k}}=K_{\ref{p3.1}}\widetilde{\cL}(\kappa)|_{S(X_0)^{\geq 1/k}}$ for any $k\geq 1$. Then it follows that $\P_{X_0}$-a.s. $\widetilde{\cL}(\kappa)=c_{\ref{tl1}}(\kappa){\cL}$ as noted above.
 
Finally by using Proposition \ref{p4.0}, one can show that $\widetilde{\cL}(\kappa)$ (and hence $\cL$) is supported on $\pmR$ in a similar way to the proof of Theorem \ref{t0} under $\N_0$ in Section \ref{s5.2}. The construction of $\widetilde{\cL}(\kappa)$ will then give us that $\widetilde{\cL}(\kappa)$ is supported on $\pmR \cap S(X_0)^c$. The proof is then complete.

\section{Proof of Lemmas \ref{c13.5}, \ref{l7.5} and \ref{al4.0} }\label{ab}

\subsection{Proof of Lemma \ref{c13.5}}

The scalings of Bessel process $\rho_s$ and $V^\infty, V^\lambda$ give us that
  \begin{align}\label{aec1.0.1}
 &E_{|x|}^{(2+2\nu)}\Big(\exp\Big(\gamma \int_0^{\tau_{r_\lambda}} (V^\infty-V^{\lambda})(\rho_s) ds\Big)\Big|\tau_{r_\lambda}<\infty\Big)\nn\\
 =&E_{|x|/r_\lambda}^{(2+2\nu)}\Big(\exp\Big(\gamma \int_0^{ \tau_{1}} (V^\infty-V^{\lambda r_\lambda^{4-d}})(\rho_s) ds\Big)\Big|\tau_{1}<\infty\Big)\nn\\
  =&E_{|x|/r_\lambda}^{(2+2\nu)}\Big(\exp\Big(\gamma \int_0^{ \tau_{1}} (V^\infty-V^{\lambda_0^{4-d}})(\rho_s) ds\Big)\Big|\tau_{1}<\infty\Big),
 \end{align}
 where we have used $r_\lambda=\lambda_0 \lambda^{-\frac{1}{4-d}}$ in the last line. For any $r>1$, we let 
 \begin{align}\label{e20.2}
 f(r):=&E_{r}^{(2+2\nu)}\Big(\exp\Big(\gamma \int_0^{ \tau_{1}} (V^\infty-V^{\lambda_0^{4-d}})(\rho_s) ds\Big)\Big|\tau_{1}<\infty\Big)\nn\\
 =&E_{r}^{(2+2\nu)}\Big(\exp\Big(\gamma \int_0^{ \tau_{1}} (V^\infty-V^{\lambda_0^{4-d}})(\rho_s) ds\Big)1(\tau_{1}<\infty)\Big) r^{2\nu},
 \end{align}
where the second line is by $P_{r}^{(2+2\nu)}(\tau_R<\infty)=(R/r)^{2\nu}$ for any $r>R>0$. 
By \eqref{aec1.0.1} and the definition of $r_\lambda$, it suffices to show that there is some constant $C_{\ref{c13.5}}(\lambda_0, \nu,\gamma)>0$ so that $\sup_{r>1} f(r)=\lim_{r\to \infty} f(r)=C_{\ref{c13.5}}(\lambda_0, \nu,\gamma)$.

Let $r>R>1$ and apply the strong Markov Property in \eqref{e20.2} to get
\begin{align}\label{ae7.0}
f(r)&=E_{r}^{(2+2\nu)}\Big(\exp\Big(\gamma \int_0^{ \tau_{R}} (V^\infty-V^{\lambda_0^{4-d}})(\rho_s) ds\Big)1(\tau_{R}<\infty)\Big)\nn\\
&\quad \quad \quad E_{R}^{(2+2\nu)}\Big(\exp\Big(\gamma \int_0^{ \tau_{1}} (V^\infty-V^{\lambda_0^{4-d}})(\rho_s) ds\Big)1(\tau_{1}<\infty)\Big)  r^{2\nu}\nn\\
&=E_{r}^{(2+2\nu)}\Big(\exp\Big(\gamma \int_0^{ \tau_{R}} (V^\infty-V^{\lambda_0^{4-d}})(\rho_s) ds\Big)\Big|\tau_{R}<\infty\Big)\nn\\
&\quad \quad \quad E_{R}^{(2+2\nu)}\Big(\exp\Big(\gamma \int_0^{ \tau_{1}} (V^\infty-V^{\lambda_0^{4-d}})(\rho_s) ds\Big)\Big|\tau_{1}<\infty\Big)\geq f(R),
\end{align}
and it follows that $r\mapsto f(r)$ is monotone increasing for $r>1$. By using Lemma \ref{l12} and Lemma \ref{l13.5} (ii), we have
   \begin{align*}
    \sup_{r\geq 1} f(r)  \leq  &  \sup_{r\geq 1}E_{r}^{(2+2\nu)}\Big(\exp\Big( \int_0^{ \tau_{1}} c_{\ref{l12}} \gamma \lambda_0^{-(p-2)} \rho_s^{-p} ds\Big)\Big|\tau_{1}<\infty\Big)<\infty,
 \end{align*}
 if we choose $\lambda_0$ large enough so that $2\gamma c_{\ref{l12}} \lambda_0^{-(p-2)} \leq 4c_{\ref{l12}} \lambda_0^{-(p-2)}<\nu^2$. Hence we conclude
 $\sup_{r\geq 1} f(r)=\lim_{r\to \infty} f(r)=C_{\ref{c13.5}}(\lambda_0, \nu, \gamma)$ for some constant $C_{\ref{c13.5}}(\lambda_0, \nu, \gamma)>0$ and the proof is complete as noted above.

\subsection{Proof of Lemma \ref{l7.5}}
Recall from \eqref{ev} that \[V^{\infty,\vec{x}}(x)=\N_x\Big(\{L^{x_1}>0\} \cup \{L^{x_2}>0\} \Big)<\infty,\]where the finiteness is by $\N_0(L^x>0)=V^\infty(x)<\infty$. Therefore by \eqref{e9.1} and monotone convergence theorem we have
\begin{align*}
&V^{\infty,\vec{x}}(x)-V^{\vec{\lambda},\vec{x}}(x)=\N_x\Big(1-1_{\{L^{x_1}=0\} \cap \{L^{x_2}=0\}}\Big)-\N_x\Big(1-e^{-\lambda_1 L^{x_1}-\lambda_2 L^{x_2}}\Big)\\
&=\N_x\Big(e^{-\lambda_1 L^{x_1}-\lambda_2 L^{x_2}}1_{\{L^{x_1}>0\} \cup \{L^{x_2}>0\}} \Big)\to 0 \text{ as } \lambda_1, \lambda_2 \to \infty. 
\end{align*}
Before turning to the proof for $U^{\vec{\lambda},\vec{x},\vec{\eps}}$ and $W^{\vec{\lambda},\vec{x},\eps}$, we first note that for all $x$ so that $|x-x_i|>\eps_i$, $i=1,2$, we have 
\begin{align}\label{ae1.1}
&K_i\equiv \N_x\Big(\Big(\exp\Big(-\lambda_i \frac{X_{G_{\eps_i}^{x_i}}(1)}{\eps_i^2}\Big)1(X_{G_{\eps_i/2}^{x_i}}=0)-1_{\{L^{x_i}=0\}} \Big)^2\Big)\\
&=\N_x\Big(\Big(\exp\Big(-2\lambda_i \frac{X_{G_{\eps_i}^{x_i}}(1)}{\eps_i^2}\Big)1(X_{G_{\eps_i/2}^{x_i}}=0)\nn\\
&\quad \quad -2\exp\Big(-\lambda_i \frac{X_{G_{\eps_i}^{x_i}}(1)}{\eps_i^2}\Big) 1(X_{G_{\eps_i/2}^{x_i}}=0)1_{\{L^{x_i}=0\}} +1_{\{L^{x_i}=0\}}\Big)\nn\\
&=\N_x\Big(\Big(\exp\Big(-2\lambda_i \frac{X_{G_{\eps_i}^{x_i}}(1)}{\eps_i^2}\Big)\P_{X_{G_{\eps_i}^{x_i}}}(X_{G_{\eps_i/2}^{x_i}}=0)\nn\\
& \quad -2\exp\Big(-\lambda_i \frac{X_{G_{\eps_i}^{x_i}}(1)}{\eps_i^2}\Big)\P_{X_{G_{\eps_i}^{x_i}}}(X_{G_{\eps_i/2}^{x_i}}=0, L^{x_i}=0) +\P_{X_{G_{\eps_i}^{x_i}}}(L^{x_i}=0)\Big),\nn
\end{align}
where the last follows by Proposition \ref{pv0.2}(i).
Apply  \eqref{ev1.2} to see that 
\[\P_{X_{G_{\eps_i}^{x_i}}}(L^{x_i}=0)=\exp\Big(-V^\infty(\eps_i) X_{G_{\eps_i}^{x_i}}(1)\Big)=\exp\Big(-\lambda_d \frac{X_{G_{\eps_i}^{x_i}}(1)}{\eps_i^2}\Big),\]
and as in the derivation of \eqref{er1.5}, we have
\[\P_{X_{G_{\eps_i}^{x_i}}}(X_{G_{\eps_i/2}^{x_i}}=0)=\exp\Big(-4U^{\infty,1}(2) \frac{X_{G_{\eps_i}^{x_i}}(1)}{\eps_i^2}\Big).\]
Use Proposition \ref{pv0.1}(i) to get
\begin{align*}
\P_{X_{G_{\eps_i}^{x_i}}}(X_{G_{\eps_i/2}^{x_i}}=0, L^{x_i}=0)
=&\P_{X_{G_{\eps_i}^{x_i}}}\Big(1(X_{G_{\eps_i/2}^{x_i}}=0)\P_{X_{G_{\eps_i/2}^{x_i}}}(L^{x_i}=0)\Big)\\
=&\P_{X_{G_{\eps_i}^{x_i}}}(X_{G_{\eps_i/2}^{x_i}}=0).
\end{align*}
Returning to \eqref{ae1.1}, we have
\begin{align}\label{e10.1.1}
K_i=&\N_x\Big(\exp\Big(-(2\lambda_i+4U^{\infty,1}(2)) \frac{X_{G_{\eps_i}^{x_i}}(1)}{\eps_i^2}\Big)\\
&\quad  -2\exp\Big(-(\lambda_i+4U^{\infty,1}(2)) \frac{X_{G_{\eps_i}^{x_i}}(1)}{\eps_i^2}\Big)+\exp\Big(-\lambda_d \frac{X_{G_{\eps_i}^{x_i}}(1)}{\eps_i^2}\Big) \Big)\nn\\
\leq&\N_x\Big(\exp\Big(-\lambda_d \frac{X_{G_{\eps_i}^{x_i}}(1)}{\eps_i^2}\Big)-\exp\Big(-(\lambda_i+4U^{\infty,1}(2)) \frac{X_{G_{\eps_i}^{x_i}}(1)}{\eps_i^2}\Big)\Big)\nn\\
=&U^{(\lambda_i+4U^{\infty,1}(2))  \eps_i^{-2}, \eps_i}(x-x_i)-U^{\lambda_d \eps_i^{-2}, \eps_i}(x-x_i)\to 0\text{ as } \eps_i \downarrow 0,\nn
\end{align}
where the equality is by \eqref{e1.2.5} and the last follows from \eqref{ae1.4}.\\

 Turning to $U^{\vec{\lambda},\vec{x},\vec{\eps}}(x)$, for $\eps_1, \eps_2>0$ small enough, by definition we have
\begin{align*}
I=&U^{\vec{\lambda},\vec{x},\vec{\eps}}(x)-V^{\infty,\vec{x}}(x)\\
=&\N_{x} \Big(1-\prod_{i=1}^2\exp\Big(- \lambda_i \frac{X_{G_{\eps_i}^{x_i}}(1)}{\varepsilon_i^2}\Big)1(X_{G_{\eps_i/2}^{x_i}}=0)\Big)-\N_x\Big(1-1_{\{L^{x_1}=0\}} 1_{\{L^{x_2}=0\}}\Big)\\
=&\N_x\Big(1_{\{L^{x_1}=0\}} 1_{\{L^{x_2}=0\}}-\prod_{i=1}^2 \exp\Big(- \lambda_i \frac{X_{G_{\eps_i}^{x_i}}(1)}{\varepsilon_i^2}\Big)1(X_{G_{\eps_i/2}^{x_i}}=0) \Big).
\end{align*}
By Jensen's inequality we have 
\begin{align*}
I^2\leq &\N_x\Big(\Big(\prod_{i=1}^2 \exp\Big(- \lambda_i \frac{X_{G_{\eps_i}^{x_i}}(1)}{\varepsilon_i^2}\Big)1(X_{G_{\eps_i/2}^{x_i}}=0)-1_{\{L^{x_1}=0\}} 1_{\{L^{x_2}=0\}}\Big)^2\Big)\\
\leq &2\N_x\Big(\Big(\prod_{i=1}^2 \exp\Big(- \lambda_i \frac{X_{G_{\eps_i}^{x_i}}(1)}{\varepsilon_i^2}\Big)1(X_{G_{\eps_i/2}^{x_i}}=0)\\
 &\quad \quad \quad \quad -1_{\{L^{x_1}=0\}}\exp\Big(-\lambda_2 \frac{X_{G_{\eps_2}^{x_2}}(1)}{\varepsilon_2^2}\Big)1(X_{G_{\eps_2/2}^{x_2}}=0) \Big)^2\Big)\\
 +&2\N_x\Big(\Big(1_{\{L^{x_1}=0\}}\exp\Big(-\lambda_2 \frac{X_{G_{\eps_2}^{x_2}}(1)}{\varepsilon_2^2}\Big)1(X_{G_{\eps_2/2}^{x_2}}=0)-1_{\{L^{x_1}=0\}} 1_{\{L^{x_2}=0\}}\Big)^2\Big)
\end{align*}
where the last inequality is by $(a+b)^2 \leq 2a^2+2b^2, \forall a,b\in \R$. Then we have
\begin{align*}
I^2&\leq 2\N_x\Big(\Big(\exp\Big(-\lambda_1 \frac{X_{G_{\eps_1}^{x_1}}(1)}{\varepsilon_1^2}\Big)1(X_{G_{\eps_1/2}^{x_1}}=0)-1_{\{L^{x_1}=0\}} \Big)^2\Big)\\
 &+ 2\N_x\Big(\Big(\exp\Big(-\lambda_2 \frac{X_{G_{\eps_2}^{x_2}}(1)}{\varepsilon_2^2}\Big)1(X_{G_{\eps_2/2}^{x_2}}=0)-1_{\{L^{x_2}=0\}} \Big)^2\Big)=2K_1+2K_2\to 0
\end{align*}
as $\eps_1, \eps_2 \downarrow 0$ where we have used \eqref{e10.1.1} in the last line.

 Turning to $W^{\vec{\lambda},\vec{x},{\eps}}(x)$,  for $\eps>0$ small enough we have
\begin{align*}
&J=W^{\vec{\lambda},\vec{x},\eps}(x)-V^{\infty,\vec{x}}(x)\\
&=\N_{x} \Big(1-e^{-\lambda_1 L^{x_1}}\exp\Big(-\lambda_2 \frac{X_{G_{\eps}^{x_2}}(1)}{\eps^2} \Big)1_{(X_{G_{\eps/2}^{x_2}}=0)}\Big)-\N_x\Big(1-1_{\{L^{x_1}=0\}} 1_{\{L^{x_2}=0\}}\Big)\\
&=\N_x\Big(1_{\{L^{x_1}=0\}} 1_{\{L^{x_2}=0\}}-e^{-\lambda_1 L^{x_1}}\exp\Big(-\lambda_2 \frac{X_{G_{\eps}^{x_2}}(1)}{\eps^2} \Big)1(X_{G_{\eps/2}^{x_2}}=0)\Big).
\end{align*}
By Jensen's inequality we have 
\begin{align*}
J^2\leq &\N_x\Big(\Big(e^{-\lambda_1 L^{x_1}}\exp\Big(-\lambda_2 \frac{X_{G_{\eps}^{x_2}}(1)}{\eps^2} \Big)1(X_{G_{\eps/2}^{x_2}}=0)-1_{\{L^{x_1}=0\}} 1_{\{L^{x_2}=0\}}\Big)^2\Big)\\
\leq &2\N_x\Big(\Big(e^{-\lambda_1 L^{x_1}}\exp\Big(-\lambda_2 \frac{X_{G_{\eps}^{x_2}}(1)}{\eps^2} \Big)1(X_{G_{\eps/2}^{x_2}}=0)\\
&\quad \quad \quad \quad -1_{\{L^{x_1}=0\}}\exp\Big(-\lambda_2 \frac{X_{G_{\eps}^{x_2}}(1)}{\varepsilon^2}\Big)1(X_{G_{\eps/2}^{x_2}}=0)\Big)^2\Big)\\
 +&2\N_x\Big(\Big(1_{\{L^{x_1}=0\}}\exp\Big(-\lambda_2 \frac{X_{G_{\eps}^{x_2}}(1)}{\varepsilon^2}\Big)1(X_{G_{\eps/2}^{x_2}}=0)-1_{\{L^{x_1}=0\}} 1_{\{L^{x_2}=0\}}\Big)^2\Big)\\
 \leq &2\N_x\Big(\Big(e^{-\lambda_1 L^{x_1}}-1_{\{L^{x_1}=0\}} \Big)^2\Big)\\
 &\quad \quad \quad \quad + 2\N_x\Big(\Big(\exp\Big(-\lambda_2 \frac{X_{G_{\eps}^{x_2}}(1)}{\varepsilon^2}\Big)1_{(X_{G_{\eps/2}^{x_2}}=0)}-1_{\{L^{x_2}=0\}} \Big)^2\Big)\\
 \leq &2 \N_x\Big(e^{-2\lambda_1 L^{x_1}} 1_{\{L^{x_1}>0\}} \Big)+2K_2 \to 0,
\end{align*}
as $\lambda_1 \to \infty$ and $\eps \downarrow 0$ where we have used monotone convergence theorem and \eqref{e10.1.1} in the last line.

\subsection{Proof of Lemma \ref{al4.0}}
Recall $G=G_{\eps_1}^{x_1} \cap G_{\eps_2}^{x_2}$. For all $x\in G$ we let \[u(x)\equiv  U^{\vec{\lambda},\vec{x},\vec{\eps}}=\N_x\Big(1-\prod_{i=1}^2 \exp\Big(-\lambda_i \frac{X_{G_{\eps_i}^{x_i}}(1)}{\varepsilon_i^2}\Big)1(X_{G_{\eps_i/2}^{x_i}}=0) \Big).\]
Define 
\begin{align}\label{ae5.2}
\widetilde{\lambda}_i=\lambda_i+4U^{\infty,1}(2), i=1,2,
\end{align}
and recall \eqref{eb1.1.4} to get for all $x\in G$,
\begin{align}\label{ae4.02}
u(x)&\leq U^{\widetilde{\lambda}_1 \eps_1^{-2},\eps_1}(x-x_1)+U^{\widetilde{\lambda}_2 \eps_2^{-2},\eps_2}(x-x_2)\leq \widetilde{\lambda}_1 \eps_1^{-2}+\widetilde{\lambda}_2 \eps_2^{-2},
\end{align}
where the last inequality follows from that $r\mapsto U^{\lambda,\eps}(r)$ is decreasing by Lemma 3.2(b) of \cite{MP17} and that $U^{\lambda,\eps}(\eps)=\lambda$. Next, for any $x'\in G$, let $D$ be an open ball that contains $x'$, whose closure is in $G$. Use Proposition \ref{pv0.2}(i) to see that for $x\in D$,
\begin{align*}
&u(x)=\N_x\Big(1-\prod_{i=1}^2 \exp\Big(-\lambda_i \frac{X_{G_{\eps_i}^{x_i}}(1)}{\varepsilon_i^2}\Big)1(X_{G_{\eps_i/2}^{x_i}}=0) \Big)\\
=&\N_x\Big(1-\E_{X_D} \Big( \prod_{i=1}^2 \exp\Big(-\lambda_i \frac{X_{G_{\eps_i}^{x_i}}(1)}{\varepsilon_i^2}\Big)1(X_{G_{\eps_i/2}^{x_i}}=0)\Big)\Big)\\
=& \N_x\Big(1- \exp\Big(-\int  u(y) X_D(dy)  \Big)\Big),
\end{align*}
 the last  equality by \eqref{e4.3.1} with $X_0=X_D$. Therefore 
 \[u(x)=\N_x\Big( 1-\exp\Big(-\int  u(y) X_D(dy)\Big)\Big),\quad\forall x\in D.\] 
 Note $u$ is bounded in $G$ by \eqref{ae4.02}, and hence on $\partial D$. Use Theorem V.6 of \cite{Leg99} to conclude \[\Delta u(x)=(u(x))^2, \ \forall x\in D,\text{ and, in particular, for }x=x'.\] 
 Since $x'$ is arbitrary, it holds for all $x\in G$.

\section{Proof of Propositions \ref{p3.1}(i) and \ref{p3.3}(i)}\label{ac}

\begin{proof}[Proof of Proposition \ref{p3.1}(i)] By symmetry it suffices to consider the case $i=1$.
Recall Lemma \ref{l10.3} to get
  \begin{align*}
\lambda_1^{1+\alpha} V_1^{\vec{\lambda},\vec{x}}(x)=&\lambda_1^{1+\alpha}\lim_{t\to \infty}E_x\Big( V_1^{\vec{\lambda},\vec{x}}(B(t\wedge T_{r_{\lambda}}))\exp\Big(-\int_0^{t\wedge T_{r_{\lambda}}} V^{\vec{\lambda},\vec{x}}(B_s) ds\Big)\Big).
  \end{align*}
where $T_{r_{\lambda}}=T_{r_{\lambda_1}}^1 \wedge T_{r_{\lambda_2}}^2$ and $T_{r_{\lambda_i}}^i=\inf\{t\geq 0: |B_t-x_i|\leq r_{\lambda_i} \}$. Here $r_{\lambda_i}=\lambda_0 \lambda_i^{-\frac{1}{4-d}}$ and we will choose $\lambda_0$ to be some fixed large constant below. By \eqref{m10.0}, we have $V_1^{\vec{\lambda},\vec{x}}(x) \to 0$ as $|x|\to \infty$ and $V_1^{\vec{\lambda},\vec{x}}(B(t\wedge T_{r_{\lambda}}))$ is uniformly bounded for all $t\geq 0$. 
Apply Dominated Convergence to see that
  \begin{align}\label{m1.1}
\lambda_1^{1+\alpha} V_1^{\vec{\lambda},\vec{x}}(x)=&\lambda_1^{1+\alpha} E_x\Big(1_{\{T_{r_{\lambda}}<\infty\}} V_1^{\vec{\lambda},\vec{x}}(B(T_{r_{\lambda}}))\exp\Big(-\int_0^{T_{r_{\lambda}}} V^{\vec{\lambda},\vec{x}}(B_s) ds\Big)\Big)\nn\\
 =& \sum_{i=1}^2 E_x\Big(1_{\{T_{r_{\lambda_i}}^i<\infty\}} 1_{\{T_{r_{\lambda_i}}^i<T_{r_{\lambda_{3-i}}}^{3-i}\}} \lambda_1^{1+\alpha}V_1^{\vec{\lambda},\vec{x}}(B(T_{r_{\lambda_i}}^i))\nn\\
 &\quad\quad \quad \quad \quad \exp\Big(-\int_0^{T_{r_{\lambda_i}}^i} V^{\vec{\lambda},\vec{x}}(B_s) ds\Big)\Big):=I_1+I_2,
  \end{align}
  We first deal with $I_2$. Note in the integrand of $I_2$ we may assume that $|B(T_{r_{\lambda_2}}^2)-x_2|=r_{\lambda_2}$ and so by \eqref{5.1} we have \mbox{$|x_1-B(T_{r_{\lambda_2}}^2)|>\Delta/2$} where $\Delta=|x_1-x_2|$. Apply \eqref{m10.0} with $x=B(T_{r_{\lambda_2}}^2)$ to get
\begin{align}\label{e10.4}
  &\lambda_1^{1+\alpha}V_1^{\vec{\lambda},\vec{x}}(B(T_{r_{\lambda_2}}^2) )\leq c_{\ref{p1.1}}  |B(T_{r_{\lambda_2}}^2)-x_1|^{-p}\leq c_{\ref{p1.1}} \Delta^{-p} 2^p.
    \end{align}
    Let $\tau_r=\inf\{t\geq 0: |B_t|\leq r\}$ and use \eqref{e10.4} and \eqref{e10.2} to see that  $I_2$ becomes
   \begin{align}\label{m1.2}
   I_2&\leq   c_{\ref{p1.1}}2^p \Delta^{-p} E_x\Big(1_{\{T_{r_{\lambda_2}}^2<\infty\}} 1_{\{T_{r_{\lambda_2}}^2<T_{r_{\lambda_{1}}}^{1}\}} \exp\Big(-\int_0^{T_{r_{\lambda_2}}^2} V^{\vec{\lambda},\vec{x}}(B_s) ds\Big)\Big)\\
      &\leq  c_{\ref{p1.1}} 2^p \Delta^{-p} E_{x-x_2}\Big(1_{\{\tau_{r_{\lambda_2}}<\infty\}}  \exp\Big(-\int_0^{\tau_{r_{\lambda_2}}} V^{\lambda_2}(B_s) ds\Big)\Big)\nn\\
      &=  c_{\ref{p1.1}} 2^p \Delta^{-p} r_{\lambda_2}^p |x-x_2|^{-p} E_{|x-x_2|}^{(2+2\nu)}\Big( \exp\Big(\int_0^{\tau_{r_{\lambda_2}}} (V^\infty-V^{\lambda_2})(\rho_s) ds\Big)\Big|\tau_{r_{\lambda_2}}<\infty\Big)\nn\\
&\leq c_{\ref{p1.1}} 2^p \Delta^{-p} r_{\lambda_2}^p |x-x_2|^{-p} C_{\ref{c13.5}}(\lambda_0, \nu,1) \to 0 \text{ as }  \lambda_2 \to \infty,\nn
   \end{align}
   where we have used Proposition \ref{p20.1} in the equality and we choose $\lambda_0> c_{\ref{c13.5}}$ to apply Lemma \ref{c13.5} in the last inequality.
   
   Now we will turn to $I_1$. Let $(Y_t, t\geq 0)$ be the $d$-dimensional coordinate process under Wiener measure, $P_x$. By slightly abusing the notation, we set $\tau_r=\tau_r^Y=\inf\{t\geq 0: |Y_t|\leq r\}$ for any $r>0$, and set
     \begin{align}\label{e10.5}
     T^{'}_{r_{\lambda_2}}= T^{',Y}_{r_{\lambda_2}}=\inf\{t\geq 0: |Y_t-(x_2-x_1)|\leq r_{\lambda_2}\}
     \end{align}
Then use translation invariance of $Y$ to get
  \begin{align*}
 I_1=&E_{x-x_1}\Big(1_{\{\tau_{r_{\lambda_1}}<\infty\}}  1_{\{\tau_{r_{\lambda_1}}<T^{'}_{r_{\lambda_2}}\}}  \lambda_1^{1+\alpha}V_1^{\vec{\lambda},\vec{x}}(Y({\tau_{r_{\lambda_1}}})+x_1)\nn \\
   &\quad \quad \quad \times \exp\Big(-\int_0^{\tau_{r_{\lambda_1}}} V^{\vec{\lambda},\vec{x}}(Y_s+x_1) ds\Big)\Big).
  \end{align*}
 Recall that $\Hat{P}_{x}^{(2-2\nu)}$ is the law of $Y$ starting from $x$ such that $Y$ is the unique solution of
  \begin{align}\label{e10.6}
       \begin{cases}
        &Y_t=x+\Hat{B}_t+\int_0^t (-\nu-\mu)\frac{Y_s}{|Y_s|^2}ds, \quad t< \tau_0,\\
        &Y_t=0, t\geq \tau_0,
        \end{cases}
  \end{align}
  where $\Hat{B}$ is a standard $d$-dimensional Brownian motion under $\Hat{P}_{x}^{(2-2\nu)}$. 
    Apply Proposition \ref{p8.1} with $g(\cdot)=V^{\vec{\lambda},\vec{x}}(\cdot+x_1)$ in the above to get
   \begin{align}\label{e10.7}
   I_1=&\frac{r_{\lambda_1}^{p}}{|x-x_1|^{p}} \Hat{E}^{(2-2\nu)}_{x-x_1}\Big(1_{\{\tau_{r_{\lambda_1}}<T^{'}_{r_{\lambda_2}}\}} \lambda_1^{1+\alpha}V_1^{\vec{\lambda},\vec{x}}(Y({\tau_{r_{\lambda_1}}})+x_1)\nn \\
   &\quad \quad \quad \times \exp\Big(-\int_0^{\tau_{r_{\lambda_1}}} (V^{\vec{\lambda},\vec{x}}(Y_s+x_1)-V^\infty(Y_s)) ds\Big)\Big)\nn\\
 =&\frac{1}{|x-x_1|^{p}} \Hat{E}^{(2-2\nu)}_{x-x_1}\Bigg([1_{\{\tau_{r_{\lambda_1}}<T^{'}_{r_{\lambda_2}}\}}]  [r_{\lambda_1}^{p} \lambda_1^{1+\alpha} V_1^{\vec{\lambda},\vec{x}}(Y({\tau_{r_{\lambda_1}}})+x_1)]\nn \\
   & \quad \times \Big[\exp\Big(-\int_0^{\tau_{r_{\lambda_1}}} (V^{\vec{\lambda},\vec{x}}(Y_s+x_1)-V^{\lambda_1}(Y_s)) ds\Big)\Big]\nn\\
   & \quad \times\Big[\exp\Big(\int_0^{\tau_{r_{\lambda_1}}} (V^\infty(Y_s)-V^{\lambda_1}(Y_s)) ds\Big)\Big]\Bigg)\nn\\
   :=&\frac{1}{|x-x_1|^{p}} \Hat{E}^{(2-2\nu)}_{x-x_1} ([\widetilde{J_1}] [\widetilde{J_2}] [\widetilde{J_3}] [\widetilde{J_4}]).
   \end{align}
where  we have ordered the fours terms in square brackets as $\widetilde{J_1}, \dots, \widetilde{J_4}$.

We first consider $\widetilde{J_2}$. Recall \eqref{e4.4.1} and use translation invariance to get 
\begin{align*}
 \widetilde{J_2}=&r_{\lambda_1}^{p} \lambda_1^{1+\alpha} \N_{Y({\tau_{r_{\lambda_1}}})+x_1}\Big(L^{x_1} \exp(-\lambda_1 L^{x_1}-\lambda_2 L^{x_2})\Big) \\
 =&r_{\lambda_1}^{p} \lambda_1^{1+\alpha}   \N_{Y({\tau_{r_{\lambda_1}}})}\Big(L^0 \exp(-\lambda_1 L^{0}-\lambda_2 L^{x_2-x_1})\Big).
\end{align*} 
By the scaling of Brownian snake and its local time $(L^{x_0})$ under the excursion measure $\N_x$ (see, e.g., Proof of Proposition V.9 in \cite{Leg99}), we have
\begin{align}\label{e10.8}
 \widetilde{J_2}=& r_{\lambda_1}^{p} \lambda_1^{1+\alpha}  r_{\lambda_1}^{-2} \N_{Y({\tau_{r_{\lambda_1}}})/r_{\lambda_1}}\Big(r_{\lambda_1}^{4-d} L^0 e^{- {\lambda_1} r_{\lambda_1}^{4-d} L^0}e^{-\lambda_2 r_{\lambda_1}^{4-d} L^{(x_2-x_1)/r_{\lambda_1}}} \Big)\nn \\
   =& \lambda_0^{p+2-d} \N_{Y({\tau_{r_{\lambda_1}}})/r_{\lambda_1}}\Big(L^0 e^{- \lambda_0^{4-d} L^0}e^{-\lambda_2 r_{\lambda_1}^{4-d} L^{(x_2-x_1)/r_{\lambda_1}} } \Big)\nn \\
\overset{\text{law}}{=}& \lambda_0^{p+2-d} \N_{Y({\tau_{1}})}\Big(L^0 e^{- \lambda_0^{4-d}L^0}e^{-\lambda_2 r_{\lambda_1}^{4-d} L^{(x_2-x_1)/r_{\lambda_1}} } \Big),
      \end{align} 
   where in the next to last equality we have used the definitions of $r_{\lambda_1}$ and $\alpha$ and the last equality follows from the scaling of $Y$.
Note for any $K>0$, we have \[\Big|\frac{x_2-x_1}{r_{\lambda_1}}\Big|>K \text{ for } \lambda_1 \text{ large enough, }\] and so by the compactness of the support of SBM (see \eqref{ea0.0}), we conclude $\N_{Y({\tau_{1}})}$-a.e.,
\[L^{(x_2-x_1)/r_{\lambda_1}}=0 \text{ for } \lambda_1 \text{ large enough.}\]
Therefore an application of Dominated Convergence Theorem will give us
  \begin{align}\label{e10.9}
 &\lim_{\lambda_1, \lambda_2 \to \infty}   \N_{Y({\tau_{1}})}\Big( L^0 \exp(-\lambda_0^{4-d} L^0-\lambda_2 r_{\lambda_1}^{4-d} L^{(x_2-x_1)/r_{\lambda_1}}) \Big),\nn \\
 =& \N_{Y({\tau_{1}})}\Big( L^0 \exp(-\lambda_0^{4-d} L^0)\Big)=\N_{e_1}\Big( L^0 \exp(-\lambda_0^{4-d} L^0) \Big)=V_1^{\lambda_0^{4-d}}(1),
     \end{align} 
     where in the next to last equality we have used spherical symmetry and $e_1$ is the first unit basis vector. The last equality follows by \eqref{e12.2}.
       In view of \eqref{e10.8}, we have proved 
\begin{align}
 \widetilde{J_2}=r_{\lambda_1}^{p} \lambda_1^{1+\alpha}V_1^{\vec{\lambda},\vec{x}}(Y({\tau_{r_{\lambda_1}}})+x_1) \overset{d}{\to} \lambda_0^{p+2-d}V_1^{\lambda_0^{4-d}}(1) \text{ in distribution} 
 \end{align}
as $\lambda_1, \lambda_2 \to \infty$, and furthermore, under $\Hat{P}_{x-x_1}^{(2-2\nu)}$ we have
\begin{align}\label{e10.10}
 \widetilde{J_2}=r_{\lambda_1}^{p} \lambda_1^{1+\alpha} V_1^{\vec{\lambda},\vec{x}}(Y({\tau_{r_{\lambda_1}}})+x_1) \to\lambda_0^{p+2-d}V_1^{\lambda_0^{4-d}}(1) \text{ in probability}
\end{align} 
 as $\lambda_1, \lambda_2 \to \infty$ since $\lambda_0^{p+2-d}V_1^{\lambda_0^{4-d}}(1)$ is a constant.\\
 
By \eqref{ec2.1.0}, with  $\Hat{P}^{(2-2\nu)}_{x-x_1}$-probability one we have
   \begin{align}\label{e10.11}
 \widetilde{J_1}=1_{\{\tau_{r_{\lambda_1}}<T^{'}_{r_{\lambda_2}}\}} \to 1 \text{ as } \lambda_1, \lambda_2 \to \infty.
   \end{align}
As for \eqref{ec2.1.1}, we use Lemma \ref{l7.5} to see that with  $\Hat{P}^{(2-2\nu)}_{x-x_1}$-probability one,
      \begin{align}\label{e10.12}
    \widetilde{J_3}=\exp\Big(-&\int_0^{\tau_{r_{\lambda_1}}} (V^{\vec{\lambda},\vec{x}}(Y_s+x_1)-V^{\lambda_1}(Y_s)) ds\Big)\nn \\
     &\to \exp\Big(-\int_0^{\tau_{0}} (V^{\vec{\infty},\vec{x}}(Y_s+x_1)-V^\infty(Y_s)) ds\Big) \text{ as } \lambda_1, \lambda_2 \to \infty.
     \end{align}
Here one can see from \eqref{e10.2} that
    \begin{align*}
0&\leq V^{\vec{\lambda},\vec{x}}(Y_s+x_1)-V^{\lambda_1}(Y_s)\leq  V^{\lambda_2}(Y_s-(x_2-x_1))\leq V^\infty(Y_s-(x_2-x_1)),
   \end{align*} 
   and so apply Dominated Convergence as before.\\
  
Combining \eqref{e10.10}, \eqref{e10.11} and \eqref{e10.12}, we conclude that under $\Hat{P}_{x-x_1}^{(2-2\nu)}$,
\begin{align}\label{ce1.4.5}
 \widetilde{J_1} \widetilde{J_2} \widetilde{J_3}  \to \lambda_0^{p+2-d}V_1^{\lambda_0^{4-d}}(1) & \exp\Big(-\int_0^{\tau_{0}} (V^{\vec{\infty},\vec{x}}(Y_s+x_1)-V^\infty(Y_s)) ds\Big)\nn\\
&  \text{ in probability as  } \lambda_1, \lambda_2 \to \infty.
\end{align}
Recall \eqref{m10.0} to see that
  \begin{align*}
 \widetilde{J_2}= &r_{\lambda_1}^{p} \lambda_1^{1+\alpha}V_1^{\vec{\lambda},\vec{x}}(Y({\tau_{r_{\lambda_1}}})+x_1)
 \leq r_{\lambda_1}^{p}  c_{\ref{p1.1}} |Y({\tau_{r_{\lambda_1}}})|^{-p}=c_{\ref{p1.1}}.
\end{align*} 
Use \eqref{e10.2} to see that $ \widetilde{J_3}\leq 1$ and so conclude
      \begin{align}\label{ce1.4.6}
      \widetilde{J_1} \widetilde{J_2} \widetilde{J_3} \leq c_{\ref{p1.1}},  \quad \Hat{P}_{x-x_1}^{(2-2\nu)}-a.s.
      \end{align}
Recall \eqref{ae8.10} and use \eqref{ce1.4.5}, \eqref{ce1.4.6} and bounded convergence theorem to get
        \begin{align}\label{5.3}
     \lim_{\lambda_1, \lambda_2 \to \infty} \Hat{E}_{x-x_1}^{(2-2\nu)}\Big(&\Big(\lambda_0^{p+2-d}V_1^{\lambda_0^{4-d}}(1) \exp\Big(-\int_0^{\tau_{0}} (V^{\vec{\infty},\vec{x}}(Y_s+x_1)-V^\infty(Y_s)) ds\Big)\nn \\
     &- \widetilde{J_1} \widetilde{J_2} \widetilde{J_3}\Big)^2\Big)=0.
            \end{align}
Recalling $\widetilde{J_4}$ as in \eqref{e10.7}, we use the fact that under $\Hat{P}_{x-x_1}^{(2-2\nu)}$, the process $\{|Y_{s\wedge \tau_{r_{\lambda_1}}}|, s\geq 0\}$ is a stopped $(2-2\nu)$-dimensional Bessel process and then use Corollary \ref{c1.4} to get for all $\lambda_1>0$,
\begin{align}\label{5.4}
&\Hat{E}_{x-x_1}^{(2-2\nu)}(\widetilde{J_4}^2)=\Hat{E}_{x-x_1}^{(2-2\nu)}\Big(\exp\Big(2\int_0^{\tau_{r_{\lambda_1}}} (V^{\infty}(Y_s)-V^{\lambda_1}(Y_s)) ds\Big)\Big)\nn \\
=&{E}_{|x-x_1|}^{(2-2\nu)}\Big(\exp\Big(2\int_0^{\tau_{r_{\lambda_1}}} (V^{\infty}(\rho_s)-V^\lambda(\rho_s)) ds\Big)\Big)\nn \\
=&{E}_{|x-x_1|}^{(2+2\nu)}\Big(\exp\Big(2\int_0^{\tau_{r_{\lambda_1}}} (V^{\infty}(\rho_s)-V^\lambda(\rho_s)) ds\Big)\Big|\tau_{r_{\lambda_1}}<\infty\Big)\nn\\
\leq& C_{\ref{c13.5}}(\lambda_0, \nu, 2)<\infty,
\end{align} 
 where we have chosen $\lambda_0>c_{\ref{c13.5}}$ so that we can apply Lemma \ref{c13.5} in the last inequality. Now we can conclude that
\begin{align*}
&\Big|\Hat{E}_{x-x_1}^{(2-2\nu)}(\widetilde{J_1} \widetilde{J_2} \widetilde{J_3}\widetilde{J_4} )-\Hat{E}_{x-x_1}^{(2-2\nu)}\Big(\lambda_0^{p+2-d}V_1^{\lambda_0^{4-d}}(1) \\
&\quad \quad \quad \quad \times \exp\Big(-\int_0^{\tau_{0}} (V^{\vec{\infty},\vec{x}}(Y_s+x_1)-V^\infty(Y_s)) ds\Big) \times \widetilde{J_4}\Big)\Big| \nn \\
\leq&\Hat{E}_{|x-x_1|}^{(2-2\nu)}\Big( \widetilde{J_4}\times \Big|\widetilde{J_1} \widetilde{J_2} \widetilde{J_3}-\lambda_0^{p+2-d}V_1^{\lambda_0^{4-d}}(1) \nn\\
&\quad \quad \quad \quad \times \exp\Big(-\int_0^{\tau_{0}} (V^{\vec{\infty},\vec{x}}(Y_s+x_1)-V^\infty(Y_s)) ds\Big)\Big|\Big)\nn\\
\leq  & \Big(\Hat{E}_{x-x_1}^{(2-2\nu)} (\widetilde{J_4}^2) \Big)^{1/2} \Big(\Hat{E}_{x-x_1}^{(2-2\nu)}\Big(\widetilde{J_1} \widetilde{J_2} \widetilde{J_3}-\lambda_0^{p+2-d}V_1^{\lambda_0^{4-d}}(1)\\
&\quad \quad \quad \quad \times \exp\Big(-\int_0^{\tau_{0}} (V^{\vec{\infty},\vec{x}}(Y_s+x_1)-V^\infty(Y_s)) ds\Big) \Big)^2\Big)^{1/2} \to 0
\end{align*} 
as $\lambda_1, \lambda_2 \to \infty$, where the second inequality is by Cauchy-Schwartz and the convergence to $0$ follows from \eqref{5.3} and \eqref{5.4}. In view of \eqref{e10.7}, we conclude
\begin{align}\label{6.1}
&\lim_{\lambda_1, \lambda_2 \to \infty} I_1=\frac{\lambda_0^{p+2-d}V_1^{\lambda_0^{4-d}}(1)}{|x-x_1|^p}\\
& \lim_{\lambda_1, \lambda_2 \to \infty} \Hat{E}_{x-x_1}^{(2-2\nu)}\Big(\exp\Big(-\int_0^{\tau_{0}} (V^{\vec{\infty},\vec{x}}(Y_s+x_1)-V^\infty(Y_s)) ds\Big) \cdot \widetilde{J_4}\Big),\nn
 \end{align}     
 providing we can show the limit on the right-hand side exists.
 
Recall $C_{\ref{c13.5}}(\lambda_0, \nu,1)$ as in Lemma \ref{c13.5}. We claim that
\begin{align}\label{5.9}
    & \lim_{\lambda_1 \to \infty}  \Hat{E}_{x-x_1}^{(2-2\nu)}\Big(\exp\Big(-\int_0^{\tau_{0}} (V^{\vec{\infty},\vec{x}}(Y_s+x_1)-V^\infty(Y_s)) ds\Big) \cdot \widetilde{J_4}\Big)\\
     =&C_{\ref{c13.5}}(\lambda_0, \nu,1) \Hat{E}_{x-x_1}^{(2-2\nu)}\Big(\exp\Big(-\int_0^{\tau_{0}} (V^{\vec{\infty},\vec{x}}(Y_s+x_1)-V^\infty(Y_s)) ds\Big)\Big).\nn
     \end{align}
    It will then follow from \eqref{m1.1}, \eqref{m1.2}, \eqref{6.1} and \eqref{5.9} that
\begin{align*}
     &\lim_{\lambda_1,\lambda_2 \to \infty}\lambda_1^{1+\alpha} V_1^{\vec{\lambda},\vec{x}}(x)=\lambda_0^{p+2-d}V_1^{\lambda_0^{4-d}}(1)C_{\ref{c13.5}}(\lambda_0, \nu,1)\nn\\
&\quad \quad |x-x_1|^{-p} \Hat{E}_{x-x_1}^{(2-2\nu)}\Big(\exp\Big(-\int_0^{\tau_{0}} (V^{\vec{\infty},\vec{x}}(Y_s+x_1)-V^\infty(Y_s)) ds\Big)\Big),
\end{align*}
  and the proof will be complete by letting $K_{\ref{p3.1}}=\lambda_0^{p+2-d}V_1^{\lambda_0^{4-d}}(1)C_{\ref{c13.5}}(\lambda_0, \nu,1)$. Recall $c_{\ref{p12}}$ as in Lemma \ref{p12} (see \eqref{6.3}) to conclude $K_{\ref{p3.1}}=c_{\ref{p12}}$.\\ 
  
    It remains to prove \eqref{5.9}. First by \eqref{ae8.10} and monotone convergence theorem, we have
\begin{align}\label{m1.3}
\lim_{\delta\to 0}& \Big|\Hat{E}_{x-x_1}^{(2-2\nu)}\Big(\exp\Big(-\int_0^{\tau_{\delta}} (V^{\vec{\infty},\vec{x}}(Y_s+x_1)-V^\infty(Y_s)) ds\Big)\Big)\nn\\
&-\Hat{E}_{x-x_1}^{(2-2\nu)}\Big(\exp\Big(-\int_0^{\tau_{0}} (V^{\vec{\infty},\vec{x}}(Y_s+x_1)-V^\infty(Y_s)) ds\Big)\Big)\Big|=0.
\end{align} 
Since $\widetilde{J_4}$ has an uniform $L^2$ bound for all $\lambda_1>0$ by \eqref{5.4}, by Cauchy-Schwartz we have
\begin{align}\label{m1.4}
& \Big|\Hat{E}_{x-x_1}^{(2-2\nu)}\Big(\exp\Big(-\int_0^{\tau_{\delta}} (V^{\vec{\infty},\vec{x}}(Y_s+x_1)-V^\infty(Y_s)) ds\Big)\times \widetilde{J_4}\Big)\\
&-\Hat{E}_{x-x_1}^{(2-2\nu)}\Big(\exp\Big(-\int_0^{\tau_{0}} (V^{\vec{\infty},\vec{x}}(Y_s+x_1)-V^\infty(Y_s)) ds\Big)\times \widetilde{J_4}\Big)\Big|\nn\\
\leq&\Big(\Hat{E}_{x-x_1}^{(2-2\nu)} (\widetilde{J_4}^2) \Big)^{1/2} \Big(\Hat{E}_{x-x_1}^{(2-2\nu)}\Big(\exp\Big(-\int_0^{\tau_{\delta}} (V^{\vec{\infty},\vec{x}}(Y_s+x_1)-V^\infty(Y_s)) ds\Big)\nn\\
&\quad  -\exp\Big(-\int_0^{\tau_{0}} (V^{\vec{\infty},\vec{x}}(Y_s+x_1)-V^\infty(Y_s)) ds\Big)\Big)^2\Big)^{1/2} \to 0 \text{ as } \delta \downarrow 0\nn
\end{align} 
uniformly for all $\lambda_1>0$, where the last follows from monotone convergence theorem and \eqref{5.4}.
Fixing any $\delta>0$, we will show that
\begin{align}\label{6.0}
&\lim_{\lambda_1 \to \infty} \Hat{E}_{x-x_1}^{(2-2\nu)}\Big(\exp\Big(-\int_0^{\tau_{\delta}} (V^{\vec{\infty},\vec{x}}(Y_s+x_1)-V^\infty(Y_s)) ds\Big) \cdot \widetilde{J_4}\Big)
\\
=& C_{\ref{c13.5}}(\lambda_0, \nu,1) \Hat{E}_{x-x_1}^{(2-2\nu)}\Big(\exp\Big(-\int_0^{\tau_{\delta}} (V^{\vec{\infty},\vec{x}}(Y_s+x_1)-V^\infty(Y_s)) ds\Big)\Big),\nn
\end{align}
     and one can easily conclude from \eqref{m1.3}, \eqref{m1.4} and \eqref{6.0} that \eqref{5.9} holds.

It remains to prove \eqref{6.0}. 
For $r_{\lambda_1}<\delta$ we use strong Markov property of $(Y_s, s\geq 0)$ to get 
\begin{align}\label{m2.1}
&\Hat{E}_{x-x_1}^{(2-2\nu)}\Big(\exp\Big(-\int_0^{\tau_{\delta}} (V^{\vec{\infty},\vec{x}}(Y_s+x_1)-V^\infty(Y_s)) ds\Big)\nn\\
&\quad\quad \quad \times\exp\Big(\int_0^{\tau_{r_{\lambda_1}}} (V^{\infty}-V^{\lambda_1})(Y_s) ds\Big)\Big)\nn\\
=&\Hat{E}_{x-x_1}^{(2-2\nu)}\Bigg(\exp\Big(-\int_0^{\tau_{\delta}} (V^{\vec{\infty},\vec{x}}(Y_s+x_1)-V^\infty(Y_s)) ds\Big)\nn\\
&\quad\quad \quad \times \exp\Big(\int_0^{\tau_{\delta}} (V^{\infty}(Y_s)-V^{\lambda_1}(Y_s)) ds\Big)\nn\\
&\quad \quad \quad \times \Hat{E}_{Y_{\tau_\delta}}^{(2-2\nu)} \Big(\exp\Big(\int_0^{\tau_{r_{\lambda_1}}} (V^{\infty}(Y_s)-V^{\lambda_1}(Y_s)) ds\Big) \Big)\Bigg).
\end{align} 
Using Corollary \ref{c1.4} as in \eqref{5.4}, we have 
\begin{align}\label{m2.2}
 &\Hat{E}_{Y_{\tau_\delta}}^{(2-2\nu)} \Big(\exp\Big(\int_0^{\tau_{r_{\lambda_1}}} (V^{\infty}(Y_s)-V^{\lambda_1}(Y_s)) ds\Big) \Big)\nn\\
&={E}_{|Y_{\tau_\delta}|}^{(2-2\nu)}\Big(\exp\Big(\int_0^{\tau_{r_{\lambda_1}}} (V^{\infty}(\rho_s)-V^{\lambda_1}(\rho_s)) ds\Big) \Big)\nn\\
  &={E}_{\delta}^{(2+2\nu)}\Big(\exp\Big(\int_0^{\tau_{r_{\lambda_1}}} (V^{\infty}(\rho_s)-V^{\lambda_1}(\rho_s)) ds\Big)\Big|\tau_{r_{\lambda_1}}<\infty \Big)\nn\\
&\uparrow C_{\ref{c13.5}}(\lambda_0, \nu,1) \text{ as } \lambda_1 \to \infty,
\end{align} 
where the last follows from Lemma \ref{c13.5} by choosing $\lambda_0>c_{\ref{c13.5}}$.
Next since $\delta>0$ is fixed, we have
\begin{align}\label{m2.3}
\lim_{\lambda_1 \to \infty} \exp\Big(\int_0^{\tau_{\delta}} (V^{\infty}(Y_s)-V^{\lambda_1}(Y_s)) ds\Big) \to 1, \Hat{P}_{x-x_1}^{(2-2\nu)}-a.s.
\end{align} 
In view of \eqref{m2.2} and \eqref{ae8.10}, with $\Hat{P}_{x-x_1}^{(2-2\nu)}$-probability one, for any $\lambda_1>0$ we have
\begin{align}\label{m2.4}
&\exp\Big(-\int_0^{\tau_{\delta}} (V^{\vec{\infty},\vec{x}}(Y_s+x_1)-V^\infty(Y_s)) ds\Big)\nn\\
&  \times \exp\Big(\int_0^{\tau_{\delta}} (V^{\infty}(Y_s)-V^{\lambda_1}(Y_s)) ds\Big)\nn\\
&  \times \Hat{E}_{Y_{\tau_\delta}}^{(2-2\nu)} \Big(\exp\Big(\int_0^{\tau_{r_{\lambda_1}}} (V^{\infty}(Y_s)-V^{\lambda_1}(Y_s)) ds\Big) \Big)\nn\\
&\leq \exp\Big(\int_0^{\tau_{\delta}} V^{\infty}(Y_s) ds\Big)\cdot C_{\ref{c13.5}}(\lambda_0, \nu,1).
\end{align} 
Similar to \eqref{m2.2}, we apply Corollary \ref{c1.4} and Lemma \ref{l13.5}(i) to get
\begin{align}\label{m2.5}
&\Hat{E}_{x-x_1}^{(2-2\nu)}\Big(\exp\Big(\int_0^{\tau_{\delta}} V^{\infty}(Y_s) ds\Big)\Big)={E}_{|x-x_1|}^{(2-2\nu)}\Big(\exp\Big(\int_0^{\tau_{\delta}} V^{\infty}(\rho_s) ds\Big)\Big)\\
=&{E}_{|x-x_1|}^{(2+2\nu)}\Big(\exp\Big(\int_0^{\tau_{\delta}} V^{\infty}(\rho_s) ds\Big)\Big|\tau_\delta<\infty\Big)\nn\\
=&{E}_{|x-x_1|/\delta}^{(2+2\nu)}\Big(\exp\Big(\int_0^{\tau_{1}} V^{\infty}(\rho_s) ds\Big)\Big|\tau_1<\infty\Big)=(|x-x_1|/\delta)^{\nu-\mu}<\infty,\nn
\end{align} 
where the second last equality is by scaling of Bessel process.
Combine \eqref{m2.2}-\eqref{m2.5} to see that the integrand in \eqref{m2.1} converges pointwise a.s. and is bounded by (the integrable) $\exp\Big(\int_0^{\tau_{\delta}} V^{\infty}(Y_s) ds\Big)\cdot C_{\ref{c13.5}}(\lambda_0, \nu,1)$. Therefore by Dominated Convergence we conclude that 
\begin{align*}
&\lim_{\lambda_1 \to \infty} \Hat{E}_{x-x_1}^{(2-2\nu)}\Big(\exp\Big(-\int_0^{\tau_{\delta}} (V^{\vec{\infty},\vec{x}}(Y_s+x_1)-V^\infty(Y_s)) ds\Big)\\
&\quad \quad \quad \quad \quad \quad \quad  \times \exp\Big(\int_0^{\tau_{r_{\lambda_1}}} (V^{\infty}-V^{\lambda_1})(Y_s) ds\Big)\Big)\\
=&C_{\ref{c13.5}}(\lambda_0, \nu,1)  \Hat{E}_{x-x_1}^{(2-2\nu)}\Big(\exp\Big(-\int_0^{\tau_{\delta}} (V^{\vec{\infty},\vec{x}}(Y_s+x_1)-V^\infty(Y_s)) ds\Big)\Big),
\end{align*} 
and the proof of \eqref{6.0} is complete.
\end{proof}

\begin{proof}[Proof of Proposition \ref{p3.3}(i)]
 We will only give the convergence of $\lambda_1^{1+\alpha} W_1^{\vec{\lambda},\vec{x},\eps}(x)$ and leave the details for the convergence of $\frac{1}{\eps^{p-2}} W_2^{\vec{\lambda},\vec{x},\eps}(x)$ to the readers.
Recall Lemma \ref{8.9} to see that
  \begin{align*}
& \lambda_1^{1+\alpha} W_1^{\vec{\lambda},\vec{x},\eps}(x)\\
=&\lambda_1^{1+\alpha}\lim_{t\to \infty}E_x\Big( W_1^{\vec{\lambda},\vec{x},\eps}(B(t\wedge T_{\lambda_1,\eps}))\exp\Big(-\int_0^{t\wedge T_{\lambda_1, \eps}} W^{\vec{\lambda},\vec{x},\eps}(B_s) ds\Big)\Big),
  \end{align*}
where $T_{\lambda_1,\eps}=T_{r_{\lambda_1}}^1 \wedge T_{2\eps}^2$ and $T_{r_{\lambda_1}}^1=\inf\{t\geq 0: |B_t-x_1|\leq r_{\lambda_1} \}$ and $T_{2\eps}^2=\inf\{t\geq 0: |B_t-x_2|\leq 2\eps\}$. Here $r_{\lambda_1}=\lambda_0 \lambda_1^{-\frac{1}{4-d}}$ and we will choose $\lambda_0$ to be some fixed large constant below. By \eqref{m8.7}, we have $W_1^{\vec{\lambda},\vec{x},\eps}(x) \to 0$ as $|x|\to \infty$ and $W_1^{\vec{\lambda},\vec{x},\eps}(B(t\wedge T_{r_{\lambda},\eps}))$ is uniformly bounded for all $t\geq 0$. Apply Dominated Convergence to get
  \begin{align}\label{bm1.1}
& \lambda_1^{1+\alpha} W_1^{\vec{\lambda},\vec{x},\eps}(x)\\
 =&\lambda_1^{1+\alpha} E_x\Big(1_{\{T_{\lambda_1,\eps}<\infty\}} W_1^{\vec{\lambda},\vec{x},\eps}(B(T_{\lambda_1, \eps}))\exp\Big(-\int_0^{T_{\lambda_1, \eps}} W^{\vec{\lambda},\vec{x},\eps}(B_s) ds\Big)\Big)\nn\\
 =& E_x\Big(1_{\{T_{r_{\lambda_1}}^1<\infty\}} 1_{\{T_{r_{\lambda_1}}^1<T_{2\eps}^{2}\}} \lambda_1^{1+\alpha}W_1^{\vec{\lambda},\vec{x},\eps}(B(T_{r_{\lambda_1}}^1))\exp\Big(-\int_0^{T_{r_{\lambda_1}}^1} W^{\vec{\lambda},\vec{x},\eps}(B_s) ds\Big)\Big)\nn\\
 & + E_x\Big(1_{\{T_{2\eps}^2<\infty\}} 1_{\{T_{2\eps}^2<T_{r_{\lambda_{1}}}^{1}\}} \lambda_1^{1+\alpha}W_1^{\vec{\lambda},\vec{x},\eps}(B(T_{2\eps}^2))\exp\Big(-\int_0^{T_{2\eps}^2} W^{\vec{\lambda},\vec{x},\eps}(B_s) ds\Big)\Big)\nn\\
 :=&I_1+I_2,\nn
  \end{align}
  We first deal with $I_2$. Note in the integrand of $I_2$ we may assume that $|B(T_{2\eps}^2)-x_2|=2\eps$ and so for $\eps<|x_1-x_2|/4$ we have $|x_1-B(T_{2\eps}^2)|>\Delta/2$ where $\Delta=|x_1-x_2|$. Apply \eqref{m8.7} with $x=B(T_{2\eps}^2)$ to get
\begin{align}\label{be10.4}
  &\lambda_1^{1+\alpha}W_1^{\vec{\lambda},\vec{x},\eps}(B(T_{2\eps}^2) )\leq c_{\ref{p1.1}}  |B(T_{2\eps}^2)-x_1|^{-p}\leq c_{\ref{p1.1}} \Delta^{-p} 2^p.
    \end{align}
    Let $\tau_r=\inf\{t\geq 0: |B_t|\leq r\}$ and use \eqref{be10.4} and \eqref{ae7.4} to see that  $I_2$ becomes
   \begin{align}\label{bm1.2}
   I_2&\leq   c_{\ref{p1.1}}2^p \Delta^{-p} E_x\Big(1_{\{T_{2\eps}^2<\infty\}} 1_{\{T_{2\eps}^2<T_{r_{\lambda_{1}}}^{1}\}} \exp\Big(-\int_0^{T_{2\eps}^2} W^{\vec{\lambda},\vec{x},\eps}(B_s) ds\Big)\Big)\nn\\
  &\leq  c_{\ref{p1.1}} 2^p \Delta^{-p} E_{x-x_2}\Big(1_{\{\tau_{2\eps}<\infty\}}  \exp\Big(-\int_0^{\tau_{2\eps}} V^{\infty}(B_s) ds\Big)\Big)\nn\\
      &=  c_{\ref{p1.1}} 2^p \Delta^{-p} (2\eps/|x-x_2|)^{p}\to 0 \text{ as }  \eps \downarrow 0,
   \end{align}
   where we have used Proposition \ref{p20.1} in the last equality with $g=V^\infty$. 
 
   Now we will turn to $I_1$. Let $(Y_t, t\geq 0)$ be the $d$-dimensional coordinate process under Wiener measure, $P_x$. By slightly abusing the notation, we set $\tau_r=\tau_r^Y=\inf\{t\geq 0: |Y_t|\leq r\}$ for any $r>0$, and set
     \begin{align}\label{be10.5}
     T^{'}_{2\eps}= T^{',Y}_{2\eps}=\inf\{t\geq 0: |Y_t-(x_2-x_1)|\leq 2\eps\}.
     \end{align}
Then use translation invariance of $Y$ to get
  \begin{align*}
   I_1=&E_{x-x_1}\Big(1_{\{\tau_{r_{\lambda_1}}<\infty\}}  1_{\{\tau_{r_{\lambda_1}}<T^{'}_{2\eps}\}}  \lambda_1^{1+\alpha}W_1^{\vec{\lambda},\vec{x},\eps}(Y({\tau_{r_{\lambda_1}}})+x_1)\nn \\
   &\quad \quad \quad \times \exp\Big(-\int_0^{\tau_{r_{\lambda_1}}} W^{\vec{\lambda},\vec{x},\eps}(Y_s+x_1) ds\Big)\Big).
   \end{align*}
 Recall that $\Hat{P}_{x}^{(2-2\nu)}$ is the law of $Y$ starting from $x$ such that $Y$ satisfy the SDE as in \eqref{e10.6}. Now apply Proposition \ref{p8.1} with $g(\cdot)=W^{\vec{\lambda},\vec{x},\eps}(\cdot+x_1)$ to get
  \begin{align}\label{be10.7}
   I_1=&\frac{r_{\lambda_1}^{p}}{|x-x_1|^{p}} \Hat{E}^{(2-2\nu)}_{x-x_1}\Big(1_{\{\tau_{r_{\lambda_1}}<T^{'}_{2\eps}\}} \lambda_1^{1+\alpha}W_1^{\vec{\lambda},\vec{x},\eps}(Y({\tau_{r_{\lambda_1}}})+x_1)\nn \\
   &\quad \quad \quad \times \exp\Big(-\int_0^{\tau_{r_{\lambda_1}}} (W^{\vec{\lambda},\vec{x},\eps}(Y_s+x_1)-V^\infty(Y_s)) ds\Big)\Big)\nn\\
 =&\frac{1}{|x-x_1|^{p}} \Hat{E}^{(2-2\nu)}_{x-x_1}\Bigg([1_{\{\tau_{r_{\lambda_1}}<T^{'}_{2\eps}\}}]  [r_{\lambda_1}^{p} \lambda_1^{1+\alpha} W_1^{\vec{\lambda},\vec{x},\eps}(Y({\tau_{r_{\lambda_1}}})+x_1)]\nn \\
   & \quad\quad \quad \times \Big[\exp\Big(-\int_0^{\tau_{r_{\lambda_1}}} (W^{\vec{\lambda},\vec{x},\eps}(Y_s+x_1)-V^{\lambda_1}(Y_s)) ds\Big)\Big]\nn\\
  & \quad\quad \quad \times \Big[\exp\Big(\int_0^{\tau_{r_{\lambda_1}}} (V^\infty(Y_s)-V^{\lambda_1}(Y_s)) ds\Big)\Big]\Bigg)\nn\\
   :=&\frac{1}{|x-x_1|^{p}} \Hat{E}^{(2-2\nu)}_{x-x_1} ([\Hat{J_1}] [\Hat{J_2}] [\Hat{J_3}] [\Hat{J_4}]).
   \end{align}
  We first consider $\Hat{J_2}$. Recall the definition of $W_1^{\vec{\lambda},\vec{x},\eps}$ as in Section \ref{s4} and use translation invariance to get 
\begin{align*}
 \Hat{J_2}=&r_{\lambda_1}^{p} \lambda_1^{1+\alpha} \N_{Y({\tau_{r_{\lambda_1}}})+x_1}\Big(L^{x_1} e^{-\lambda_1 L^{x_1}}\exp\Big(-\lambda_2 \frac{X_{G_\eps^{x_2}}(1)}{\eps^2}\Big) 1_{\{X_{G_{\eps/2}^{x_2}}=0\}}\Big) \\
 =&r_{\lambda_1}^{p} \lambda_1^{1+\alpha}   \N_{Y({\tau_{r_{\lambda_1}}})}\Big(L^0 e^{-\lambda_1 L^{0}}\exp\Big(-\lambda_2 \frac{X_{G_\eps^{x_2-x_1}}(1)}{\eps^2}\Big) 1_{\{X_{G_{\eps/2}^{x_2-x_1}}=0\}}\Big).
 \end{align*} 
By the scaling of Brownian snake and its local time and exit measure under the excursion measure $\N_x$ (see, e.g., Proof of Proposition V.9 in \cite{Leg99}), we have
\begin{align*}
 \Hat{J_2}=& r_{\lambda_1}^{p} \lambda_1^{1+\alpha}  r_{\lambda_1}^{-2} \N_{Y({\tau_{r_{\lambda_1}}})/r_{\lambda_1}}\Bigg(r_{\lambda_1}^{4-d} L^0 \exp(- {\lambda_1} r_{\lambda_1}^{4-d} L^0)\nn\\
&\quad \quad \times \exp\Big(- \lambda_2 \frac{X_{G_{\eps/r_{\lambda_1}}^{(x_2-x_1)/r_{\lambda_1}}}(1)}{(\eps/r_{\lambda_1})^{2}}\Big)1\Big(X_{G_{\eps/2r_{\lambda_1}}^{(x_2-x_1)/r_{\lambda_1}}}=0\Big)  \Bigg)
      \end{align*} 
Use the definitions of $r_{\lambda_1}$ and $\alpha$ to see that the above becomes
\begin{align}\label{be10.8}
   \Hat{J_2} =& \lambda_0^{p+2-d} \N_{Y({\tau_{r_{\lambda_1}}})/r_{\lambda_1}}\Big(L^0 \exp(- \lambda_0^{4-d} L^0)\\
&\quad \quad \times \exp\Big(- \lambda_2 \frac{X_{G_{\eps/r_{\lambda_1}}^{(x_2-x_1)/r_{\lambda_1}}}(1)}{(\eps/r_{\lambda_1})^{2}}\Big)1\Big(X_{G_{\eps/2r_{\lambda_1}}^{(x_2-x_1)/r_{\lambda_1}}}=0\Big)  \Big) \nn\\
\overset{\text{law}}{=}& \lambda_0^{p+2-d} \N_{Y({\tau_{1}})}\Big(L^0 \exp(- \lambda_0^{4-d} L^0)\nn\\
&\quad \quad \times \exp\Big(- \lambda_2 \frac{X_{G_{\eps/r_{\lambda_1}}^{(x_2-x_1)/r_{\lambda_1}}}(1)}{(\eps/r_{\lambda_1})^{2}}\Big)1\Big(X_{G_{\eps/2r_{\lambda_1}}^{(x_2-x_1)/r_{\lambda_1}}}=0\Big)  \Big),\nn
      \end{align} 
 where the last equality follows from the scaling of $Y$.
Note for any $K>0$,  for all $0<\eps<|x_1-x_2|/2$, we have 
\[\Big|\frac{x_2-x_1}{r_{\lambda_1}}\Big|-\frac{\eps}{r_{\lambda_1}}\geq \frac{|x_1-x_2|/2}{r_{\lambda_1}}>K \text{ for } \lambda_1 \text{ large enough, }\] 
and so by \eqref{ea0.0} and \eqref{ea1.1} we conclude $\N_{Y({\tau_{1}})}$-a.e.
\[X_{G_{\eps/r_{\lambda_1}}^{(x_2-x_1)/r_{\lambda_1}}}(1)=X_{G_{\eps/2r_{\lambda_1}}^{(x_2-x_1)/r_{\lambda_1}}}(1)=0 \text{ for } \lambda_1 \text{ large enough.}\]
Therefore an application of Dominated Convergence Theorem will give us
  \begin{align}\label{be10.9}
 \lim_{\lambda_1 \to \infty, \eps \downarrow 0}  & \N_{Y({\tau_{1}})}\Big( L^0 e^{-\lambda_0^{4-d} L^0}\exp\Big(- \lambda_2 \frac{X_{G_{\eps/r_{\lambda_1}}^{(x_2-x_1)/r_{\lambda_1}}}(1)}{(\eps/r_{\lambda_1})^{2}}\Big)1\Big(X_{G_{\eps/2r_{\lambda_1}}^{(x_2-x_1)/r_{\lambda_1}}}=0\Big) \Big)\nn \\
 =& \N_{Y({\tau_{1}})}\Big( L^0 e^{-\lambda_0^{4-d} L^0}\Big)=\N_{e_1}\Big( L^0 e^{-\lambda_0^{4-d} L^0} \Big)=V_1^{\lambda_0^{4-d}}(1),
     \end{align} 
     where in the next to last equality we have used spherical symmetry and $e_1$ is the first unit basis vector. The last equality follows by \eqref{e12.2}.
       In view of \eqref{be10.8}, we have proved 
\begin{align}
\Hat{J_2}=r_{\lambda_1}^{p} \lambda_1^{1+\alpha}W_1^{\vec{\lambda},\vec{x},\eps}(Y_{\tau_{r_{\lambda_1}}}+x_1) \to \lambda_0^{p+2-d}V_1^{\lambda_0^{4-d}}(1)  \text{ in distribution}
\end{align}
as $\lambda_1 \to \infty, \eps \downarrow 0$, and furthermore under $\Hat{P}_{x-x_1}^{(2-2\nu)}$, we have
\begin{align}\label{be10.10}
\Hat{J_2}=r_{\lambda_1}^{p} \lambda_1^{1+\alpha} W_1^{\vec{\lambda},\vec{x},\eps}(Y_{\tau_{r_{\lambda_1}}}+x_1) \to\lambda_0^{p+2-d}V_1^{\lambda_0^{4-d}}(1) \text{ in probability}
\end{align} 
as $\lambda_1 \to \infty, \eps \downarrow 0$ since $\lambda_0^{p+2-d}V_1^{\lambda_0^{4-d}}(1)$ is a constant.\\
 
By \eqref{ec2.1.0}, with  $\Hat{P}^{(2-2\nu)}_{x-x_1}$-probability one we have
   \begin{align}\label{be10.11}
   \Hat{J_1}=1_{\{\tau_{r_{\lambda_1}}<T^{'}_{2\eps}\}} \to 1 \text{ as }\lambda_1 \to \infty, \eps \downarrow 0.
   \end{align}
As for \eqref{ec2.1.1}, we use Lemma \ref{l7.5} to see that with  $\Hat{P}^{(2-2\nu)}_{x-x_1}$-probability one,
      \begin{align}\label{be10.12}
   \Hat{J_3}=\exp\Big(-&\int_0^{\tau_{r_{\lambda_1}}} (W^{\vec{\lambda},\vec{x},\eps}(Y_s+x_1)-V^{\lambda_1}(Y_s)) ds\Big)\nn \\
     &\to \exp\Big(-\int_0^{\tau_{0}} (V^{\vec{\infty},\vec{x}}(Y_s+x_1)-V^\infty(Y_s)) ds\Big) \text{ as } \lambda_1 \to \infty, \eps \downarrow 0.
     \end{align}
Here one can see from \eqref{8.1} that
    \begin{align}\label{ce4.2.1}
0\leq W^{\vec{\lambda},\vec{x},\eps}(Y_s+x_1)-V^{\lambda_1}(Y_s)\leq  U^{\widetilde{\lambda}_2\eps^{-2},\eps}(Y_s-(x_2-x_1)),
   \end{align} 
where $\widetilde{\lambda}_2$ is as in \eqref{ae2.4}. Then argue as in the derivation of \eqref{aea6.3} and apply Dominated Convergence as before.\\
 
Combine \eqref{be10.10}, \eqref{be10.11} and \eqref{be10.12} to see that under $\Hat{P}_{x-x_1}^{(2-2\nu)}$, we have
\begin{align}\label{ce3.4.6}
\Hat{J_1} \Hat{J_2} \Hat{J_3} \to \lambda_0^{p+2-d}V_1^{\lambda_0^{4-d}}(1) & \exp\Big(-\int_0^{\tau_{0}} (V^{\vec{\infty},\vec{x}}(Y_s+x_1)-V^\infty(Y_s)) ds\Big)\nn\\
&  \text{ in probability }  \text{as  } \lambda_1\to \infty, \eps\downarrow 0.
\end{align}
Recall from \eqref{m8.7} to see that
  \begin{align*}
 \widetilde{J_2}\leq &r_{\lambda_1}^{p} \lambda_1^{1+\alpha}W_1^{\vec{\lambda},\vec{x},\eps}(Y({\tau_{r_{\lambda_1}}})+x_1)
 \leq r_{\lambda_1}^{p}  c_{\ref{p1.1}} |Y({\tau_{r_{\lambda_1}}})|^{-p}=c_{\ref{p1.1}}.
\end{align*} 
By \eqref{ce4.2.1} we have $\Hat{J_3}\leq 1$ and so conclude
      \begin{align}\label{ce3.4.5}
      \Hat{J_1}\Hat{J_2} \Hat{J_3} \leq c_{\ref{p1.1}}, \ \Hat{P}_{x-x_1}^{(2-2\nu)}-a.s.
      \end{align}
  Recall \eqref{ae8.10} and use \eqref{ce3.4.6}, \eqref{ce3.4.5} and bounded convergence theorem to get
        \begin{align}\label{b5.3}
     \lim_{\lambda_1\to \infty, \eps\downarrow 0} \Hat{E}_{x-x_1}^{(2-2\nu)}\Big(&\Big(\lambda_0^{p+2-d}V_1^{\lambda_0^{4-d}}(1) \exp\Big(-\int_0^{\tau_{0}} (V^{\vec{\infty},\vec{x}}(Y_s+x_1)-V^\infty(Y_s)) ds\Big)\nn \\
     &-\Hat{J_1}\Hat{J_2} \Hat{J_3}\Big)^2\Big)=0.
            \end{align}
Recall $\widetilde{J_4}$ from \eqref{e10.7} to see that
\[
\Hat{J_4}=\exp\Big(\int_0^{\tau_{r_{\lambda_1}}} (V^\infty(Y_s)-V^{\lambda_1}(Y_s)) ds\Big)=\widetilde{J_4}.
\]
By \eqref{5.4} and by choosing $\lambda_0>c_{\ref{c13.5}}$, we have
\begin{align}\label{b5.4}
\Hat{E}_{x-x_1}^{(2-2\nu)}(\Hat{J_4}^2)=\Hat{E}_{x-x_1}^{(2-2\nu)}(\widetilde{J_4}^2)\leq C_{\ref{c13.5}}(\lambda_0, \nu, 2)<\infty, \forall \lambda_1>0.
\end{align} 
Now we conclude
\begin{align*}
&\Big|\Hat{E}_{x-x_1}^{(2-2\nu)}(\Hat{J_1}\Hat{J_2}\Hat{J_3}\Hat{J_4})-\Hat{E}_{x-x_1}^{(2-2\nu)}\Big(\lambda_0^{p+2-d}V_1^{\lambda_0^{4-d}}(1) \\
&\quad \quad \quad \quad \exp\Big(-\int_0^{\tau_{0}} (V^{\vec{\infty},\vec{x}}(Y_s+x_1)-V^\infty(Y_s)) ds\Big) \cdot \Hat{J_4}\Big)\Big| \nn \\
\leq&\Hat{E}_{|x-x_1|}^{(2-2\nu)}\Big( \Hat{J_4} \cdot \Big|\Hat{J_1}\Hat{J_2}\Hat{J_3}-\lambda_0^{p+2-d}V_1^{\lambda_0^{4-d}}(1) \\
&\quad \quad \quad \quad \exp\Big(-\int_0^{\tau_{0}} (V^{\vec{\infty},\vec{x}}(Y_s+x_1)-V^\infty(Y_s)) ds\Big)\Big|\Big)\nn\\
\leq  &\Hat{E}_{x-x_1}^{(2-2\nu)}\Big(\Big(\lambda_0^{p+2-d}V_1^{\lambda_0^{4-d}}(1) \exp\Big(-\int_0^{\tau_{0}} (V^{\vec{\infty},\vec{x}}(Y_s+x_1)-V^\infty(Y_s)) ds\Big)\nn \\
     &-\Hat{J_1}\Hat{J_2} \Hat{J_3}\Big)^2\Big)^{1/2} \Big(\Hat{E}_{x-x_1}^{(2-2\nu)} (\Hat{J_4}^2) \Big)^{1/2} \to 0\text{ as } \lambda_1\to \infty, \eps\downarrow 0,
\end{align*}
where the second inequality is by Cauchy-Schwartz and the convergence to $0$ follows from \eqref{b5.3} and \eqref{b5.4}. In view of \eqref{be10.7}, we have
\begin{align}\label{b6.1}
\lim_{\lambda_1\to \infty, \eps\downarrow 0} I_1=&\frac{\lambda_0^{p+2-d}V_1^{\lambda_0^{4-d}}(1)}{|x-x_1|^p}\times \nn\\
& \lim_{\lambda_1 \to \infty}\Hat{E}_{x-x_1}^{(2-2\nu)}\Big(\exp\Big(-\int_0^{\tau_{0}} (V^{\vec{\infty},\vec{x}}(Y_s+x_1)-V^\infty(Y_s)) ds\Big) \cdot \Hat{J_4}\Big)\nn\\
 =&\frac{\lambda_0^{p+2-d}V_1^{\lambda_0^{4-d}}(1)}{|x-x_1|^p}\times \nn\\
 &  C_{\ref{c13.5}}(\lambda_0, \nu,1) \Hat{E}_{x-x_1}^{(2-2\nu)}\Big(\exp\Big(-\int_0^{\tau_{0}} (V^{\vec{\infty},\vec{x}}(Y_s+x_1)-V^\infty(Y_s)) ds\Big)\Big),
     \end{align}
where the last equality is by \eqref{5.9} (recall $\Hat{J_4}=\widetilde{J_4}$).

Now we conclude from \eqref{bm1.1}, \eqref{bm1.2} and \eqref{b6.1}  that
\begin{align*}
     \lim_{\lambda_1 \to \infty, \eps \downarrow 0}&\lambda_1^{1+\alpha} W_1^{\vec{\lambda},\vec{x},\eps}(x)=\lambda_0^{p+2-d}V_1^{\lambda_0^{4-d}}(1)C_{\ref{c13.5}}(\lambda_0, \nu,1)|x-x_1|^{-p} \\
&\Hat{E}_{x-x_1}^{(2-2\nu)}\Big(\exp\Big(-\int_0^{\tau_{0}} (V^{\vec{\infty},\vec{x}}(Y_s+x_1)-V^\infty(Y_s)) ds\Big)\Big),
\end{align*}
  and the proof is complete.
   \end{proof}


\section{Proof of Proposition \ref{p3.2}(i) and Proposition \ref{p3.3}(ii)}\label{ad}
\begin{proof}[Proof of Proposition \ref{p3.2}(i)]
 For any  $x_1\neq x_2$, we fix $x\neq x_1, x_2$. In order the find the limit of $\lambda_1^{1+\alpha} \lambda_2^{1+\alpha} (-V_{1,2}^{\vec{\lambda},\vec{x}}(x))$ as $\lambda_1, \lambda_2 \to \infty$, by Lemma \ref{l4.5}, it suffices to find the limits of the following as $\lambda_1, \lambda_2 \to \infty$.
 \begin{align}\label{cm6.0}
 K_1+K_2& \equiv \lambda_1^{1+\alpha} \lambda_2^{1+\alpha}E_x\Big( \int_0^{T_{r_\lambda}} \prod_{i=1}^2 V_i^{\vec{\lambda},\vec{x}}(B_t)\exp\Big(-\int_0^{t} V^{\vec{\lambda},\vec{x}}(B_s) ds\Big)dt\Big)\nn\\
 +&\lambda_1^{1+\alpha} \lambda_2^{1+\alpha} E_x\Big( \exp\Big(-\int_0^{T_{r_\lambda}} V^{\vec{\lambda},\vec{x}}(B_s) ds\Big)1_{(T_{r_\lambda}<\infty)} (-V_{1,2}^{\vec{\lambda},\vec{x}}(B_{T_{r_\lambda}}))\Big).
\end{align}
 In the above $T_{r_{\lambda}}=T_{r_{\lambda_1}}^1 \wedge T_{r_{\lambda_2}}^2$ and $T_{r_{\lambda_i}}^i=\inf\{t\geq 0: |B_t-x_i|\leq r_{\lambda_i} \}$. Here $r_{\lambda_i}=\lambda_0 \lambda_i^{-\frac{1}{4-d}}$ and we will choose $\lambda_0$ to be some fixed large constant below. Let $\lambda_1, \lambda_2>0$ be large so that
 \begin{align}\label{ce8.4}
0<4(r_{\lambda_1}\vee r_{\lambda_2})<\min\{|x_1-x|, |x_2-x|,  |x_1-x_2|\}.
\end{align}
We first consider $K_2$.  On $\{T_{r_\lambda}<\infty\}$, by considering $T_{r_\lambda}=T_{r_{\lambda_i}}^i<T_{r_{\lambda_{3-i}}}^{3-i}$ we may set $x_\lambda(\omega)=B(T_{r_\lambda})=B(T_{r_{\lambda_i}}^i)$ so that  $|x_\lambda-x_i|=r_{\lambda_i}$ and by \eqref{ce8.4} we have $|x_{3-i}-x_\lambda|\geq \Delta/2$ where $\Delta=|x_1-x_2|$.  Lemma \ref{l4.2} and the above imply
\[
(-V_{1,2}^{\vec{\lambda},\vec{x},\vec{\eps}}(B(T_{r_\lambda})))\leq \frac{2}{\lambda_{i}}c_{\ref{p1.1}} \lambda_{3-i}^{-(1+\alpha)}\Delta^{-p} 2^p\leq c\Delta^{-p} \frac{1}{\lambda_{3-i}^{1+\alpha}} \frac{1}{\lambda_{i}}.\]
 This shows that 
 \begin{align}\label{1.4}
 K_2\leq \lambda_1^{1+\alpha}& \lambda_2^{1+\alpha}\sum_{i=1}^2  c\Delta^{-p} \frac{1}{\lambda_{3-i}^{1+\alpha}} \frac{1}{\lambda_{i}}\nn\\
 &\quad \quad E_x\Big( 1(T_{r_{\lambda_i}}^i<\infty)1(T_{r_{\lambda_i}}^i<T_{r_{\lambda_{3-i}}}^{3-i})\exp\Big(-\int_0^{ T_{r_{\lambda_i}}^i} V^{\vec{\lambda},\vec{x}}(B_s)ds\Big) \Big).
 \end{align}
From \eqref{m1.2}, by choosing $\lambda_0>c_{\ref{c13.5}}$ we have for $i=1,2,$
   \begin{align}\label{1.6}
 &E_x\Big( 1(T_{r_{\lambda_i}}^i<\infty)1(T_{r_{\lambda_i}}^i<T_{r_{\lambda_{3-i}}}^{3-i})\exp\Big(-\int_0^{ T_{r_{\lambda_i}}^i} V^{\vec{\lambda},\vec{x}}(B_s) ds\Big) \Big)\nn\\
 &\leq r_{\lambda_i}^p |x-x_i|^{-p} C_{\ref{c13.5}}(\lambda_0,\nu,1),
  \end{align}
%
and so \eqref{1.4} becomes 
 \begin{align}\label{1.7}
K_2\leq& \lambda_1^{1+\alpha} \lambda_2^{1+\alpha} c\Delta^{-p}\sum_{i=1}^2 \frac{1}{\lambda_{3-i}^{1+\alpha}} \frac{1}{\lambda_{i}} r_{\lambda_i}^p |x-x_i|^{-p} C_{\ref{c13.5}}(\lambda_0,\nu,1)\nn\\
\leq &C\Delta^{-p} \lambda_0^p(\lambda_1^{-\frac{2}{4-d}}+\lambda_2^{-\frac{2}{4-d}}) \sum_{i=1}^2  |x-x_i|^{-p} \to 0\text{ as } \lambda_1, \lambda_2 \to \infty,
 \end{align}
where in the last equality we have used the definitions of $r_{\lambda_i}$ and $\alpha$. 

Now we will turn to $K_1$. Recall 
 \begin{align*}
K_1=\int   \int_0^\infty & \lambda_1^{1+\alpha} \lambda_2^{1+\alpha} V_1^{\vec{\lambda},\vec{x}}(B_t)  V_2^{\vec{\lambda},\vec{x}}(B_t)\nn\\
&\quad \quad \exp\Big(-\int_0^{t} V^{\vec{\lambda},\vec{x}}(B_s) ds\Big)1(t\leq T_{r_\lambda})  dt dP_x.
 \end{align*}
By Proposition \ref{p3.1} and Lemma \ref{l7.5}, for $Leg\times P_x$-a.e. $(t,\omega)$, we have
 \begin{align}\label{2.4}
 \lim_{\lambda_1, \lambda_2 \to \infty}& \lambda_1^{1+\alpha} \lambda_2^{1+\alpha} V_1^{\vec{\lambda},\vec{x}}(B_t)  V_2^{\vec{\lambda},\vec{x}}(B_t)\exp\Big(-\int_0^{t} V^{\vec{\lambda},\vec{x}}(B_s) ds\Big)1(t\leq T_{r_\lambda})\nn\\
&=K_{\ref{p3.1}}^2 U_1^{\vec{\infty},\vec{x}}(B_t)U_2^{\vec{\infty},\vec{x}}(B_t)\exp\Big(-\int_0^{t} V^{\vec{\infty},\vec{x}}(B_s) ds\Big).
 \end{align}
Use the bounds \eqref{m10.0} and \eqref{e10.2}  to see that
\begin{align}\label{e4.14.1}
& \lambda_1^{1+\alpha} \lambda_2^{1+\alpha} V_1^{\vec{\lambda},\vec{x}}(B_t)  V_2^{\vec{\lambda},\vec{x}}(B_t)\exp\Big(-\int_0^{t} V^{\vec{\lambda},\vec{x}}(B_s) ds\Big)1(t\leq T_{r_\lambda})  \nonumber\\
 \leq & c_{\ref{p1.1}}^2 |B_t-x_1|^{-p}|B_t-x_2|^{-p} \exp\Big(-\int_0^{t} V^{\vec{\lambda},\vec{x}}(B_s)  ds\Big) 1(t\leq T_{r_\lambda})\nn \\
 \leq &c_{\ref{p1.1}}^2\sum_{i=1}^2  |B_t-x_1|^{-p}|B_t-x_2|^{-p} 1(|B_t-x_i|\leq |B_t-x_{3-i}|)\nn \\
 &\quad \quad \quad \quad \exp\Big(-\int_0^{t}V^{\lambda_i}(B_s-x_i) ds\Big) 1(t\leq T_{r_\lambda})\nn \\
  \leq &c_{\ref{p1.1}}^2 2^p\Delta^{-p}\sum_{i=1}^2  |B_t-x_i|^{-p} \exp\Big(-\int_0^{t} V^{\lambda_i}(B_s-x_i) ds\Big) 1(t\leq T_{r_{\lambda_i}}^i) .
  \end{align}
  where we have used $|B_t-x_{3-i}|>\Delta/2$ on $\{|B_t-x_i|\leq |B_t-x_{3-i}|\}$ and $T_{r_\lambda}\leq T_{r_{\lambda_i}}^i$ in the last inequality. It is clear that for $Leg\times P_x$-a.e. $(t, \omega)$ we have
   \begin{align}\label{2.1}
 &\lim_{\lambda_i \to \infty}    |B_t-x_i|^{-p} \exp\Big(-\int_0^{t} V^{\lambda_i}(B_s-x_i) ds\Big) 1(t\leq T_{r_{\lambda_i}}^i)\nn \\
 &= |B_t-x_i|^{-p} \exp\Big(-\int_0^{t} V^{\infty}(B_s-x_i) ds\Big).
  \end{align}
  In view of \eqref{2.4}, \eqref{e4.14.1} and \eqref{2.1}, if one can show that for $i=1,2$,
  \begin{align}\label{2.1.1}
 \lim_{\lambda_i \to \infty} & \int   \int_0^\infty   |B_t-x_i|^{-p} \exp\Big(-\int_0^{t} V^{\lambda_i}(B_s-x_i) ds\Big) 1(t\leq T_{r_{\lambda_i}}^i) dt dP_x\nn \\
 &=\int   \int_0^\infty   |B_t-x_i|^{-p} \exp\Big(-\int_0^{t} V^{\infty}(B_s-x_i) ds\Big) dt dP_x<\infty,
  \end{align}
then a generalized Dominated Convergence Theorem (see, e.g., Exercise 20 of Chp. 2 of \cite{Fol99}) implies that 
  \begin{align}\label{ce3.1.1}
\lim_{\lambda_1, \lambda_2 \to \infty}K_1&=\lim_{\lambda_1, \lambda_2 \to \infty} \int   \int_0^\infty  \lambda_1^{1+\alpha} \lambda_2^{1+\alpha} V_1^{\vec{\lambda},\vec{x}}(B_t)  V_2^{\vec{\lambda},\vec{x}}(B_t)\nn\\
&\quad \quad  \quad  \quad \quad  \exp\Big(-\int_0^{t} V^{\vec{\lambda},\vec{x}}(B_s) ds\Big)1(t\leq T_{r_\lambda})  dt dP_x \nn\\
&=K_{\ref{p3.1}}^2 \int   \int_0^\infty U_1^{\vec{\infty},\vec{x}}(B_t)  U_2^{\vec{\infty},\vec{x}}(B_t)\exp\Big(-\int_0^{t} V^{\vec{\infty},\vec{x}}(B_s) ds\Big) dt dP_x\nn\\
&=K_{\ref{p3.1}}^2(-U_{1,2}^{\vec{\infty},\vec{x}}(x)),
  \end{align}
 where the last is by  \eqref{eu12}. The proof will then be finished by Lemma \ref{l4.5}, \eqref{cm6.0}, \eqref{1.7} and \eqref{ce3.1.1}.
  
  It remains to prove \eqref{2.1.1} and it suffices to consider $i=1$.
We first show that for any $0<q<6-p$, we have
    \begin{align}\label{me6.3.4}
&\sup_{\lambda>0} \int   \int_0^\infty   |B_t-x_1|^{-q} \nn\\
&\quad \quad  \quad \quad \quad \quad \exp\Big(-\int_0^{t} V^{\lambda}(B_s-x_1) ds\Big)1(t\leq T_{r_{\lambda}}^1) dt dP_x<\infty.
  \end{align}
  Assuming the above, we can apply Fatou's Lemma to see that
\begin{align}\label{me3.4.2}
&\int   \int_0^\infty   |B_t-x_1|^{-q} \exp\Big(-\int_0^{t} V^{\infty}(B_s-x_1) ds\Big) dt dP_x\\
&\leq  \liminf_{\lambda\to \infty} \int   \int_0^\infty   |B_t-x_1|^{-q}  \nn\\
&\quad \quad  \quad \quad \quad \quad \exp\Big(-\int_0^{t} V^{\lambda}(B_s-x_1) ds\Big)1(t\leq T_{r_{\lambda}}^1) dt dP_x<\infty,\nn
\end{align}
thus giving the finiteness in \eqref{2.1.1} (recall $p\in (2,3)$). 

To see that \eqref{me6.3.4} holds, by Fubini's theorem and translation invariance we have
    \begin{align}\label{3.2}
I(\lambda):=& \int   \int_0^\infty   |B_t-x_1|^{-q} \exp\Big(-\int_0^{t} V^{\lambda}(B_s-x_1) ds\Big)1(t\leq T_{r_{\lambda}}^1) dt dP_x\nn\\
=& \int_0^\infty  E_{x-x_1}\Big( |B_t|^{-q} \exp\Big(-\int_0^{t} V^{\lambda}(B_s) ds\Big) 1(t\leq \tau_{r_{\lambda}}) \Big) dt,
  \end{align}
    where $\tau_r=\inf\{t\geq 0: |B_t|\leq r\}$ for any $r>0$.
Let $\mu, \nu$ are as in \eqref{ev1.5} and then apply Lemma \ref{l12.4} to get
    \begin{align}\label{3.6}
&E_{x-x_1}\Big( |B_t|^{-q} \exp\Big(-\int_0^{t} V^{\lambda}(B_s) ds\Big) 1(t\leq \tau_{r_{\lambda}}) \Big) \\
=&E_{|x-x_1|}^{(2+2\mu)}\Big( \rho_t^{-q} \exp\Big(-\int_0^{t} V^{\lambda}(\rho_s) ds\Big) 1(t\leq \tau_{r_{\lambda}})\Big)\nn\\
=&|x-x_1|^{\nu-\mu}E_{|x-x_1|}^{(2+2\nu)}\Big(\rho_t^{-q-\nu+\mu} \exp\Big(\int_0^{t} (V^\infty-V^{\lambda})(\rho_s) ds\Big)1_{(t\leq \tau_{r_{\lambda}})} \Big),\nn
  \end{align}
    where we slightly abuse the notation and let $\tau_r=\tau_r^\rho=\inf\{t\geq 0: \rho_t\leq r\}$ for any $r>0$.
 Use the above to see that \eqref{3.2} becomes
    \begin{align*}
&I(\lambda)=|x-x_1|^{\nu-\mu} \int_0^\infty E_{|x-x_1|}^{(2+2\nu)}\Big(\rho_t^{-q-\nu+\mu} \exp\Big(\int_0^{t} (V^\infty-V^{\lambda})(\rho_s) ds\Big)1_{(t\leq \tau_{r_{\lambda}})} \Big)dt\\
&=|x-x_1|^{\nu-\mu} E_{|x-x_1|}^{(2+2\nu)}\Big(\int_0^{\tau_{r_{\lambda}}} \rho_t^{-q-\nu+\mu} \exp\Big(\int_0^{t} (V^\infty-V^{\lambda})(\rho_s) ds\Big) dt\Big),
  \end{align*}
  where the second equality is by Fubini's theorem. Now use the scaling of Bessel process and $V^\infty, V^\lambda$ (recall $r_\lambda=\lambda_0 \lambda^{-\frac{1}{4-d}}$) to see that
    \begin{align}\label{me4.3.1}
I(\lambda)=|x-x_1|^{\nu-\mu} & E_{|x-x_1|/r_\lambda}^{(2+2\nu)}\Big(\int_0^{\tau_1} r_\lambda^{2-q-\nu+\mu} \rho_t^{-q-\nu+\mu} \nn\\
&\quad \quad \quad  \quad \quad   \exp\Big(\int_0^{t} (V^\infty-V^{\lambda_0^{4-d}})(\rho_s) ds\Big) dt\Big)\nn\\
\leq |x-x_1|^{\nu-\mu}  &E_{|x-x_1|/r_\lambda}^{(2+2\nu)}\Big(\int_0^{\tau_1} r_\lambda^{2-q-\nu+\mu} \rho_t^{-q-\nu+\mu}\nn \\
&\quad \quad  \quad  \exp\Big(\int_0^{t} c_{\ref{l12}}\lambda_0^{-(p-2)} \rho_s^{-p} ds\Big) dt\Big),
  \end{align}
 where the last inequality is by Lemma \ref{l12}.
 
   We interrupt the proof for another auxiliary result from \cite{MP17}.
\begin{lemma} \label{l3.1}
There is some universal constant $c_{\ref{l3.1}}>0$ such that for any $r>0$ with $r<|x|$ and $0<\delta<(p-2)(2-\mu)$ and $2+\mu-\nu<q<6-p$, we have
\begin{align*}
E_{|x|/r}^{(2+2\nu)} \Big(  \int_0^{\tau_1} {\rho_t}^{-q-\nu+\mu}& \exp\Big(\int_0^{t} \delta \rho_s^{-p}  ds\Big) dt\Big)\leq c_{\ref{l3.1}}  r^{-2+q+\nu-\mu} |x|^{2-q-\nu+\mu}.
\end{align*}
\end{lemma}
\begin{proof}
This is included in the proof of Proposition~6.1 of \cite{MP17} with  $r=r_{\lambda}$. In particular, the above expectation appears in (9.23) of \cite{MP17} and is bounded by $eJ_i$ in (9.27) of that paper.  Following the inequalities in that work, noting we only 
need to use Lemma 9.6(b) with $a=1$, $\gamma>1$ and $\gamma+p-2<1+\nu$ where $2\gamma=q+\nu-\mu$, we arrive at the above bound.
\end{proof}
Returning to \eqref{me4.3.1}, we choose $\lambda_0>0$ so
that $c_{\ref{l12}}\lambda_0^{-(p-2)}<(p-2)(2-\mu)$. If $\lambda$ is sufficiently large so that $r_\lambda<|x-x_1|$ we may apply
Lemma \ref{l3.1} to conclude
    \begin{align*}
I(\lambda) &\leq |x-x_1|^{\nu-\mu} r_\lambda^{2-q-\nu+\mu} c_{\ref{l3.1}} |x-x_1|^{2-q-\nu+\mu} r_\lambda^{-2+q+\nu-\mu}\\
&=c_{\ref{l3.1}} |x-x_1|^{2-q}<\infty,
  \end{align*}
 and we finish the proof of \eqref{me6.3.4}.

%
%

 Next we show that for any fixed $T>0$,
    \begin{align}\label{3.5}
 \lim_{\lambda \to \infty} & \int   \int_0^T   |B_t-x_1|^{-p} \exp\Big(-\int_0^{t} V^{\lambda}(B_s-x_1) ds\Big) 1(t\leq T_{r_{\lambda}}^1) dt dP_x\nn\\
 =&\int   \int_0^T   |B_t-x_1|^{-p} \exp\Big(-\int_0^{t} V^{\infty}(B_s-x_1) ds\Big) dt dP_x.
  \end{align}
Since we are working under a finite measure $1(t\leq T) dt dP_x$, it suffices to show that $\{ |B_t-x_1|^{-p} \exp\Big(-\int_0^{t} V^{\lambda}(B_s-x_1) ds\Big) 1(t\leq T_{r_{\lambda}}^1)\}$ is a uniformly integrable family indexed by $\lambda$ sufficiently large. This in turn will follow from a $(1+\gamma)$ moment bound for $\gamma>0$ which is uniform in $\lambda$ sufficiently large. Since $p\in (2,3)$, we can pick $\gamma>0$ small such that $q:=p(1+\gamma)<6-p$. Therefore by \eqref{me6.3.4} we have
  \begin{align*}
&\int  \int_0^T   |B_t-x_1|^{-p(1+\gamma)} \exp\Big(-(1+\gamma)\int_0^{t} V^{\lambda}(B_s-x_1) ds\Big) 1(t\leq T_{r_{\lambda}}^1) dt dP_x\\
&\leq \int   \int_0^\infty   |B_t-x_1|^{-q} \exp\Big(-\int_0^{t} V^{\lambda}(B_s-x_1) ds\Big) 1(t\leq T_{r_{\lambda}}^1) dt dP_x<\infty
  \end{align*}
  and \eqref{3.5} follows as noted above. 
  
Use \eqref{me3.4.2} with $q=p$ to get 
  \begin{align}\label{m6.4}
\lim_{T \to \infty} \int   \int_T^\infty   |B_t-x_1|^{-p} \exp\Big(-\int_0^{t} V^{\infty}(B_s-x_1) ds\Big) dt dP_x=0.
  \end{align}
  We claim that
   \begin{align}\label{ae3.1}
 &\lim_{T\to \infty} \sup_{\lambda>0}\int   \int_T^\infty   |B_t-x_1|^{-p} \nn\\
 &\quad \quad \quad \exp\Big(-\int_0^{t} V^{\lambda}(B_s-x_1) ds\Big) 1_{(t\leq T_{r_{\lambda}}^1)} dt dP_x=0.
  \end{align}
Then the proof of \eqref{2.1.1} will follow immediately from \eqref{3.5}, \eqref{m6.4} and \eqref{ae3.1}.

It remains to prove \eqref{ae3.1}.
Similar to the derivation of \eqref{3.2} and \eqref{3.6} with $q=p$, we have
\begin{align}\label{me4.4.1}
\int   &\int_T^\infty  |B_t-x_1|^{-p} \exp\Big(-\int_0^{t} V^{\lambda}(B_s-x_1) ds\Big) 1_{(t\leq T_{r_{\lambda}}^1)} dt dP_x\\
 = |x-x_1|^{\nu-\mu} &\int_T^\infty  E_{|x-x_1|}^{(2+2\nu)}\Big(\rho_t^{-p-\nu+\mu} \exp\Big(\int_0^{t} (V^\infty-V^{\lambda})(\rho_s) ds\Big)1_{(t\leq \tau_{r_{\lambda}})} \Big)dt.\nn
\end{align}
Use $p=\mu+\nu$ to see that the integrand of the right-hand side term of the above equals
\begin{align}\label{4.2}
&E_{|x-x_1|}^{(2+2\nu)}\Big( \rho_{t}^{-2\nu} \exp\Big(\int_0^{t} (V^{\infty}-V^{\lambda})(\rho_s) ds\Big) 1(t\leq \tau_{r_{\lambda}})\Big)\nn\\
&\leq\Big(E_{|x-x_1|}^{(2+2\nu)}( \rho_{t}^{-2\nu})\Big)^{1/2}  \nn\\
& \times \Big(E_{|x-x_1|}^{(2+2\nu)}\Big( \rho_{t}^{-2\nu} \exp\Big(\int_0^{t} 2(V^{\infty}-V^{\lambda})(\rho_s) ds\Big) 1(t\leq \tau_{r_{\lambda}})\Big)\Big)^{1/2},
\end{align}
where in the inequality we have applied Cauchy-Schwarz inequality.

For the first term on the right-hand side of \eqref{4.2}, we use the scaling of Bessel process to get
\begin{align}\label{4.3}
E_{|x-x_1|}^{(2+2\nu)}( \rho_{t}^{-2\nu})=t^{-\nu} E_{|x-x_1|}^{(2+2\nu)}( \rho_{1}^{-2\nu}):= t^{-\nu} C(\nu, |x-x_1|),
\end{align}
where the finiteness of $E_{|x-x_1|}^{(2+2\nu)}( \rho_{1}^{-2\nu})$ follows easily from the known transition density of Bessel process (see, e.g. Chp. XI of \cite{RY94}). For the second term on the right-hand side of \eqref{4.2}, by (2.c) of \cite{Yor92} one can conclude that for any $r>\eps>0$,
 \begin{align}
  {P}_{r}^{(2-2\nu)}\Big|_{\cF_{\tau_\eps\wedge t}^\rho}=\frac{r^{2\nu}}{\rho_{{\tau_\eps\wedge t}}^{2\nu}}{P}^{(2+2\nu)}_r\Big|_{\cF_{\tau_\eps\wedge t}^\rho}.
   \end{align}
Use the above to get
\begin{align}\label{4.4}
&E_{|x-x_1|}^{(2+2\nu)}\Big( \rho_{t}^{-2\nu} \exp\Big(\int_0^{t} 2(V^{\infty}-V^{\lambda})(\rho_s) ds\Big) 1(t\leq \tau_{r_{\lambda}})\Big)\\
&=E_{|x-x_1|}^{(2+2\nu)}\Big( \rho_{t\wedge \tau_{r_{\lambda}} }^{-2\nu} \exp\Big(\int_0^{t\wedge \tau_{r_{\lambda}}} 2(V^{\infty}-V^{\lambda})(\rho_s) ds\Big) 1(t\leq t\wedge \tau_{r_{\lambda}})\Big)\nn\\
&=|x-x_1|^{-2\nu} E_{|x-x_1|}^{(2-2\nu)}\Big(\exp\Big(\int_0^{t\wedge \tau_{r_{\lambda}}} 2(V^{\infty}-V^{\lambda})(\rho_s) ds\Big) 1(t\leq t\wedge \tau_{r_{\lambda}})\Big)\nn\\
&\leq |x-x_1|^{-2\nu} E_{|x-x_1|}^{(2-2\nu)}\Big(\exp\Big(\int_0^{ \tau_{r_{\lambda}}} 2(V^{\infty}-V^{\lambda})(\rho_s) ds\Big) \Big)\nn\\
&= |x-x_1|^{-2\nu} E_{|x-x_1|}^{(2+2\nu)}\Big(\exp\Big(\int_0^{ \tau_{r_{\lambda}}} 2(V^{\infty}-V^{\lambda})(\rho_s) ds\Big)\Big|\tau_{r_\lambda}<\infty \Big)\nn\\
&\leq  |x-x_1|^{-2\nu}  C_{\ref{c13.5}}(\lambda_0, \nu,2), \nn
\end{align}
where the last equality is by Corollary \ref{c1.4} and in the last inequality we have used Lemma \ref{c13.5} with $\lambda_0>c_{\ref{c13.5}}$ and $\gamma=2$. 
Now we conclude from \eqref{4.2}, \eqref{4.3} and \eqref{4.4} that
\begin{align*}
&E_{|x-x_1|}^{(2+2\nu)}\Big( \rho_{t}^{-2\nu} \exp\Big(\int_0^{t} (V^{\infty}-V^{\lambda})(\rho_s) ds\Big) 1(t\leq \tau_{r_{\lambda}})\Big)\nn\\
&\leq  C(\nu, |x-x_1|)^{1/2} t^{-\nu/2}C_{\ref{c13.5}}(\lambda_0, \nu,2)^{1/2} |x-x_1|^{-\nu}=C(|x-x_1|) t^{-\nu/2}.
\end{align*}
Returning to \eqref{me4.4.1},  we apply the above to get
\begin{align}\label{m6.5}
&\int   \int_T^\infty   |B_t-x_1|^{-p} \exp\Big(-\int_0^{t} V^{\lambda_1}(B_s-x_1) ds\Big) 1(t\leq T_{r_{\lambda_1}}^1) dt dP_x\nn\\
&\leq |x-x_1|^{\nu-\mu} \int_T^\infty C(|x-x_1|) t^{-\nu/2} dt \nn\\
&\leq C(|x-x_1|)\cdot (\nu/2-1)^{-1} T^{1-\nu/2} \to 0 \text{ as } T\to \infty,
  \end{align}
  thus giving \eqref{ae3.1}. The proof is then complete.
  \end{proof}
  
\begin{proof}[Proof of Proposition \ref{p3.3}(ii)]
For any  $x_1\neq x_2$, we fix $x\neq x_1, x_2$. In order the find the limit of $\lambda_1^{1+\alpha}\eps^{-(p-2)} (-W_{1,2}^{\vec{\lambda},\vec{x},\eps}(x))$ as $\lambda_1 \to \infty, \eps \downarrow 0$, by Lemma \ref{l4.6}, it suffices to find the limits of the following as $\lambda_1 \to \infty, \eps \downarrow 0$.
 \begin{align}\label{m6.0}
 K_1+K_2& \equiv \frac{\lambda_1^{1+\alpha} }{\eps^{p-2}} E_x\Big( \int_0^{T_{\lambda_1,\eps}} \prod_{i=1}^2 W_i^{\vec{\lambda},\vec{x},\eps}(B_t)\exp\Big(-\int_0^{t} W^{\vec{\lambda},\vec{x},\eps}(B_s) ds\Big)dt\Big)\nn\\
 +&\frac{\lambda_1^{1+\alpha} }{\eps^{p-2}}  E_x\Big( \exp\Big(-\int_0^{T_{\lambda_1,\eps}} W^{\vec{\lambda},\vec{x},\eps}(B_s) ds\Big)1_{(T_{\lambda_1,\eps}<\infty)} (-W_{1,2}^{\vec{\lambda},\vec{x},\eps}(B( T_{\lambda_1,\eps})))\Big).
\end{align}
 In the above $T_{\lambda_1,\eps}=T_{r_{\lambda_1}}^1 \wedge T_{2\eps}^2$ where $T_{r_{\lambda_1}}^1=\inf\{t\geq 0: |B_t-x_1|\leq r_{\lambda_1} \}$ and $T_{2\eps}^2=\inf\{t\geq 0: |B_t-x_2|\leq 2\eps \}$. Here $r_{\lambda_1}=\lambda_0 \lambda_1^{-\frac{1}{4-d}}$ and we will choose $\lambda_0$ to be some fixed large constant below. Let $\eps>0$ small and $\lambda_1>0$ large so that
\begin{align}\label{ae3.3}
0<4(r_{\lambda_1}\vee \eps)<\min\{|x_1-x|, |x_2-x|,  |x_1-x_2|\}.
\end{align}
We first consider $K_2$.  On $\{T_{\lambda_1,\eps}<\infty\}$, by considering $T_{\lambda_1,\eps}=T_{r_{\lambda_1}}^1<T_{2\eps}^{2}$ we may set $x_\lambda(\omega)=B(T_{\lambda_1,\eps})=B(T_{r_{\lambda_1}}^1)$ so that  $|x_\lambda-x_1|=r_{\lambda_1}$ and hence $|x_{2}-x_\lambda|\geq \Delta/2$ where $\Delta=|x_1-x_2|$.  Lemma \ref{l4.2} and the above imply
\[
(-W_{1,2}^{\vec{\lambda},\vec{x},\eps}(B(T_{\lambda_1,\eps})))\leq 2\lambda_{1}^{-1}|B(T_{\lambda_1,\eps})-x_2|^{-p} \eps^{p-2}\leq 2^{p+1}\Delta^{-p} \lambda_{1}^{-1}\eps^{p-2}.\]
Similarly by considering $T_{\lambda_1,\eps}=T_{2\eps}^{2}<T_{r_{\lambda_1}}^1$  we have $|B(T_{2\eps}^{2})-x_1|\geq \Delta/2$ and hence by Lemma \ref{l4.2},
\begin{align*}
(-W_{1,2}^{\vec{\lambda},\vec{x},\eps}(B(T_{\lambda_1,\eps})))&\leq 2\lambda_{2}^{-1}c_{\ref{p1.1}} \lambda_{1}^{-(1+\alpha)} |B(T_{\lambda_1,\eps})-x_1|^{-p}\\
&\leq 2^{p+1}\Delta^{-p} c_{\ref{p1.1}} \lambda_{2}^{-1} \lambda_{1}^{-(1+\alpha)}.
\end{align*}
 This shows that 
 \begin{align}\label{a1.4}
 K_2\leq  &  \frac{\lambda_1^{1+\alpha} }{\eps^{p-2}}2^{p+1}\Delta^{-p}\lambda_{1}^{-1}\eps^{p-2}\nn\\
 &\quad \quad E_x\Big( 1_{(T_{r_{\lambda_1}}^1<\infty)}1_{(T_{r_{\lambda_1}}^1<T_{2\eps}^{2})}\exp\Big(-\int_0^{ T_{r_{\lambda_1}}^1} W^{\vec{\lambda},\vec{x},\eps}(B_s)ds\Big) \Big)\nn\\
+& \frac{\lambda_1^{1+\alpha} }{\eps^{p-2}} 2^{p+1}\Delta^{-p}c_{\ref{p1.1}} \lambda_{2}^{-1}\lambda_{1}^{-(1+\alpha)}\nn\\
 &\quad \quad E_x\Big( 1_{(T_{2\eps}^{2}<\infty)}1_{(T_{2\eps}^{2}<T_{r_{\lambda_1}}^1)}\exp\Big(-\int_0^{T_{2\eps}^{2}} W^{\vec{\lambda},\vec{x},\eps}(B_s)ds\Big) \Big).
 \end{align}
By \eqref{ae7.4},  for all $x \text{ so that } x\neq x_1 \text{ and } |x-x_2|>\eps$ we have
 \begin{align}\label{a1.5}
 W^{\vec{\lambda},\vec{x},\eps}(x) \geq V^{\lambda_1}(x-x_1) \vee V^{\infty}(x-x_2).
 \end{align}
 Let $\tau_{r}=\inf\{t: |B_t|\leq r\}$. Use the above to see that
  \begin{align}\label{a1.6}
 &E_x\Big( 1_{(T_{r_{\lambda_1}}^1<\infty)}1_{(T_{r_{\lambda_1}}^1<T_{2\eps}^{2})}\exp\Big(-\int_0^{ T_{r_{\lambda_1}}^1} W^{\vec{\lambda},\vec{x},\eps}(B_s)ds\Big) \Big)\nn\\
\leq&E_{x-x_1}\Big( 1(\tau_{r_{\lambda_1}}<\infty) \exp\Big(-\int_0^{ \tau_{r_{\lambda_1}}} V^{\lambda_1}(B_s)  ds\Big) \Big)\nn \\
\leq & r_{\lambda_1}^p |x-x_1|^{-p} C_{\ref{c13.5}}(\lambda_0,\nu,1)
  \end{align}
where the last line follows in a similar way to the derivation of \eqref{m1.2} by choosing $\lambda_0>c_{\ref{c13.5}}$. 

Similarly by \eqref{a1.5} and \eqref{bm1.2} we have
  \begin{align}\label{a1.6a}
 &E_x\Big( 1_{(T_{2\eps}^{2}<\infty)}1_{(T_{2\eps}^{2}<T_{r_{\lambda_1}}^1)}\exp\Big(-\int_0^{T_{2\eps}^{2}} W^{\vec{\lambda},\vec{x},\eps}(B_s)ds\Big) \Big)\nn\\
  \leq&E_{x-x_2}\Big( 1(\tau_{2\eps}<\infty) \exp\Big(-\int_0^{ \tau_{2\eps}} V^{\infty}(B_s)  ds\Big) \Big)=(2\eps/|x-x_1|)^{p}.
  \end{align}
Apply \eqref{a1.6} and \eqref{a1.6a} in \eqref{a1.4} to get
 \begin{align}\label{a1.7}
  K_2\leq  &  \frac{\lambda_1^{1+\alpha} }{\eps^{p-2}}2^{p+1}\Delta^{-p}\lambda_{1}^{-1}\eps^{p-2} r_{\lambda_1}^p |x-x_1|^{-p} C_{\ref{c13.5}}(\lambda_0,\nu,1)\nn\\
&+ \frac{\lambda_1^{1+\alpha} }{\eps^{p-2}} 2^{p+1}\Delta^{-p}c_{\ref{p1.1}} \lambda_{2}^{-1}\lambda_{1}^{-(1+\alpha)}(2\eps/|x-x_1|)^{p}\nn\\
\leq&C\Delta^{-p} \lambda_0^p |x-x_1|^{-p}\lambda_{1}^{-\frac{2}{4-d}} +C\lambda_{2}^{-1}\Delta^{-p}|x-x_2|^{-p} \eps^2  \to 0
 \end{align}
as $\lambda_1 \to \infty, \eps \downarrow 0$, where in the last equality we have used the definitions of $r_{\lambda_1}$ and $\alpha$.

Now we will turn to $K_1$. Recall 
 \begin{align*}
K_1&=\int   \int_0^\infty  \frac{ \lambda_1^{1+\alpha}}{\eps^{p-2}} W_1^{\vec{\lambda},\vec{x},\eps}(B_t)  W_2^{\vec{\lambda},\vec{x},\eps}(B_t)\\
&\quad \quad \quad \exp\Big(-\int_0^{t} W^{\vec{\lambda},\vec{x},\eps}(B_s) ds\Big)1(t\leq T_{\lambda_1,\eps})  dt dP_x.
 \end{align*}
By Proposition \ref{p3.3} and Lemma \ref{l7.5}, for $Leb\times P_x$-a.e. $(t,\omega)$, we have
 \begin{align}\label{a2.4}
\lim_{\lambda_1 \to \infty, \eps \downarrow 0} &\frac{\lambda_1^{1+\alpha}}{\eps^{p-2}} W_1^{\vec{\lambda},\vec{x},\eps}(B_t)  W_2^{\vec{\lambda},\vec{x},\eps}(B_t)\exp\Big(-\int_0^{t} W^{\vec{\lambda},\vec{x},\eps}(B_s) ds\Big)1(t\leq T_{\lambda_1,\eps})\nn\\
&=K_{\ref{p3.1}}C_{\ref{p3.1}}(\lambda_2) U_1^{\vec{\infty},\vec{x}}(B_t)U_2^{\vec{\infty},\vec{x}}(B_t)\exp\Big(-\int_0^{t} V^{\vec{\infty},\vec{x}}(B_s) ds\Big).
 \end{align}
Use the bounds \eqref{m8.7}, \eqref{m8.8} and \eqref{a1.5} to see that
\begin{align}\label{ae4.14.1}
& \frac{\lambda_1^{1+\alpha} }{\eps^{p-2}}  W_1^{\vec{\lambda},\vec{x},\eps}(B_t)  W_2^{\vec{\lambda},\vec{x},\eps}(B_t)\exp\Big(-\int_0^{t} W^{\vec{\lambda},\vec{x},\eps}(B_s) ds\Big)1(t\leq T_{\lambda_1,\eps})  \nonumber\\
 \leq & c_{\ref{p1.1}} |B_t-x_1|^{-p}|B_t-x_2|^{-p} \exp\Big(-\int_0^{t} W^{\vec{\lambda},\vec{x},\eps}(B_s)  ds\Big) 1(t\leq T_{\lambda_1,\eps})\nn \\
 = &c_{\ref{p1.1}}\sum_{i=1}^2  |B_t-x_1|^{-p}|B_t-x_2|^{-p} 1(|B_t-x_i|\leq |B_t-x_{3-i}|)\nn \\
 &\quad \quad \quad \quad \exp\Big(-\int_0^{t} W^{\vec{\lambda},\vec{x},\eps}(B_s) ds\Big) 1(t\leq T_{\lambda_1,\eps})\nn \\
  \leq &c_{\ref{p1.1}} 2^p\Delta^{-p}  |B_t-x_1|^{-p} \exp\Big(-\int_0^{t} V^{\lambda_1}(B_s-x_1) ds\Big) 1(t\leq T_{r_{\lambda_1}}^1)\nn\\
&\quad \quad +c_{\ref{p1.1}} 2^p\Delta^{-p}  |B_t-x_2|^{-p} \exp\Big(-\int_0^{t} V^{\infty}(B_s-x_2) ds\Big) 
  \end{align}
  where we have used $|B_t-x_{3-i}|>\Delta/2$ on $\{|B_t-x_i|\leq |B_t-x_{3-i}|\}$ and $T_{\lambda_1,\eps}\leq T_{r_{\lambda_1}}^1$ in the last inequality. It is clear that for $Leb\times P_x$-a.e. $(t, \omega)$,
\begin{align}\label{a2.1}
 &\lim_{\lambda_1 \to \infty}    |B_t-x_1|^{-p} \exp\Big(-\int_0^{t} V^{\lambda_1}(B_s-x_1) ds\Big) 1(t\leq T_{r_{\lambda_1}}^1)\nn \\
 &= |B_t-x_1|^{-p} \exp\Big(-\int_0^{t} V^{\infty}(B_s-x_1) ds\Big).
\end{align}
By \eqref{2.1.1} we have
 \begin{align}\label{a2.1.1}
 \lim_{\lambda_1 \to \infty} & \int   \int_0^\infty   |B_t-x_1|^{-p} \exp\Big(-\int_0^{t} V^{\lambda_1}(B_s-x_1) ds\Big) 1(t\leq T_{r_{\lambda_1}}^1) dt dP_x \nn\\
 &=\int  \int_0^\infty   |B_t-x_1|^{-p} \exp\Big(-\int_0^{t} V^{\infty}(B_s-x_1) ds\Big) dt dP_x<\infty,
  \end{align}
and  by \eqref{me3.4.2} with $q=p$ we have 
 \begin{align}\label{a2.1.2}
\int  \int_0^\infty   |B_t-x_2|^{-p} \exp\Big(-\int_0^{t} V^{\infty}(B_s-x_2) ds\Big) dt dP_x<\infty.
  \end{align}
 In view of \eqref{a2.4}, \eqref{ae4.14.1}, \eqref{a2.1}, \eqref{a2.1.1} and \eqref{a2.1.2},  a generalized Dominated Convergence Theorem (see, e.g., Exercise 20 of Chp. 2 of \cite{Fol99}) implies that 
  \begin{align*}
&\lim_{\lambda_1 \to \infty, \eps \downarrow 0} K_1=\lim_{\lambda_1 \to \infty, \eps \downarrow 0} \int   \int_0^\infty   \frac{\lambda_1^{1+\alpha}}{\eps^{p-2}} W_1^{\vec{\lambda},\vec{x},\eps}(B_t)  W_2^{\vec{\lambda},\vec{x},\eps}(B_t)\nn\\
&\quad \quad \quad \quad \exp\Big(-\int_0^{t} W^{\vec{\lambda},\vec{x},\eps}(B_s) ds\Big)1_{\{t\leq T_{\lambda_1,\eps}\}} dt dP_x\nn \\
=&K_{\ref{p3.1}} C_{\ref{p3.1}}(\lambda_2) \int   \int_0^\infty U_1^{\vec{\infty},\vec{x}}(B_t)  U_2^{\vec{\infty},\vec{x}}(B_t) \exp\Big(-\int_0^{t} V^{\vec{\infty},\vec{x}}(B_s) ds\Big) dt dP_x\\
=&K_{\ref{p3.1}} C_{\ref{p3.1}}(\lambda_2) (-U_{1,2}^{\vec{\infty},\vec{x}}(x)),
  \end{align*}
   where the last is by  \eqref{eu12}. The proof will then be finished by Lemma \ref{l4.6}, \eqref{m6.0}, \eqref{a1.7} and the above.
  \end{proof}

\end{document}